\documentclass[a4paper,11pt]{amsart}

\usepackage[utf8]{inputenc}
\usepackage[T1]{fontenc}
\usepackage{amscd,amssymb,amsmath,latexsym,graphicx,color,url,lscape,tikz,mathtools,enumitem}
\usepackage[all]{xy}
\usepackage{chngcntr}
\counterwithin{figure}{section}
\usepackage{kbordermatrix}
\renewcommand{\kbldelim}{(}
\renewcommand{\kbrdelim}{)}
\usepackage[textwidth=27mm, textsize=small]{todonotes}

\usepackage{hyperref}
\hypersetup{breaklinks=true}

\def\ov#1{{\overline{#1}}}

\def\wh#1{{\widehat{#1}}}
\def\wt#1{{\widetilde{#1}}}

\newcommand{\bigboxplus}{
  \mathop{
    \vphantom{\bigoplus}
    \mathchoice
      {\vcenter{\hbox{\resizebox{\widthof{$\displaystyle\bigoplus$}}{!}{$\boxplus$}}}}
      {\vcenter{\hbox{\resizebox{\widthof{$\bigoplus$}}{!}{$\boxplus$}}}}
      {\vcenter{\hbox{\resizebox{\widthof{$\scriptstyle\oplus$}}{!}{$\boxplus$}}}}
      {\vcenter{\hbox{\resizebox{\widthof{$\scriptscriptstyle\oplus$}}{!}{$\boxplus$}}}}
  }\displaylimits
}

\newcommand{\MV}{\operatorname{MV}}
\newcommand{\conv}{\operatorname{conv}}

\newcommand{\vol}{\operatorname{vol}}

\newcommand{\Hom}{\operatorname{Hom}}
\newcommand{\supp}{\operatorname{supp}}

\newcommand{\Res}{\operatorname{Res}}
\newcommand{\Elim}{\operatorname{Elim}}
\newcommand{\rank}{\operatorname{rank}}
\newcommand{\cc}{{\rm c}}

\newcommand{\epi}{\operatorname{epi}}

\newcommand{\gr}{\operatorname{gr}}
\newcommand{\diag}{\operatorname{diag}}

\newcommand{\init}{\operatorname{init}}
\newcommand{\rc}{\operatorname{rc}}
\renewcommand{\and}{\quad \text{and} \quad}

\newcommand{\ord}{\operatorname{ord}}

\newcommand{\ri}{\operatorname{ri}}

\newcommand{\transposed}{{\rm T}}
\newcommand{\MI}{\operatorname{MI}}
\newcommand{\ev}{\operatorname{eval}}

\newcommand{\dd}{{\rm d\hspace*{-0.5mm}}}
\newcommand{\sat}{{\rm sat}}
\newcommand{\red}{{\rm red}}

\newcommand{\C}{\mathbb{C}}

\newcommand{\K}{\mathbb{K}}

\newcommand{\N}{\mathbb{N}}
\renewcommand{\P}{\mathbb{P}}
\newcommand{\Q}{\mathbb{Q}}
\newcommand{\R}{\mathbb{R}}

\newcommand{\T}{\mathbb{T}}

\newcommand{\Z}{\mathbb{Z}}

\newcommand{\cA}{{\mathcal A}}
\newcommand{\cB}{{\mathcal B}}
\newcommand{\cC}{{\mathcal C}}
\newcommand{\cD}{{\mathcal D}}
\newcommand{\cE}{{\mathcal E}}

\newcommand{\cG}{{\mathcal G}}
\newcommand{\cH}{{\mathcal H}}
\newcommand{\cI}{{\mathcal I}}

\newcommand{\cM}{{\mathcal M}}
\newcommand{\cN}{{\mathcal N}}

\newcommand{\cP}{{\mathcal P}}
\newcommand{\cQ}{{\mathcal Q}}

\newcommand{\cS}{{\mathcal S}}

\newcommand{\bfa}{{\boldsymbol{a}}}
\newcommand{\bfb}{{\boldsymbol{b}}}
\newcommand{\bfc}{{\boldsymbol{c}}}
\newcommand{\bfd}{{\boldsymbol{d}}}

\newcommand{\bff}{{\boldsymbol{f}}}

\newcommand{\bfs}{{\boldsymbol{s}}}
\newcommand{\bft}{{\boldsymbol{t}}}
\newcommand{\bfu}{{\boldsymbol{u}}}
\newcommand{\bfv}{{\boldsymbol{v}}}

\newcommand{\bfx}{{\boldsymbol{x}}}

\newcommand{\bfz}{{\boldsymbol{z}}}
\newcommand{\bfp}{{\boldsymbol{p}}}

\newcommand{\bfF}{{\boldsymbol{F}}}
\newcommand{\bfG}{{\boldsymbol{G}}}

\newcommand{\bfcA}{{\boldsymbol{\cA}}}

\newcommand{\bfcD}{{\boldsymbol{\cD}}}

\newcommand{\bfalpha}{{\boldsymbol{\alpha}}}
\newcommand{\bfbeta}{{\boldsymbol{\beta}}}
\newcommand{\bfgamma}{{\boldsymbol{\gamma}}}
\newcommand{\bfomega}{{\boldsymbol{\omega}}}
\newcommand{\bfdelta}{{\boldsymbol{\delta}}}

\newcommand{\bfzeta}{{\boldsymbol{\zeta}}}

\newcommand{\bfnu}{{\boldsymbol{\nu}}}
\newcommand{\bfrho}{{\boldsymbol{\rho}}}

\newcommand{\bfDelta}{{\boldsymbol{\Delta}}}

\newcommand{\bfzero}{{\boldsymbol{0}}}

\newcounter{thm}
\numberwithin{thm}{section}
\numberwithin{equation}{section}

\theoremstyle{definition}
\newtheorem{definition}[thm]{Definition}

\newtheorem{remark}[thm]{Remark}
\newtheorem{example}[thm]{Example}

\theoremstyle{plain}
\newtheorem{lemma}[thm]{Lemma}
\newtheorem{proposition}[thm]{Proposition}
\newtheorem{theorem}[thm]{Theorem}
\newtheorem{corollary}[thm]{Corollary}

\newtheorem{prop-def}[equation]{Proposition-Definition}

\begin{document}

\title[The Canny-Emiris conjecture for the sparse resultant]{The
  Canny-Emiris conjecture for the sparse resultant}

\author[D'Andrea]{Carlos D'Andrea}
\address{Departament de Matem\`atiques i Inform\`atica, Universitat de Barcelona.
Gran Via~585, 08007 Barcelona, Spain}
\email{cdandrea@ub.edu}

\author[Jeronimo]{Gabriela Jeronimo}
\address{Departamento de Matem\'atica and
  IMAS (CONICET-UBA), Facultad de Ciencias Exactas y Naturales,
  Universidad de Buenos Aires. Ciudad Universitaria, Pabell\'on I,
  1428 Buenos Aires, Argentina }
\email{jeronimo@dm.uba.ar}

\author[Sombra]{Mart{\'\i}n~Sombra}
\address{Instituci\'o Catalana de Recerca
  i Estudis Avan\c{c}ats (ICREA). Passeig Llu{\'\i}s Companys~23,
  08010 Barcelona, Spain  \vspace*{-2.5mm}}
\address{Departament de Matem\`atiques i
  Inform\`atica, Universitat de Barcelona. Gran Via 585, 08007
  Bar\-ce\-lo\-na, Spain}
\email{sombra@ub.edu}

\date{\today}


\begin{abstract}
  We present a product formula for the initial parts of the sparse
  resultant associated to an arbitrary family of supports,
  generalizing a previous result by Sturmfels. This allows to compute
  the homogeneities and  degrees of this sparse resultant, and its
  evaluation at systems of Laurent polynomials with smaller supports.
  We obtain an analogous product formula for some of the initial parts of
  the principal minors of the Sylvester-type square matrix associated
  to a mixed subdivision of a polytope.

  Applying these results, we prove that under suitable hypothesis, the
  sparse resultant can be computed as the quotient of the determinant
  of such a square matrix by one of its principal minors. This
  generalizes the classical Macaulay formula for the homogeneous
  resultant, and confirms a conjecture of Canny and~Emiris.
\end{abstract}

\maketitle

\vspace{-7mm}

\setcounter{tocdepth}{1}
\tableofcontents

\vspace{-8mm}

\section{Introduction}

In \cite{Macaulay:sfe}, Macaulay introduced the notion of homogeneous
resultant, extending the Sylvester resultant to systems of homogeneous
polynomials in several variables with given degrees.  In the same
paper, he also presented an intriguing family of formulae, each of
them allowing to compute it as the quotient of the determinant of a
Sylvester-type square matrix by one of its principal minors.

The sparse resultant is a generalization of the homogeneous resultant
to systems of multivariate Laurent polynomials with prefixed
monomials.  It is a basic tool of elimination theory and polynomial
equation solving, and it is also connected to combinatorics, toric
geometry, and hypergeometric functions, see for instance \cite{GKZ94,
  Stu94, Est10:npdp, DS15}. As a consequence, there has been a lot of
interest in efficient methods for computing it, see also \cite{EM99,
  CanEmi:sbasr, DAndrea:msfsr, CLO05,DicEmi2005:spe, JS18} and the
references therein.

In \cite{CannyEmiris:easmr, CanEmi:sbasr}, Canny and Emiris introduced
a family of Sylvester-type square matrices whose determinants are
nonzero multiples of the sparse resultant, and showed that the sparse
resultant can be expressed as the gcd of several of these
determinants.  Besides, for each of these matrices they identified a
certain principal submatrix and, following Macaulay, conjectured that
the quotient of their determinants coincides with the sparse
resultant, at least in some cases.  Their construction relies heavily
on the combinatorics of the polytopes defined as the convex hull of
the exponents of the given monomials, and of a chosen family of affine
functions on them.  Shortly afterwards, Sturmfels extended the method
by allowing the use of convex piecewise affine functions on these
polytopes \cite{Stu94}.

Using this circle of ideas, the first author found a recursive
procedure to build Sylvester-type square matrices with a distinguished
principal submatrix, and obtained another family of formulae for the
sparse resultant extending those of Macaulay for the homogeneous
resultant \cite{DAndrea:msfsr}.  Some connections between the D'Andrea
construction and that of Canny and Emiris were explored by Emiris and
Konaxis for families of monomials whose associated polytopes are
scaled copies of a fixed one \cite{EmirisKonaxis:slmtfgusr}.  There
are also some determinantal formulae for sparse resultants, but their
applicability is limited to a short list of special cases
\cite{SZ94,WZ94, DE03,Khe02,Khe05,BFMT18,Gro20,BMT20,EMT21}.

The main result of this paper is a proof of a generalized version of
the Canny-Emiris conjecture, with precise conditions for its validity.
Our approach is based on a systematic study of the Canny-Emiris
matrices and their interplay with mixed subdivisions of polytopes. In
particular, we compute the orders and initial parts of its principal
minors and establish the compatibility of this construction with the
restriction of the defining data. We also prove a product formula for
the initial parts of the sparse resultant, generalizing a previous one
by Sturmfels \cite{Stu94}.

Classically, sparse resultants and Canny-Emiris matrices were studied
in the situation where the family of exponents of the given monomials
is essential in the sense of Sturmfels, that is, when the sparse
resultant does depend on all the sets of variables and, in addition,
the affine span of these exponents coincides with the ambient lattice,
see \cite[\S1]{Stu94} or Remark \ref{rem:1} for details. Whereas this
is, without any doubt, the main case of interest, a crucial part of
our analysis consists in extending and studying these notions in full
generality. Having constructions and properties that behave uniformly
allows us to descend to the simple cases where the result can be
proved directly.

We also show that the Macaulay formula for the homogeneous resultant
corresponding to the critical degree appears as a particular case of
our result, thus obtaining an independent proof for it.

\medskip We next explain these results with more detail.  Let
$M\simeq \Z^{n}$ be a lattice of rank $n$. Set
$\T_{M}= \Hom(M,\C^{\times }) \simeq (\C^{\times})^{n} $ for the
associated torus and, for $a\in M$, denote by
$ \chi^{a} \colon \T_{M}\to \C^{\times}$ the corresponding character.
For $i=0,\dots, n$ let $\cA_{i} \subset M$ be a nonempty finite
subset, $\bfu_{i}=\{u_{i,a}\}_{a\in \cA_{i}}$ a set of $\#\cA_{i}$
variables and
 $$
  F_{i}=\sum_{a\in \cA_{i}}u_{i,a} \, \chi^{a}\in \Z[\bfu_{i}][M]
$$
the general Laurent polynomial with support $\cA_{i}$, where
$\Z[\bfu_{i}][M] \simeq \Z[\bfu_{i}][x_1^{\pm 1},\dots, x_n^{\pm 1}]$
denotes the group $\Z[\bfu_{i}]$-algebra of $M$.

Let
$\Res_{\bfcA}, \Elim_{\bfcA} \in \Z[\bfu] = \Z[\bfu_{0},\dots,
\bfu_{n}] $ be the sparse resultant and the sparse eliminant
associated to the family of supports $\bfcA=(\cA_{0},\dots, \cA_{n})$
in the sense of \cite{Est10:npdp, DS15}. The sparse resultant is the
resultant of the multiprojective toric variety with torus $\T_{M}$
associated to $\bfcA$ in the sense of R\'emond's multiprojective
elimination theory, whereas the sparse eliminant corresponds to what
is classically referred to as the sparse resultant, as is done
in~\cite{GKZ94, Stu94, CLO05} for instance.  Both are well-defined up
to the sign, the sparse resultant is a power of the sparse eliminant,
and they coincide when the family of supports $\bfcA$ is essential and
its affine span coincides with $M$, see \cite{DS15} or
\S\ref{sec:basic-prop-sparse} for precisions.

For each $i$ denote by $\Delta_{i} $ the convex hull of $\cA_{i}$ in
the vector space $M_{\R}=M\otimes\R$ and set
$\Delta=\sum_{i=0}^{n}\Delta_{i}$ for the Minkowski sum of these
lattice polytopes. For a vector
$\bfomega=(\bfomega_{0},\dots, \bfomega_{n})\in \R^{\bfcA}
=\prod_{i=0}^{n}\R^{\cA_{i}}$~set
\begin{equation}
  \label{eq:45}
\vartheta_{\bfomega_{i}}\colon \Delta_{i}\longrightarrow \R,
\ i=0,\dots, n, \and \Theta_{\bfomega} \colon \Delta\longrightarrow \R
\end{equation}
for the convex piecewise affine functions parametrizing the lower
envelope of the convex hull of the lifted supports
$ \wh\cA_{i}=\{(a,\omega_{i,a})\}_{a\in \cA_{i}} \subset M \times \R$,
$i=0,\dots, n$, and of their sum
$\sum_{i=0}^{n}\wh\cA_{i} \subset M \times \R$, respectively.   These
functions  define a mixed subdivision $S(\Theta_{\bfomega})$
of~$\Delta$, and for each $n$-cell $D$ of $S(\Theta_{\bfomega})$ they
also determine a decomposition
  \begin{equation*}
    D=\sum_{i=0}^{n} D_{i}
  \end{equation*}
  where each $D_{i}$ is a cell of the subdivision
  $S(\vartheta_{\bfomega_{i}})$ of $\Delta_{i}$, called the $i$-th
  component of~$D$. We can then consider the restriction
  \begin{displaymath}
   \bfcA_{D}=( \cA_{0}\cap D_{0},\dots, \cA_{n}\cap D_{n}) 
  \end{displaymath}
  of the given family of supports to these components.

  Our first main result, contained in Theorem \ref{mt1}, is the
  following factorization for the initial part of the sparse resultant
  with respect to $\bfomega$, defined as the the sum of the monomial
  terms whose exponents have minimal weight with respect to this
  vector. It generalizes a previous one by Sturmfels for the case when
  $\bfcA$ is essential \cite[Theorem 4.1]{Stu94}. 
  
\begin{theorem}
  \label{thm:6}
  Let $\bfomega\in \R^{\bfcA}$. Then
  \begin{displaymath}
    {\init}_\bfomega(\Res_{\bfcA})=\pm\prod_{D}\Res_{\bfcA_{D}} ,
  \end{displaymath}
the product being over the $n$-cells of $S(\Theta_{\bfomega})$.
\end{theorem}

A result by Philippon and the third author for the Chow weights of a
multiprojective toric variety \cite[Proposition~4.6]{PS08a} implies
that the order of the sparse resultant with respect to $\bfomega$ can
be expressed as the mixed integral of the $\vartheta_{\bfomega_{i}}$'s
(Theorem~\ref{mt1}).  Applying this together with Theorem \ref{thm:6},
we derive product formulae for the evaluation of $\Res_{\bfcA}$ by
setting some of the coefficients of the input Laurent polynomials to
zero (Theorem \ref{mt2} and Proposition \ref{minn}), correcting and
generalizing a previous one by Minimair \cite{Min03}, see
Remark~\ref{rem:13}.  These factorizations might be interesting from
the computational point of view, since they allow to extract the
sparse resultant associated to a family of supports contained in those
of $\bfcA$ as a factor of such an evaluation (Remark~\ref{osts}).

Apart from being homogeneous with respect to the sets of variables
$\bfu_{i}$, the sparse resultant is also homogeneous with respect to a
weighted grading on $\C[\bfu]$ associated to the action of $\T_{M}$ by
pullbacks on the system of Laurent polynomials
$\bfF=(F_{0},\dots, F_{n})$.  As another application of the
Philippon-Sombra formula, we compute its degree with respect to this
grading,  extending  a result by Gelfand, Kapranov and
Zelevinsky \cite[Chapter 9, Proposition 1.3]{GKZ94} and by Sturmfels
\cite[\S6]{Stu94} (Theorem~ \ref{antt}).

 To state our second main result, let
\begin{equation}
  \label{eq:88}
  \rho_{i}\colon \Delta_{i}\longrightarrow \R, \ i=0,\dots, n, \and
  \rho\colon \Delta \longrightarrow \R
\end{equation}
be the family of convex piecewise affine functions and its
inf-convolution defined by a vector of $\R^{\bfcA}$ as in
\eqref{eq:45}. {Set $\bfrho=(\rho_{0},\dots,\rho_{n})$ and} suppose that
the mixed subdivision $S(\rho)$ is tight (Definition \ref{def:7}).

Following Canny and Emiris \cite{CannyEmiris:easmr, CanEmi:sbasr} and
Sturmfels \cite{Stu94}, this data together with a generic translation
vector $\delta\in M_{\R}$ determines linear subspaces of
$ \C(\bfu)[M]^{n+1}$ and of $ \C(\bfu)[M] $, both of them generated by
monomials indexed by the lattice points in the translated polytope
$\Delta+\delta$, and such that the expression
\begin{displaymath}
(G_{0},\dots, G_{n}) \longmapsto \sum_{i=0}^{n}G_{i}\, F_{i}
\end{displaymath}
defines a linear map between them, see~\S\ref{sec:constr-first-prop}
for details.  The matrix of this linear map is denoted by
$\cH_{\bfcA,\bfrho}$, and we denote by $\cE_{\bfcA,\bfrho}$ the
principal submatrix corresponding to the lattice points in
$\Delta+\delta$ contained in the translated nonmixed $n$-cells of
$S(\rho)$ (Definition \ref{def:3}).

There is a nice interplay between these square matrices and the mixed
subdivisions of $\Delta$ that are coarser than $S(\rho)$.  Let
\begin{displaymath}
 \phi_{i}  \colon \Delta_{i}\longrightarrow  \R, \ i =0, \dots, n, \and
\phi\colon \Delta \longrightarrow \R
\end{displaymath}
be another family of convex piecewise affine functions and its
respective inf-convolution, and suppose that $S(\phi)$ is coarser than
$ S( \rho)$, a condition that is denoted by
$S(\phi) \preceq S( \rho)$.  For an $n$-cell $D$ of $S(\phi)$ denote
by $\bfrho_{D}=(\rho_{0}|_{D_{0}}, \dots, \rho_{n}|_{D_{n}})$ the
restriction to its components of this family of functions.

\begin{theorem}
  \label{thm:8}
  For $\bfomega=(\phi_{i}(a))_{i,a}\in \R^{\bfcA}$ we have that
  \begin{displaymath}
  \init_{\bfomega}(\det(\cH_{\bfcA,\bfrho} )) = \prod_{D}\det(\cH_{\bfcA_{ D},\bfrho_{D}}),
\end{displaymath}
the product being over the $n$-cells of $S(\phi)$.
\end{theorem}

More generally, a similar factorization holds for all the principal
minors of the Canny-Emiris matrix and in particular, for the
determinant of $\cE_{\bfcA,\bfrho}$ (Theorem~\ref{thm:2}).  Hence for
the vector defined by the $\phi_{i}$'s, the initial part of each of
these minors factorizes in the same way as the corresponding initial
part of the sparse resultant.  In contrast with the situation for the
sparse resultant, we do not know if this factorization holds for every
$\bfomega\in \R^{\bfcA}$ and as a matter of fact, it would be most
interesting to extend it to a larger class of vectors.

Another important property is that the Canny-Emiris matrices
associated to the restricted data $\bfcA_{D}$ and $\bfrho_{D}$ can be
retrieved as the evaluation of a principal submatrix of
$\cH_{\bfcA, \bfrho}$ by setting some of its coefficients to zero, and
that this construction is compatible with refinements of mixed
subdivisions (Propositions \ref{prop:17} and \ref{prop:7}).  We also
determine the homogeneities and degrees of $\det(\cH_{\bfcA,\bfrho})$
(Proposition \ref{prop:6}) and show that, under a mild hypothesis,
this determinant is a nonzero multiple of the sparse resultant
(Proposition \ref{prop:14}).  As a side question, such a hypothesis
does not seem necessary, and it would be interesting to get rid of it
(Remark \ref{rem:4}).

 The Canny-Emiris conjecture \cite[Conjecture 13.1]{CanEmi:sbasr}
states that, if the family of supports $\bfcA$ is essential and its
affine span coincides with $M$, then there is a family $\bfrho$ of
affine functions on the $\Delta_{i}$'s and a translation vector
$\delta\in M_{\R}$ such that
\begin{equation}
  \label{eq:86}
\Elim_{\bfcA} =\pm  \frac{\det(\cH_{\bfcA,\bfrho})}{\det(\cE_{\bfcA,\bfrho})}.
\end{equation}
As noted in \cite[\S13]{CanEmi:sbasr}, this identity does not hold
unconditionally since there are examples of families of convex
piecewise affine functions whose associated Canny-Emiris matrix and
distinguished principal submatrix do not verify it
(Example~\ref{exm:7}).

In \cite{DAndrea:msfsr}, the first author presented a recursive
procedure, using several mixed subdivisions on polytopes of every
possible dimension up to $n$, for constructing a square matrix with a
distinguished principal submatrix {such that the quotient of
  the determinants of these matrices coincides with
  $\Elim_{\bfcA}$}. In \cite{EmirisKonaxis:slmtfgusr}, Emiris and
Konaxis showed that in the generalized unmixed case, the D'Andrea
formula can be produced by a single mixed subdivision of $\Delta$, at
the price of adding many more points to the supports.

Our third main result gives a positive answer to a generalized version
of the Canny-Emiris conjecture.  To bypass the recursive steps of the
previous approaches, we consider chains of mixed subdivisions of
$\Delta$
\begin{displaymath}
   S(\theta_{0})\preceq \dots\preceq S(\theta_{n})
\end{displaymath}
with $S(\theta_{n})\preceq S(\rho)$.  The tight mixed subdivision
$S(\rho)$ is said to be admissible if there is such a chain which is
incremental in the sense of Definition~\ref{def:9} and satisfies the
conditions in Definition~\ref{def:5}.

Not every tight mixed subdivision of $\Delta$ is admissible (Example
\ref{exm:7}). However, for the family of supports $\bfcA$ one can
always find convex piecewise affine
functions~$\bfrho=(\rho_{0},\dots,\rho_{n})$ whose associated mixed
subdivision $S(\rho)$ is admissible.  For instance, this can be
realized by considering convex piecewise affine functions as
in~\eqref{eq:45} associated to a generic vector
$\bfnu=(\bfnu_{0}, \dots, \bfnu_{n})\in \R^{\bfcA}$ such that
$ \bfnu_{0}\gg \dots \gg \bfnu_{n} =\bfzero$.  Moreover, this vector
can be chosen so that the $\rho_{i}$'s are affine (Example \ref{exm:9}
and Corollary \ref{cor:2}).

\begin{theorem}
\label{thm:9}
If $S(\rho)$ is admissible, then
\begin{displaymath}
  \Res_{\bfcA} =\pm
\frac{\det(\cH_{\bfcA,\bfrho})}{\det(\cE_{\bfcA,\bfrho})}.
\end{displaymath}
\end{theorem}

 In the setting of the Canny-Emiris conjecture \eqref{eq:86}, the
sparse eliminant coincides with the sparse resultant. Hence this
statement follows from Theorem \ref{thm:9} taking a family of affine
functions whose associated mixed subdivision is admissible.

The statement of Theorem \ref{thm:9} is contained in Theorem
\ref{thm:3} and its proof uses a descent argument similar to that of
Macaulay in \cite{Macaulay:sfe} and the first author in
\cite{DAndrea:msfsr}, but its implementation is different. In contrast
to these references, our approach works directly with the Canny-Emiris
matrices associated to restrictions of the given data, without any
need of extending the Canny-Emiris construction to a larger one.  On
the other hand, it is interesting to note that such an enlargement is
possible, in analogy with the situation in \cite{Macaulay:sfe,
  DAndrea:msfsr}: the Canny-Emiris construction can be enlarged by
replacing the translation vector $\delta$ by a convex piecewise affine
function on a polytope, and Theorem \ref{thm:9} extends to this more
general situation (Remark \ref{rem:6}).

This result calls in for several research questions. To begin with, it
would be interesting to extend the class of mixed subdivisions to
which the quotient formula for the sparse resultant holds. Indeed,
such an extension might be possible by enlarging the range of validity
of Theorem \ref{thm:8}. In the mean time, for computational purposes
it would be interesting to have a fast way of checking if a given
tight mixed subdivision of $\Delta$ is admissible. In the same line,
it would be interesting to determine the probability that a given
tight mixed subdivision is admissible, with respect to a suitable
probability distribution.

As an application, we show that the Macaulay formula for the
homogeneous resultant corresponding to the critical degree is a
particular case of Theorem \ref{thm:9}, thus providing an independent
proof for it (Corollary~\ref{cor:4}). This is done by considering a
specific admissible mixed subdivision of scalar multiples of the
standard simplex such that its Canny-Emiris matrix and distinguished
principal submatrix coincide with those in that formula (Proposition
\ref{prop:2}).

The paper is organized as follows. In \S \ref{sec:polyhedral-geometry}
we explain the necessary notions and results from polyhedral geometry,
including convex piecewise affine functions on polyhedra and their
associated mixed subdivisions, and mixed volumes and integrals.
In~\S\ref{sec:basic-prop-sparse} we recall the basic definitions and
properties of sparse resultants and study some further aspects,
including their orders and initial parts, their homogeneities and
corresponding degrees, and their behavior under the evaluation at
systems of Laurent polynomials with smaller supports. In
\S\ref{sec:canny-emir-matr}  we study Canny-Emiris matrices: their
behavior under restriction of the data, the orders, initial parts,
homogeneities and degrees of their principal minors, some divisibility
properties of their determinants, and we give the proof of the
Canny-Emiris conjecture.   In \S\ref{sec:homog-result-1} we study
 the Macaulay formula for the homogeneous resultant in the framework
of the Canny-Emiris construction, and give some additional examples
and observations.

\medskip \noindent {\bf Acknowledgments.}  We thank the anonymous
reviewers {for their comments and suggestions for improvement
  on a previous version of} this paper. We also thank Carles Checa for
useful comments. Part of this work was done while the authors met at
the universities of Barcelona and of Buenos Aires.

D'Andrea was partially supported by the European Horizon 2020 research
and innovation program under the Marie Sklodowska-Curie grant
agreements No. 675789 and 860843. Jeronimo was partially supported by
the Argentinian CONICET research grant PIP 11220130100527CO and the
University of Buenos Aires research grants UBACYT 20020160100039BA and
UBACYT 20020190100116BA.  Both D'Andrea and Sombra were also partially
supported by the Spanish MINECO research projects MTM2015-65361-P and
PID2019-104047GB-I00, the Spanish Mar{\'\i}a de Maeztu program for
units of excellence MDM-2014-0445 (BGSMath Mar\'ia de Maeztu), and the
Argentinian research project ANPCyT PICT-2013-0294.

\section{Polyhedral geometry} \label{sec:polyhedral-geometry}

\subsection{Convex piecewise affine functions and mixed
  subdivisions}\label{sec:piec-affine-conv}

In this section we study the mixed subdivisions of convex polyhedra
produced by families of convex piecewise affine functions. We also
introduce some notions that will play a key role in our analysis of
Canny-Emiris matrices, and establish their feasability for a given
family of supports. Some of the techniques we use are similar to those
in \cite{HS95,JMSW09}. The necessary background on polyhedral geometry
can be found in \cite[Part~1]{Ewa96}.

Let $M\simeq \Z^{n}$ be a lattice of rank $n \in \N$ and
$N=M^{\vee}=\Hom(M,\Z) \simeq \Z^{n}$ its dual lattice. Set
$M_{\R}=M\otimes \R\simeq \R^{n}$ and
$N_{\R}=N\otimes \R\simeq \R^{n}$ for the associated $n$-dimensional
vector spaces, and denote by $\langle v, x\rangle$ the pairing between
$v\in N_{\R}$ and $x \in M_{\R}.$

A \emph{ convex polyhedron} of $M_{\R}$ is a subset of this vector
space given as the intersection of a finite family of closed
halfspaces.  For a convex polyhedron $\Delta$ of $M_{\R}$ we
denote by $\ri(\Delta)$ its \emph{relative interior}, that is, the
interior of this convex polyhedron relative to the minimal affine
subspace containing it.  Its \emph{support function} is the function
$h_{\Delta}\colon N_{\R}\to \R \cup\{-\infty\}$ defined~by
\begin{equation}
  \label{eq:28}
  h_{\Delta}(v)=\inf \{\langle v,x\rangle \mid x\in \Delta\}.
\end{equation}
The assignment $\Delta \mapsto h_{\Delta}$ is additive with respect
to the Minkowski sum of convex polyhedra and the pointwise sum of
functions.

For a vector $v\in N_{\R}$, the \emph{face of $\Delta$ in the
  direction of $v$} is defined as
\begin{equation}
  \label{eq:51}
  \Delta^{v}=\{x\in \Delta \mid \langle v, x\rangle = h_{\Delta}(v)\}.
\end{equation}

Let $\rho \colon \Delta \to \R $ be a convex piecewise affine
function. Its \emph{graph} and its \emph{epigraph} are the subsets of
$M_{\R}\times\R$ respectively defined as
\begin{displaymath}
\gr(\rho) =  \{(x,\rho(x)) \mid x\in \Delta \} \and
\epi(\rho)=\{(x,z)  \mid x\in \Delta,  z\ge \rho(x)\}.
\end{displaymath}
The epigraph is a convex polyhedron, whose faces of the form
$\epi(\rho)^{(v,1)}$, $v\in N_{\R}$, are contained in the graph, and
are called the \emph{faces} of $\gr(\rho)$.

The \emph{subdivision of $\Delta$ induced by $\rho $}, denoted by
$S(\rho)$, is the polyhedral subdivision of $\Delta$ given by the
image of the faces of the graph of $\rho $ with respect to the
projection $\pi\colon M_{\R}\times \R \rightarrow M_{\R}$.  Its
elements are called the \emph{cells} of this subdivision. For
$j \ge -1$, we denote by $S(\rho)^{j}$ the set of cells of $S(\rho)$
of dimension $j$, or \emph{$j$-cells}. Their union gives the
$j$-\emph{skeleton} of $S(\rho)$, denoted by $|S(\rho)^{j}|$.  For a
vector $v\in N_{\R}$, the corresponding cell of $S(\rho)$ is denoted
by
  \begin{equation}
    \label{eq:35}
\Gamma(\rho,v)= \pi\big(\epi(\rho)^{(v,1)}\big).
\end{equation}
For $x\in \Delta$ we have that
\begin{equation}
  \label{eq:79}
\rho(x)\ge \langle -v,x\rangle +
h_{\epi(\rho)}(v,1),
\end{equation}
and the equality holds if and only if $x\in \Gamma(\rho,v)$.

Let $\rho \colon \Delta\to\R$ and $\rho'\colon \Delta' \to \R$ be
convex piecewise affine functions on convex polyhedra. Their
\emph{inf-convolution}, denoted by $\rho \boxplus \rho'$, is the
convex piecewise affine function on the Minkowski sum
$\Delta+\Delta'$ defined by
\begin{equation}
  \label{eq:6}
  (  \rho\boxplus \rho')(x)=\inf\{\rho(y)+\rho'(y') \mid y\in \Delta, y'\in \Delta'  \text{ and } x=y+y'\}.
\end{equation}
Alternatively, it can be defined as the function parametrizing the
lower envelope of $\epi(\rho)+\epi(\rho')$, that is,
\begin{equation*}
  (  \rho\boxplus \rho')(x)=\inf\{z\in \R \mid (x,z)\in \epi(\rho)+\epi(\rho')\}.
\end{equation*}
The Minkowski sum $ \epi(\rho)+\epi(\rho')$ is a convex
polyhedron, and so  $\rho\boxplus \rho'$ is a convex piecewise affine
function on $ \Delta+\Delta'$ and for every point $x$ in this set,
the infimum in~\eqref{eq:6} is attained.

Now for $s\in \N$ let $\rho _{i} \colon \Delta_{i}\to\R$, $i=0,\dots, s$, be a
family of $s+1$ convex piecewise affine functions on convex polyhedra and
set $ \rho= \bigboxplus_{i=0}^{s}\rho_{i}$ for their inf-convolution,
which is a convex piecewise affine function on the Minkowski sum
$\Delta=\sum_{i=0}^{s}\Delta_{i}$.  The subdivision $S(\rho)$ of
$\Delta$ is called a \emph{mixed subdivision} of $\Delta$.

For $i=0,\dots, s$ we respectively denote by
\begin{equation}
  \label{eq:13}
  \Delta_{i}^{\cc}=\sum_{j\ne i}\Delta_{j} \and
  \rho_{i}^{\cc}=\bigboxplus_{j\ne i}\rho_{j}
\end{equation}
the convex polyhedron and the convex piecewise affine function
respectively defined by the $i$-th complementary Minkowski sum and by
the $i$-th complementary inf-convolution. We have that
$\Delta_{i}^{\cc}+\Delta_{i}=\Delta$ and
$\rho _{i}^{\cc} \boxplus \rho_{i}=\rho$.

  For  $C \in  S(\rho)$ consider the  subset of $M_{\R}^{s+1}$ defined as
  \begin{displaymath}
   \Pi_{C}=\Big\{(x_{0},\dots, x_{s}) \in \prod_{i=0}^{s}
    \Delta_{i} \, \Big| \
\sum_{i=0}^{s}x_{i}\in C \text{ and }    \rho\Big(\sum_{i=0}^{s}x_{i}\Big)=\sum_{i=0}^{s}\rho_{i}(x_{i}) \Big\}.
  \end{displaymath}
  For $i=0,\dots, s$ let $\pi_{i} \colon M_{\R}^{s+1}\to M_{\R}$ denote
  the projection onto the $i$-th factor.  The \emph{$i$-th component}
  of $C$ is the nonempty subset of $\Delta_{i}$ defined as
\begin{equation}
  \label{eq:66}
  C_{i}=   \pi_{i}(\Pi_{C}).
\end{equation}

The next two results give the basic properties of the components of
the cells of a mixed subdivision.

\begin{proposition}
  \label{prop:10}
  Let $C\in S(\rho)$. Then
  \begin{enumerate}
  \item \label{item:2} for  $v\in N_{\R}$ such that
    $C=\Gamma(\rho,v)$ we have that
    $C_{i}=\Gamma(\rho_{i},v)\in S(\rho_{i})$ for all~$i$,
  \item \label{item:3} $\displaystyle C=\sum_{i=0}^{s} C_{i}$,
  \item \label{item:5} for $x\in C$ and $x_{i}\in \Delta_{i}$,
    $i=0,\dots, s$, such that $x=\sum_{i=0}^{s}x_{i}$ we have that
    $\rho(x)=\sum_{i=0}^{s}\rho_{i}(x_{i})$ if and only if
    $x_{i}\in C_{i}$ for all $i$.
    \end{enumerate}
\end{proposition}

\begin{proof}
  Let $v \in N_{\R}$ and set for short $\kappa=h_{\epi(\rho)}(v,1)$
  and $\kappa_{i}=h_{\epi(\rho_{i})}(v,1)$ for each~$i$. We have that
  $ \epi(\rho)= \sum_{i=0}^{s}\epi(\rho_{i}) $ and so, by the
  additivity of the support function,
\begin{displaymath}
  \kappa=\sum_{i=0}^{s}\kappa_{i}.
\end{displaymath}

Let $i\in \{0,\dots, s\}$ and $x_{i}\in C_{i}$. Choose
$x_{j}\in C_{j}$, $j\ne i$, such that
$(x_{0},\dots, x_{s})\in \Pi_{C}$ and set $x=\sum_{j=0}^{s}x_{j}$,
so that $x\in C$ and $\rho(x)=\sum_{j=0}^{s}\rho_{j}(x_{j})$.  Hence
$(x,\rho(x))\in \epi(\rho)^{(v,1)}$ and
$(x_{j},\rho_{j}(x_{j}))\in \epi(\rho_{j})$ for all $j$ and so
\begin{equation*}
\kappa =  \langle (v,1),(x,\rho(x))\rangle =
  \sum_{j=0}^{s} \langle
  (v,1),(x_{j},\rho_{j}(x_{j}))\rangle  \ge \sum_{j=0}^{s}  \kappa_{j}=\kappa.
\end{equation*}
Thus $\langle (v,1),(x_{i},\rho_{i}(x_{i}))\rangle = \kappa_{i}$ or
equivalently $(x_{i},\rho_{i}(x_{i}))\in \epi(\rho_{i})^{(v,1)}$,
which implies that $x_{i}\in\Gamma(\rho_{i},v)$.

Conversely, let $x_{i}\in\Gamma(\rho_{i},v)$. Choose
$x_{j}\in\Gamma(\rho_{j},v)$, $j\ne i$, and set $x=\sum_{j=0}^{s}x_{j}$
and $t=\sum_{j=0}^{s}\rho_{j}(x_{j})$. We have that $t\ge \rho(x)$ and
so $(x,t)\in \epi(\rho)$. Moreover
\begin{equation*}
  \langle  (v,1),(x,t)\rangle =  \sum_{j=0}^{s}  \langle
  (v,1),(x_{j},\rho_{j}(x_{j}))\rangle = \sum_{j=0}^{s}\kappa_{j}=\kappa.
\end{equation*}
Hence $(x,t)\in \epi(\rho)^{(v,1)}$ and so $x\in C$ and
$t=\rho(x)$. In particular, $x_{i} \in C_{i}$ and we conclude that
$C_{i}=\Gamma(\rho_{i},v)$, proving \eqref{item:2}.

Now let $x_{i}\in C_{i}$, $i=0,\dots, s$, and set
$x=\sum_{i=0}^{s}x_{i}$. By \eqref{item:2}, for any $v\in N_{\R}$ such
that $C=\Gamma(\rho,v)$ we have that $C_{i}=\Gamma(\rho_{i},v)$, and the last
part of the proof of this statement shows that $x\in C$ and
$\rho(x)=\sum_{i=0}^{s}\rho_{i}(x_{i})$. This proves both that
$ \sum_{i=0}^{s}C_{i}\subset C$ and the ``if'' part in \eqref{item:5}.

Conversely, for each $x\in C$ the infimum in \eqref{eq:6} is attained,
and so there are $x_{i}\in C_{i}$, $i=0,\dots, s$, with
$x=\sum_{i=0}^{s}x_{i}$. Hence $C=\sum_{i=0}^{s}C_{i}$ as stated in
\eqref{item:3}, whereas the ``only if'' part in~\eqref{item:5} is
immediate from the definition of the components in \eqref{eq:66}.
\end{proof}

\begin{proposition}
\label{prop:20}
  Let $C, C'\in S(\rho)$ and $i\in \{0,\dots, s\}$ such that their
  respective $i$-th components have both dimension $n$ and
  coincide. Then $C=C'$.
\end{proposition}

\begin{proof}
  Since both $C_{i}$ and $C_{i}'$ have dimension $n$ and coincide,
  there is a unique $v\in N_{\R}$ with $ C_{i}=C_{i}'=\Gamma(\rho_{i},v)$.
Proposition \ref{prop:10}\eqref{item:2} then implies that
  $C=\Gamma(\rho,v)=C'$.
\end{proof}

\begin{definition}
  \label{def:7}
  A mixed subdivision $S(\rho)$ on $\Delta$ is \emph{tight} if for every
  $n$-cell $C$ of $ S(\rho)$,
  \begin{displaymath}
 \sum_{i=0}^{s}\dim(C_{i})=n.
  \end{displaymath}
  If this condition holds, when $s=n-1$ an $n$-cell of $S(\rho)$ is
  \emph{mixed} if all its components are segments and when $s=n$,
  for $k=0,\dots, n$ an $n$-cell of $S(\rho)$ is \emph{$k$-mixed} if
  its $i$-th component is a segment for all $i\ne k$ (and so its
  $k$-th component is a point).
\end{definition}


The set of mixed subdivisions of $\Delta$ is partially ordered by refinements:
for another mixed subdivision $S(\rho')$ of $\Delta$ given by a
family of convex piecewise affine functions
$ \rho'_{i}\colon \Delta_{i}\to \R$, $i=0,\dots, s$, we say that
$S(\rho)$ is a \emph{refinement} of $S(\rho')$, denoted~by
\begin{equation*}
S(\rho)\succeq S(\rho') \quad \text{ or } \quad S(\rho')\preceq S(\rho),
\end{equation*}
if for all $C\in S(\rho)$ there is $D\in S(\rho')$ such that
$C\subset D$ and that $C_{i}\subset D_{i}$ for all $i$.

\begin{definition}
  \label{def:9}
  An \emph{incremental} chain of mixed subdivisions of $\Delta$ is a
  chain
  \begin{math}
     S(\theta_{0})\preceq \dots\preceq S(\theta_{s})
  \end{math}
  where, for $k=0,\dots, s $, the mixed subdivision $S(\theta_{k})$ is
  induced by the inf-convolution $ \theta_{k}\colon \Delta\to \R$ of
  a family of convex piecewise affine functions
  $\theta_{k,i} \colon \Delta_{i}\to\R$, $i=0,\dots, s$, such that
  $\theta_{k,i}= 0|_{\Delta_{i}}$ for $i\ge k$.

  This incremental chain is \emph{tight} if, for each $k$, the mixed
  subdivision of $\Delta$ (considered as the sum of the $k+1$
  polytopes
  $\Delta_0, \dots, \Delta_{k-1}, \sum_{{i=k}}^{s} \Delta_{i}$)
  induced by the convex piecewise affine functions
\begin{displaymath}
\theta_{k,i}\colon \Delta_{i} \to \R, \ i=0,\dots,
k-1, \and {  \bigboxplus_{i=k}^{s}\theta_{k,i}=0\Big|_{\sum_{{i=k}}^{s}
\Delta_{i}}},
\end{displaymath}
is tight in the sense of Definition \ref{def:7}.
\end{definition}

\begin{remark}
  \label{rem:9}
  The notion of incremental chain of mixed subdivisions of $\Delta$
  might be easily extended to chains of arbitrary length. We have
  chosen to restrict it to chains of length equal to the number of
  polyhedra $\Delta_{i}$ because it is the only case of interest for
  the proof of Theorem \ref{thm:3}. 
\end{remark}

\begin{remark}
  \label{rem:17}
  If $ S(\theta_{0})\preceq \dots\preceq S(\theta_{s})$ is a tight
  incremental chain, then $S(\theta_{s})$ is a tight mixed
  subdivision.
\end{remark}

\begin{example}
  \label{exm:8}
  Let $n=2$ and $M=\Z^{2}$. Set $d_{0}=1$, $d_{1}=3$ and $d_{2}=2$ and
  for $i=0,1,2$ consider the triangle
  $\Delta_{i}=\{(x_{1},x_{2})\in (\R_{\ge 0})^{2}\mid x_{1}+x_{2}\le
  d_{i}\}$. Consider also the affine functions
  $\rho_{i}\colon \Delta_{i}\to \R$, $i=0,1$, defined by
  \begin{displaymath}
    \rho_{0}(x_{1},x_{2})= 3\, x_{1}+6\, x_{2} \and
    \rho_{1}(x_{1},x_{2})= 2\, x_{1}+x_{2}.
  \end{displaymath}
  For $k, i=0,1,2$ set $\theta_{k,i}=\rho_{i}$ if $i<k$ and
  $\theta_{k,i}=0|_{\Delta_{i}}$ if $i\ge k$, and then for $k=0,1,2$ set
  $\theta_{k}=\theta_{k,0}\boxplus\theta_{k,1}\boxplus\theta_{k,2}$. Hence
  $S(\theta_{0}) \preceq S(\theta_{1}) \preceq S(\theta_{2}) $ is a
  tight incremental chain of mixed subdivisions of the triangle
  $\Delta=\{(x_{1},x_{2})\in (\R_{\ge 0})^{2}\mid x_{1}+x_{2}\le 6\}$
  (Figure~\ref{fig:7}).

\begin{figure}[ht]
\begin{tikzpicture}[scale=0.5]
\draw[thick] (0,0)--(6,0)--(0,6)--cycle;
\draw (3,-0.1) node[below] {$S(\theta_0)$};
\end{tikzpicture} \qquad
\begin{tikzpicture}[scale=0.5]
\draw[thick] (0,0)--(6,0)--(0,6)--cycle;
\draw (3,-0.1) node[below] {$S(\theta_1)$};
\draw (0,5)--(5,0);
\draw (0,5)--(1,5);
\end{tikzpicture} \qquad
\begin{tikzpicture}[scale=0.5]
\draw[thick] (0,0)--(6,0)--(0,6)--cycle;
\draw (0,5)--(5,0);
\draw (0,5)--(1,5);
\draw (2,0)--(2,3);
\draw (2,3)--(3,3);
\draw (2,0)--(0,2);
\draw (3,-0.1) node[below] {$S(\theta_2)$};
\end{tikzpicture}
\vspace{-3mm}
\caption{A tight incremental chain}
\label{fig:7}
\end{figure}
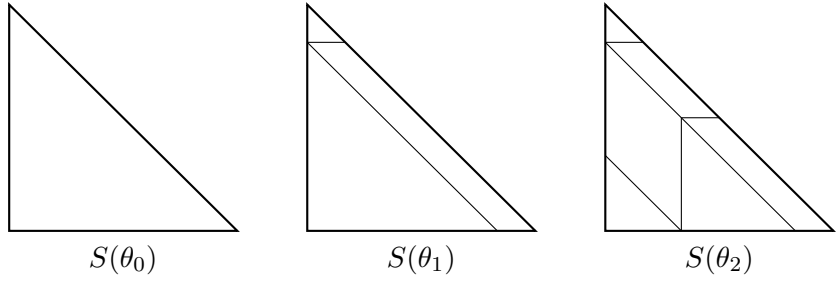
\end{example}

Convex piecewise affine functions on lattice polytopes might be
constructed by means of finite sets of lattice points and lifting
vectors, as we next describe.  For $i=0,\dots, s$ let
$\cA_{i}\subset M$ be a nonempty finite subset and
$\bfnu_{i}\in \R^{\cA_{i}}$ a vector.  Set $\Delta_{i}=\conv(\cA_{i})$
for the lattice polytope of $M_{\R}$ given by the convex hull of
$\cA_{i}$ and
\begin{equation}
  \label{eq:14}
  \vartheta_{\bfnu_{i}}
\colon \Delta_{i}
  \longrightarrow  \R
\end{equation}
for the convex piecewise affine function parametrizing the lower
envelope of the lifted polytope
$ \conv ( \{(a,\nu_{i,a})\}_{a\in \cA_{i}} ) \subset M_{\R} \times
\R$.  Set also $\bfcA=( \cA_{0},\dots, \cA_{s})$,
$\bfnu=(\bfnu_{0},\dots, \bfnu_{s}) \in \R^{\bfcA}=\prod_{i=0}^{s}\R^{\cA_{i}}$ and
\begin{equation}
  \label{eq:19}
  \Theta_{\bfnu} =\bigboxplus_{i=0}^{s}
  \vartheta_{\bfnu_{i}} .
\end{equation}
This latter is a convex piecewise affine function on the Minkowski sum
$\Delta=\sum_{i=0}^{s}\Delta_{i}$, and $ S(\Theta_{\bfnu})$ is a mixed
subdivision of this polytope.

The next result shows that a generic choice of lifting vectors
produces a mixed subdivision that is tight. Furthermore, this choice
can be made among the lifting vectors whose associated  functions
 are affine.

Consider the linear map $T_{\bfcA}\colon N^{s+1}\to \R^{\bfcA}$ defined
by
\begin{equation}
  \label{eq:26}
  T_{\bfcA}(v_{0},\dots, v_{s})=(\langle v_{i}, a\rangle)_{i\in \{0,\dots, s\},
    a\in \cA_{i}}.
\end{equation}
The convex piecewise affine functions associated to the vectors in its
image are affine.

{The next result is similar to those in \cite[page 1546]{HS95}
  and \cite[Lemma 2.1]{JMSW09}.}

\begin{proposition}
  \label{prop:1}
  There is a finite union of hyperplanes $W\subset \R^{\bfcA}$ not
  containing $T_{\bfcA}(N^{s+1})$ such that for all
  $\bfnu\in \R^{\bfcA}\setminus W$ the mixed subdivision
  $S(\Theta_{\bfnu})$ of $\Delta$  is tight.
  \end{proposition}

  To prove it, we need the following auxiliary result.   For
  $i=0,\dots, s$ let $e_{i}$ be the $(i+1)$-th vector in the standard
  basis of $\R^{s+1}$.

\begin{lemma}
  \label{lemm:13}
  For $\bfnu\in\R^{\bfcA}$ let
  $\wh L\subset M_{\R}\times \R^{s+1} \times \R$ denote the linear span
  of the vectors
  $ \{(a, e_{i},\nu_{i,a})\}_{i\in \{0,\dots, s\},  a\in \cA_{i}}$.  Set also
  $\wh \cA_{i}= \{(a,\nu_{i,a})\}_{a\in \cA_{i}} \subset M_{\R}\times
  \R$, $i=0,\dots, s$. Then
  \begin{equation*}
    \dim(\wh L)=
    \dim\Big(\sum_{i=0}^{s}\conv(\wh\cA_{i})\Big) +s+1.
  \end{equation*}
\end{lemma}

\begin{proof}
  For each $i$ let $\wh L_{i}\subset M_{\R}\times \R $ denote the
  affine span of the nonempty finite subset~$\wh\cA_{i}$. Making
  linear combinations between the generators of the linear subspace
  $\wh L$, we easily deduce that
  $ \dim(\wh L)= \dim(\sum_{i=0}^{s}\wh L_{i}) +s+1$.  The statement then
  follows from the fact that
  $ \dim(\sum_{i=0}^{s}\wh L_{i}) =
  \dim(\sum_{i=0}^{s}\conv(\wh\cA_{i}))$.
\end{proof}

\begin{proof}[Proof of Proposition \ref{prop:1}]
  For $i=0,\dots, s$ let $\bfu_{i}$ be a set of $\# \cA_{i}$ variables
  and let $\bfu=(\bfu_{0}, \dots, \bfu_{s})$.  Fix an isomorphism
  $M_{\R}\simeq \R^{n}$. Then for each family
  $\bfcD=(\cD_{0},\dots, \cD_{s})$ of nonempty subsets
  $\cD_{i}\subset\cA_{i}$ satisfying the conditions
  \begin{enumerate}
  \item \label{item:1} $    \sum_{i=0}^{s}\#\cD_{i}= n+s+2$,
    \smallskip
    \item \label{item:61} $ \dim(\conv(\cD_{i})) =
    \#\cD_{i}-1$, $ \ i=0,\dots, s$,
      \smallskip
  \item \label{item:48} $
    \dim(\sum_{i=0}^{s}\conv(\cD_{i})) =n$
  \end{enumerate}
  set $ \cG_{\bfcD}\in \R[\bfu]^{ \bfcD\times (n+s+2)} $ for the
  square matrix made of the row vectors
  \begin{displaymath}
 (a, e_{i},u_{i,a}) \in \R[\bfu]^{n+s+2}, \ i=0,\dots, s
  \and  a\in \cD_{i}.
  \end{displaymath}
Set also
  $ G_{\bfcD}=\det(\cG_{\bfcD})\in \R[\bfu]$, which is a linear form.

  The conditions on $\bfcD$ imply that for each $i$ there is
  $\bfnu \in \R^{\bfcA}$ such that, setting
\begin{equation}
  \label{eq:82}
 \wh\cD_{i}= \{(a,\nu_{i,a})\}_{a\in \cD_{i}} \subset
M_{\R}\times \R, \ i=0,\dots, s,
\end{equation}
we have that $ \dim( \sum_{i=0}^{s}\conv(\wh \cD_{i})) = n+1 $.
Moreover, this condition can be fulfilled with a vector
$\bfnu \in T_{\bfcA}(N^{s+1})$. Indeed, the condition \eqref{item:61}
implies that the $\mathcal{D}_{i}$'s are simplexes which by the
condition \eqref{item:1}, have dimensions that sum up $n+1$. Assuming
without loss of generality that each $\mathcal{D}_{i}$ contains $0$ as
one its points, the remaining points in {the union of} these
sets determine $n+1$ directions which by condition \eqref{item:48}
span $M_{\mathbb{R}}$. {Then a possible choice for $\bfnu$
  satisfying \eqref{eq:82} consists in setting $\nu_{i,a}=0$ for each
  $i$ and all $a\in \mathcal{D}_{i}$ except one of them, for which
  this coordinate is set to $1$.}

Lemma~\ref{lemm:13} then implies
$ G_{\bfcD}(\bfnu)\ne 0$ and in particular, the zero set of
$G_{\bfcD}$ is a hyperplane not containing the linear subspace
$T_{\bfcA}(N^{s+1})$. The set $W$ is then defined as the union of all
these hyperplanes.

Now let $\bfnu\in \R^{\bfcA} $ and suppose that $S(\Theta_{\bfnu})$ is
not tight. Let $C$ be an $n$-cell of this mixed subdivision such that
$ \sum_{i=0}^{s}\dim(C_{i})>n$.  Then we can choose nonempty finite
subsets $\cD_{i} \subset \cA_{i}\cap C_{i}$, $i=0,\dots, s$,
satisfying the conditions \eqref{item:1}, \eqref{item:61} and
\eqref{item:48}.  Such a choice may be accomplished by picking
simplexes defined by points in the finite subsets
$ \cA_{i}\cap C_{i}$, $i=0,\dots, s$, whose dimensions sum up $n+1$.

Let $P$ be the face of the graph of $\Theta_{\bfnu}$
corresponding to  $C$, and for each $i$ let $P_{i}$ be the
face of the graph of $\vartheta_{\bfnu_{i}}$ corresponding to the
component $C_{i}$.  For each $i$ the lifted set $\wh\cD_{i}$ as in \eqref{eq:82}
is contained in $ P_{i}$ and so
\begin{displaymath}
\sum_{i=0}^{s}\conv (\wh \cD_{i}) \subset \sum_{i=0}^{s} P_{i} = P.
\end{displaymath}
Hence $ \dim( \sum_{i=0}^{s} \conv (\wh \cD_{i})) \le \dim(P)=\dim(C) = n$.
Lemma \ref{lemm:13} then implies that $G_{\bfcD}(\bfnu)=0$ and so
$\bfnu\in W$, concluding the proof.
\end{proof}

The next corollary shows that we might fix one of the lifting vectors
to zero and still get a mixed subdivision that is tight. Set for
short
\begin{displaymath}
  \bfcA'=(\cA_{0},\dots, \cA_{s-1})
\end{displaymath}
and let $T_{\bfcA'}\colon N^{s} \to \R^{\bfcA'}$ be the
corresponding linear map as in \eqref{eq:26}.

\begin{corollary}
  \label{cor:7}
  There is a finite union of hyperplanes $W'\subset \R^{\bfcA'}$ not
  containing  $T_{\bfcA'}(N^{s})$ such that for
  all $\bfnu' \in \R^{\bfcA'}\setminus W'$ the mixed subdivision
  $S(\Theta_{\bfnu})$ of $\Delta$ associated to the vector
  $\bfnu=(\bfnu', \bfzero)\in \R^{\bfcA}$ is tight.
\end{corollary}

\begin{proof}
  With notation as in Proposition \ref{prop:1}, choose a vector
  $(w_{0},\dots, w_{s})\in N^{s+1}$ whose image with respect to the
  linear map $T_{\bfcA}$ does not lie in $ W$. Let
  $\bfzeta_{i} =(\langle w_{s},a\rangle )_{a\in \cA_{i}}\in
  \R^{\cA_{i}}$, $i=0,\dots, s$, and
\begin{displaymath}
  W'=\{(\bfnu'_{0},\dots, \bfnu'_{s-1})\in \R^{\bfcA'}\mid (\bfnu'_{0}+
  \bfzeta_{0},\dots, \bfnu'_{s-1}+ \bfzeta_{s-1}, \bfzeta_{s})\in W\},
\end{displaymath}
which is a finite union of hyperplanes of $\R^{\bfcA'}$. We have that
\begin{displaymath}
(  T_{\bfcA'}(w_{0}-w_{s}, \dots, w_{s-1}-w_{s})  , \bfzero) +
(\bfzeta_{0},\dots,  \bfzeta_{s-1}, \bfzeta_{s})=
 T_{\bfcA}(w_{0},\dots, w_{s-1}, w_{s})  \notin W.
\end{displaymath}
Hence $T_{\bfcA'}(w_{0}-w_{s}, \dots, w_{s-1}-w_{s}) \notin W'$, and
so $W' \not\supset T_{\bfcA'}(N^{s-1})$.

By Proposition \ref{prop:1}, for
$\bfnu' =(\bfnu' _{0}, \dots, \bfnu' _{s-1})\in \R^{\bfcA'}\setminus
W'$ the mixed subdivision associated to
$ (\bfnu'_{0}+ \bfzeta_{0},\dots, \bfnu'_{s-1}+ \bfzeta_{s-1},
\bfzeta_{s})\in \R^{\bfcA}$ is tight.  Hence this is also the case for
the mixed subdivision associated to the vector
$(\bfnu',\bfzero) \in \R^{\bfcA}$, since the corresponding functions
differ by a globally defined linear one.
\end{proof}

The next result shows that small perturbations of a given family of
lifting vectors produce finer mixed subdivisions.

\begin{proposition}
  \label{prop:9}
  Let $\bfnu\in \R^{\bfcA}$. There is a neighborhood $U $ of $\bfnu$
  such that for all $\wt \bfnu\in U $ we have that
  $S(\Theta_{\wt \bfnu})\succeq S(\Theta_{\bfnu})$.
\end{proposition}

\begin{proof}
For each $n$-cell $C$ of
  $S(\Theta_{\bfnu})$ denote by $v_{C}$ the unique vector in
  $N_{\R}$ such that $C=\Gamma(\Theta_{\bfnu}, v_{C})$. By Proposition
  \ref{prop:10}\eqref{item:2} and the inequality in \eqref{eq:79}, for
  each $i$ there is $\kappa_{C,i}\in \R$ such that, for
  $x\in \Delta_{i}$,
  \begin{equation*}
  \vartheta_{\bfnu_{i}}(x)\ge \langle -v_{C},x\rangle +\kappa_{C,i}
  \end{equation*}
  with equality if and only if $x \in C_{i}$.  Hence there is $c>0$
  such that for all $a\in \cA_{i}\setminus C_{i}$,
\begin{equation}
  \label{eq:80}
   \vartheta_{\bfnu_{i}}(a)\ge \langle -v_{C}, a  \rangle
+\kappa_{C,i}  + c.
\end{equation}

Let $\varepsilon>0$ and $\wt \bfnu\in  \R^{\bfcA} $ with
$\|\wt\bfnu-\bfnu\|_{\infty}<\varepsilon$, where $\|\cdot\|_{\infty}$
denotes the $\ell^{\infty}$-norm of~$\R^{\bfcA}$. Then for all $i$ and
$x\in \Delta_{i}$ we have that
\begin{equation}
  \label{eq:83}
  |\vartheta_{\wt \bfnu_{i}}(x) -\vartheta_{ \bfnu_{i}}(x) |
<\varepsilon.
\end{equation}

Fix a norm $\|\cdot\|$ on~$M_{\R}$ and let $\|\cdot\|^{\vee}$ be the
corresponding operator norm on $N_{\R}$, so that for $v\in N_{\R}$ and $x\in M_{\R}$ we have that
\begin{equation}
  \label{eq:23}
  |  \langle v,x\rangle | \le \| v\|^{\vee} \, \|x\|.
\end{equation}

Let $\wt C$ be an $n$-cell of $S(\Theta_{\wt\bfnu})$ and, similarly as
before, denote by $v_{\wt C} \in N_{\R}$ and
$\kappa_{\wt C, i} \in \R$, $i=0,\dots, s$, the corresponding vector
and constants.  Then for each $n$-cell $C$ of $S(\Theta_{\bfnu})$ with
$\dim(\wt C \cap C)=n$ there is $K>0$ such that
\begin{equation}
  \label{eq:84}
\| v_{\wt C} - v_{C}\|^{\vee}, |\kappa_{\wt C, i}- \kappa_{C,i}| <K \, \varepsilon.
\end{equation}
Since the number of possible pairs $(\wt C, C)$ for varying
$\wt \bfnu\in \R^{\bfcA}$ is finite, the constant $K>0$ can be taken
independently of the choice of these $n$-cells.

From the inequalities in \eqref{eq:80} and \eqref{eq:83} we deduce that for all $a\in \cA_{i}\setminus C_{i}$,
\begin{displaymath}
\vartheta_{\wt \bfnu_{i}}(a)>  \vartheta_{\bfnu_{i}}( a  ) -
  \varepsilon \ge
\langle -v_{C},a\rangle +\kappa_{C,i}  +c-
\varepsilon,
\end{displaymath}
and from the inequalities in \eqref{eq:23} and \eqref{eq:84},
\begin{displaymath}
\langle -v_{C},a\rangle +\kappa_{C,i}
>
\langle -v_{\wt C},a\rangle +\kappa_{\wt C,i}  -   K \varepsilon\Big(  \sup_{x\in \Delta_{i}}\|x\| +1\Big).
\end{displaymath}
Then for $\varepsilon>0$ sufficiently small we have that
$ \vartheta_{\wt \bfnu_{i}}(a) >\langle -v_{\wt C},a\rangle
+\kappa_{\wt C,i} $ for all $a\in \cA_{i}\setminus C_{i}$, which
implies that $\wt C_{i}\subset C_{i}$ for all $i$. In turn, by
Proposition~\ref{prop:10}\eqref{item:3} this implies that
$\wt C\subset C$.

Since this holds for every $n$-cell of $S(\Theta_{\wt \bfnu})$, we
deduce that this mixed subdivision refines $S(\Theta_{\bfnu})$. The
statement follows by taking $U$ as the ball of $\R^{\bfcA}$ centered
at  $\bfnu$  of radius $\varepsilon$ with respect to the
$\ell^{\infty}$-norm.
\end{proof}

As an application of these results, we can exhibit an explicit family
of tight incremental chains of mixed subdivisions of the polytope  $\Delta$ 
associated to the data $\bfcA$.

\begin{example}
  \label{exm:9}
  Set $\cA_{> i }=\sum_{j> i} \cA_{j} \subset M$, $i=0,\dots, s-1$,
  and for each $i$ denote by $ W_{i}'$ the finite union of hyperplanes
  given by Corollary~\ref{cor:7} applied to the family
  $(\cA_{0}, \dots, \cA_{i}, \cA_{> i })$ of $i+{2}$ nonempty
  subsets of $M$.

For $i=0,\dots, s-1$ choose iteratively 
$\bfnu_{i} \in \R^{\cA_{i}}$ such that
$(\bfnu_{0},\dots, \bfnu_{i}) \notin W_{i}'$ and 
$ (\bfnu_{0},\dots, \bfnu_{i-1}, \bfnu_{i}, \bfzero, \dots, \bfzero)
\in \R^{\bfcA} $ lies in the neighborhood of
$ (\bfnu_{0},\dots, \bfnu_{i-1}, \bfzero, \bfzero, \dots, \bfzero)$
given by Proposition \ref{prop:9}.  Then for $k=0,\dots, s$ consider
the family of convex piecewise affine functions
$ \theta_{k,i}\colon \Delta_{i}\to \R$, $i=0,\dots, s$, defined as
$ \theta_{k,i}= \vartheta_{ \bfnu_{i}} $ if $i<k$ and as
$ \theta_{k,i}=0|_{\Delta_{i}}$ if $i\ge k$. Their  inf-convolution
\begin{displaymath}
  \theta_{k}=\bigboxplus_{i=0}^{s}\theta_{k,i},
\end{displaymath}
 is a convex piecewise affine function on $\Delta$.  By
Corollary \ref{cor:7} and Proposition~\ref{prop:9}, 
\begin{displaymath}
  S(\theta_{0})\preceq \dots\preceq S(\theta_{s})
\end{displaymath}
is a tight incremental chain of mixed subdivisions of $\Delta$.
Moreover, Corollary \ref{cor:7} allows to choose the vector
$(\bfnu_{0},\dots, \bfnu_{s-1})\in \R^{\bfcA'}$ in the image of the
linear map $T_{\bfcA'}$, so that the convex piecewise affine functions
$\theta_{k,i}$ are indeed affine.
\end{example}
  
\subsection{Mixed volumes and mixed integrals}\label{sec:mixed-volumes-mixed}

The mixed volume of $n$ convex bodies of $M_{\R}$ is a polarization of
the notion of volume of a single one. Here we recall its definition
and basic properties,  referring to \cite[Chapter IV]{Ewa96} for the
corresponding proofs.  We restrict the presentation to polytopes,
which are the only convex bodies  appearing  in this paper.

We denote by $\vol_{M}$ the Haar measure on the vector space $M_{\R}$
that is normalized so that the lattice $M$ has covolume~1.

\begin{definition}
  \label{def:19}
  The \emph{mixed volume} of a family of polytopes
  $\Delta_{i}\subset M_{\R}$, $i=1,\dots, n$, is defined as
  \begin{equation*}
  \MV_M(\Delta_{1},\dots,\Delta_{n}) = \sum_{j=1}^n
  (-1)^{n - j} \sum_{\mathclap{1 \le i_1 < \cdots < i_j \le n}}
  \vol_M(\Delta_{i_1} + \cdots + \Delta_{i_j}).
\end{equation*}
For $n=0$ we agree that $ \MV_{M}=1$.
\end{definition}

For a single polytope $\Delta$ we have that
$\MV_M(\Delta, \dots, \Delta) = n! \vol_M(\Delta)$. The mixed volume
is symmetric and linear in each variable $\Delta_i$ with respect to
the Min\-kows\-ki sum, invariant with respect to linear maps that
preserve the measure $\vol_{M}$, and monotone with respect to the
inclusion of polytopes.  We have that
$\MV_M(\Delta_{1},\dots,\Delta_{n}) \ge 0$, and the equality holds if
and only if there is a subset $I\subset \{1,\dots, n\}$ such that
$\dim(\sum_{i\in I} \Delta_{i}) <\#I$.  If  the
$\Delta_{i}$'s are lattice polytopes, then
$ \MV_M(\Delta_{1},\dots,\Delta_{n})\in \N$.

Given a family of convex piecewise affine functions
$\rho_i\colon \Delta_i\to\R$, $i=1,\dots, n$, with inf-convolution
$\rho=\boxplus_{i=1}^{n}\rho_{i}$ and such that the mixed subdivision
$S(\rho)$ is tight (Definition \ref{def:7}), the mixed volume of
the $\Delta_{i}$'s can be computed as the sum  of  the volumes of the mixed
$n$-cells \cite[Theorem 2.4]{HS95}:
\begin{equation}
  \label{eq:38}
  \MV_{M}(\Delta_{1},\dots, \Delta_{n})= \sum_{\mathclap{C \text{ mixed}}} \vol_{M}(C).
\end{equation}

Analogously, the mixed integral of a family of $n+1$ concave functions
on convex bodies is a polarization of the notion of integral of a
single one.  It was introduced in \cite[\S8]{PS08a}, and is equivalent
to the shadow mixed volume defined in \cite[\S1]{Est08}. Here we
recall its definition and properties, translating them to the convex
setting and restricting to piecewise affine functions on polytopes.
 We refer to \cite[\S4.3]{PS08a} and \cite[\S8]{PS08b} for the
corresponding proofs and more information about this~notion.

\begin{definition}\label{mintegral}
  The {\em mixed integral} of a family of convex piecewise affine
  functions $\rho_i\colon \Delta_i\to\R$, $i=0,\dots, n$, is defined
  as
$$
\MI_M(\rho_0,\ldots,\rho_n)=\sum_{j=0}^n(-1)^{n-j} \hspace{-3mm}\sum_{0\leq
  i_0<\ldots < i_j\leq n} \int_{\Delta_{i_0}+\dots
  +\Delta_{i_j}}\rho_{i_0}\boxplus\ldots\boxplus\rho_{i_j} \dd \vol_{M}.
$$
\end{definition}

For a convex piecewise affine function on a polytope
$\rho \colon \Delta\to \R$ we have that
$\MI_{M}(\rho,\dots, \rho)=(n+1)! \int_{\Delta}\rho \, \dd
\vol_{M}$. The mixed integral is symmetric and additive in each
variable $\rho_{i}$ with respect to the inf-convolution, and monotone.

It is possible to express mixed integrals in terms of mixed
volumes. For $i=0,\dots,n$ choose $\kappa_{i} \in \R_{\ge0}$ with
$\kappa_{i}\ge \rho_{i}(x)$ for all ${x\in\Delta_{i}}$ and consider
the polytope
\begin{equation*}
 \Delta_{i,\rho_{i},\kappa_{i}}=\conv(\gr(\rho_{i}),
  \Delta_{i}\times \{\kappa_{i}\}) \subset M_{\R}\times \R.
\end{equation*}
Then by \cite[Proposition 4.5(d)]{PS08a},
\begin{multline}
  \label{eq:46}
  \MI_{M}(\rho_{0},\dots, \rho_{n})= - \MV_{M\times \Z} (
  \Delta_{0,\rho_{0},\kappa_{0}},\dots, \Delta_{n,\rho_{n},\kappa_{n}}) \\
  +\sum_{i=0}^{n} \kappa_{i} \, \MV_{M}(\Delta_{0},\dots,
  \Delta_{i-1},\Delta_{i+1}, \dots, \Delta_{n}).
\end{multline}

For each $i$, the convex piecewise affine function
$\rho_{i}\colon \Delta_{i}\to \R$ is \emph{lattice} if there are
$\cA_{i}\subset M$ and $\bfnu_{i}\in \Z^{\cA_{i}}$ such that
$\Delta_{i}=\conv(\cA_{i})$ and $\rho_{i}=\vartheta_{\bfnu_{i}}$ as in
\eqref{eq:14}.

\begin{proposition}
  \label{prop:4}
  For $i=0,\dots, n$ let $\rho_{i} \colon \Delta_{i}\to \R$ be a
  lattice convex piecewise affine function. Then $\MI_{M}(\rho_{0},\dots, \rho_{n})\in \Z$.
\end{proposition}

\begin{proof}
  This follows directly from \eqref{eq:46} and the analogous property
  for the mixed volume. 
\end{proof}

\section{Sparse resultants}\label{sec:basic-prop-sparse}

\subsection{Definitions and basic properties}\label{sec:defin-basic-prop}

In this section we recall the basic notations, definitions and
properties of sparse eliminants and resultants from
\cite{DS15}.

We keep the notation of the previous sections. In particular
$M\simeq \Z^{n}$ is a lattice of rank $n \ge 0$ and
$N=M^{\vee} \simeq \Z^{n}$ its dual lattice. Let
\begin{equation*}
  \T_{M}= \Hom(M,\C^{\times }) = N\otimes_{\Z} \C^{\times} \simeq
  (\C^{\times})^{n}
\end{equation*}
be the torus over $\C$ associated to $M$.  Then
$M=\Hom(\T_{M},\C^{\times})$, and for $a\in M$ we denote by
$ \chi^{a} \colon \T_{M}\to \C^{\times}$ the corresponding character
of $\T_{M}$.

For $i=0,\dots, n$ let $\cA_{i}$ be a nonempty finite subset of $M$,
$\Delta_{i}=\conv(\cA_{i})$ the lattice polytope of $M_{\R}$ given by
its convex hull, $\bfu_{i}=\{u_{i,a}\}_{a\in \cA_{i}}$ a set of
$\#\cA_{i}$ variables and
\begin{equation*}
  F_{i}=\sum_{a\in \cA_{i}}u_{i,a} \, \chi^{a}\in \Z[\bfu_{i}][M]
\end{equation*}
the general Laurent polynomial with support $\cA_{i}$, where
$\Z[\bfu_{i}][M] \simeq \Z[\bfu_{i}][x_1^{\pm 1},\dots, x_n^{\pm 1}]$
denotes the group $\Z[\bfu_{i}]$-algebra of $M$. Set for short
\begin{displaymath}
 \bfcA=(\cA_{0},\dots, \cA_{n}), \quad \bfDelta=(\Delta_{0},\dots,
 \Delta_{n}), \quad \bfu=(\bfu_{0},\dots, \bfu_{n}) \and \bfF=(F_{0},\dots, F_{n}).
\end{displaymath}
The \emph{incidence variety} of $\bfF$ is defined as
\begin{displaymath}
  \Omega_{\bfcA}=Z(\bfF) \subset \T_M\times
  \prod_{i=0}^{n}\P(\C^{\cA_{i}}),
\end{displaymath}
that is, the zero set of these Laurent polynomials in
that product space. It is an irreducible algebraic subvariety of
codimension $n+1$ defined over $\Q$. Denote by
$\varpi \colon \T_M\times \prod_{i=0}^{n}\P(\C^{\cA_{i}})\to
\prod_{i=0}^{n}\P(\C^{\cA_{i}})$ the projection onto the second
factor.  The \emph{direct image} of $\Omega_{\bfcA}$ with respect to
$\varpi$ is the Weil divisor of $\prod_{i=0}^{n}\P(\C^{\cA_{i}})$ defined
as
\begin{displaymath}
  \varpi_{*}\Omega_{\bfcA}=
  \begin{cases}
    \deg(\varpi|_{\Omega_{\bfcA}}) \, \overline{\varpi(\Omega_{\bfcA})}
  & \text{ if }
  \overline{\varpi(\Omega_{\bfcA})} \text{ is a hypersurface}, \\
  0 & \text{ otherwise},
\end{cases}
\end{displaymath}
where $\overline{\varpi(\Omega_{\bfcA})}$ is the Zariski closure of
the image of the incidence variety with respect to the projection, and
$\deg(\varpi|_{\Omega_{\bfcA}}) $ is the degree of the restriction of
this map to the incidence variety.

\begin{definition}
  \label{def:16}
  The \emph{sparse resultant}, denoted by $\Res_{\bfcA}$, is defined
  as any primitive polynomial in $\Z[\bfu]$ giving an equation for
  $\varpi_{*}\Omega_{\bfcA}$.  The \emph{sparse eliminant}, denoted by
  $\Elim_{\bfcA}$, is defined as any irreducible polynomial in
  $\Z[\bfu]$ giving an equation for $\ov{\varpi(\Omega_{\bfcA})} $, if
  this a hypersurface, and as 1 otherwise.
\end{definition}

Given a ring $A$ and Laurent polynomials $f_{i}\in A[M]$ with support
contained in $\cA_{i}$ for each $i$, we apply the usual notation
\begin{equation}
  \label{eq:54}
\Elim_{\bfcA}(f_{0},\dots, f_{n}) \and  \Res_{\bfcA}(f_{0},\dots,
f_{n})
\end{equation}
to denote the evaluation at the coefficients of the $f_{i}$'s.

Both the sparse resultant and the sparse eliminant are well-defined up
to the sign, and  are both invariant by translations and
permutations of the supports \cite[Proposition 3.3]{DS15}.  The sparse
eliminant does not depend on the lattice $M$ but the sparse resultant
does, as the following proposition shows.

\begin{proposition}\label{outil}
  Let $\varphi\colon M\to M'$ be a  monomorphism  of lattices of
  rank $n$. Then
  $\Elim_{\varphi(\bfcA)}= \pm
  \Elim_{\bfcA}$ and
  $\Res_{\varphi(\bfcA)}= \pm
  \Res_{\bfcA}^{[M':\varphi(M)]}$.
\end{proposition}

\begin{proof}
  The  monomorphism  $\varphi\colon M\to M'$ induces a finite
  map of degree $[M' : \varphi(M)]$
  \begin{displaymath}
    \varphi^{*}\colon \T_{M'}=\Hom(M',\C^{\times})\longrightarrow \T_{M}=\Hom(M,\C^{\times}).
  \end{displaymath}
  Setting $F_{i}'=\sum_{a\in \cA_{i}}u_{i,a} \, \chi^{\varphi(a)}$ for
  the general Laurent polynomial with support $\varphi(\cA_{i})$ for
  each $i$, the system $F'_0=\dots= F'_n=0$ has a nontrivial solution
  in $\T_{M'}$ if and only if $F_0=\dots= F_n=0$ has a nontrivial
  solution in~$\T_{M}$. Hence $\varphi^{*}$ induces a commutative
  diagram
  \begin{displaymath}
\xymatrix{  \Omega_{\varphi(\bfcA)}\ar[r] \ar[d]^{\varpi'}&
  \Omega_{\bfcA} \ar[d]^{\varpi}\\
    \prod_{i=0}^{n} \P(\C^{\varphi(\cA_{i})})  \ar@{=}[r]& \prod_{i=0}^{n}
\P(\C^{\cA_{i}})}
\end{displaymath}
which implies the stated equality between the sparse eliminants.  From
here we also deduce that $\ov{ \varpi'  (\Omega_{\varphi(\bfcA)})} $ is
not a hypersuperface if and only if this also holds for
$\ov{\varpi(\Omega_{\bfcA})} $, in which case both
$\Res_{\varphi(\bfcA)}$ and $ \Res_{\bfcA}$ are equal to $\pm 1$,
proving the second equality in this case. Otherwise, the
multiplicativity of the degree implies that
\begin{displaymath}
  \deg(\varpi'|_{\Omega(\varphi(\bfcA)})= [M' : \varphi(M)] \,
\deg(\varpi|_{\Omega_\bfcA})
\end{displaymath}
and so
$\varpi'_{*}\Omega_{\varphi(\bfcA)}= [M' : \varphi(M)]
\,\varpi_{*}\Omega_{\bfcA}$, which implies the second equality in this other
case and completes the proof.
\end{proof}

The sparse resultant is homogeneous in each  set of
variables $\bfu_{i}$ of degree \cite[Proposition 3.4]{DS15}:
\begin{equation}
  \label{eq:2}
  \deg_{\bfu_{i}}(\Res_{\bfcA})= \MV_{M}(\Delta_{0},\dots,
  \Delta_{i-1},\Delta_{i+1},\dots, \Delta_{n}), \ i=0,\dots, n.
\end{equation}

Let $\cA\subset M$ be a nonempty finite subset and
$f= \sum_{a\in \cA}\alpha_{a}\, \chi^{ a}\in \C[M]$ a Laurent
polynomial with support contained in $\cA$. For $v \in N_{\R}$ we
respectively set
\begin{equation}
  \label{eq:52}
  \cA^{v}=\cA\cap \conv(\cA)^{v} \and
  \init_{v}(f)= \sum_{ a\in
  \cA^{ v}}\alpha_{ a}\, \chi^{ a}
\end{equation}
for the \emph{restriction of $\cA$ to the face $\conv(\cA)^{v}$} as
defined in \eqref{eq:51} and  the \emph{initial part of $f$ in the
  direction of $v$}.

For $ v\in N\setminus \{0\}$, the \emph{sparse resultant in the
  direction of $ v$}, denoted by
$\Res_{\cA_{1}^{ v },\dots,\cA_{n}^{v}}$, is the sparse resultant
associated to the orthogonal lattice $v^{\bot}\cap M \simeq\Z^{n-1}$
and the supports $\cA_{i}^{ v }$, $i=1,\dots, n$, modulo suitable
translations placing them inside this lattice, see \cite[Definition
4.1]{DS15} for details.  By \cite[Proposition 3.8]{DS15}, this
directional resultant is nontrivial only when $ v $ is the inner
normal to a face of dimension $n-1$ of the Minkowski sum
$\sum_{i=1}^{n}\Delta_i$. In particular, the number of non-trivial
directional sparse resultants of the family of supports $ \bfcA$ is
finite.

The following result is the Poisson formula for the sparse resultant
\cite[Theorem~4.2]{DS15}.
For a subset $B\subset M_{\R}$, its  \emph{support function}
$h_{B}\colon N_{\R}\to \R \cup\{-\infty\}$ is defined~by
\begin{equation*}
  h_{B}(v)=\inf \{\langle v,x\rangle \mid x\in B\}.
\end{equation*}
This generalizes the support function of a convex polyhedron in
\eqref{eq:28}.

\begin{theorem}
\label{poisson}
  For $i=0,\dots, n$ let $f_{i}\in \C[M]$ with support
  contained in $\cA_{i}$ and suppose that
  $\Res_{\cA_{1}^{v}, \dots, \cA_{n}^{v}}(\init_{v}(f_{1}),\dots,
  \init_{v}(f_{n})) \ne 0$ for all $v\in N\setminus\{0\}$. Then
 \begin{equation*}
  \Res_{\bfcA}(f_{0}, f_{1},\dots, f_{n})= \pm
  \prod_{v}\Res_{\cA_{1}^{v}, \dots, \cA_{n}^{v}}(\init_{v}(f_{1}),\dots, \init_{v}(f_{n}))^{-h_{\cA_{0}}(v)}
  \cdot \prod_{p} f_{0}(p)^{m_{p}} ,
\end{equation*}
the first product being over the primitive vectors $v\in N$ and the
second over the solutions $p\in \T_{M}$ of the system of equations
$f_{1}=\dots=f_{n}=0$, where $m_{p}$ denotes the intersection
multiplicity of this system of equations at the point $p$.
\end{theorem}

For a subset $J\subset \{0,\dots, n\}$ put
$ \bfcA_J= (\cA_{i})_{i\in J} $ and $\bfu_J= (\bfu_{i})_{i\in J} $.

\begin{definition}
  \label{def:1}
  The \emph{fundamental subfamily} of $\bfcA$ is the family of
  supports $\bfcA_{J}$ for the minimal subset
  $J\subset \{0,\dots, n\}$ such that $\Res_{\bfcA} \in \Z[\bfu_{J} ]$
  or equivalently, such that $\Elim_{\bfcA} \in \Z[\bfu_{J} ]$.
\end{definition}

For each $i$ set $L_{\cA_{i}}$ for the
sublattice of $M$ generated by the differences of the elements of
$\cA_{i}$. For a subset $J\subset \{0,\dots, n\}$ consider the sum
$L_{\bfcA_J}=\sum_{i\in J} L_{\cA_i}$ and its saturation
$L_{\bfcA_{J}}^{\sat}= (L_{\bfcA_J} \otimes_{\Z} \R)\cap M$.

\begin{remark}
  \label{rem:1}
  By \cite[Corollary 1.1 and Lemma 1.2]{Stu94} or \cite[Proposition
  3.13]{DS15} when $J\ne \emptyset$ the fundamental subfamily
  $\bfcA_{J}$ coincides with the unique essential subfamily of
  $\bfcA$, that is, the unique subfamily such that
  $\rank( L_{\bfcA_{J}}) = \#J-1$ and
  $ \rank( L_{\bfcA_{J'}}) \ge \#J'$ for all $J'\subsetneq J$, whereas
  when $J=\emptyset$, that is when $\Res_{\bfcA}=\pm1$, we have that
  $\bfcA_{J}=\emptyset$ and $\bfcA$ has at least two essential
  subfamilies.
\end{remark}

The sparse eliminant and the sparse resultant are  related by
\begin{equation}
  \label{eq:48}
  \Res_{\bfcA}=\pm \Elim_{\bfcA}^{d_{\bfcA}}
\end{equation}
with $d_{\bfcA}\in \N_{>0}$.

\begin{proposition}
  \label{prop:12}
  Let $\bfcA_{J}$ be the fundamental subfamily of $\bfcA$ and suppose
  that $J\ne \emptyset$. Then $\rank(L_{\bfcA_{J}})= \#J - 1$, and
the  exponent in \eqref{eq:48} can be written as
\begin{equation*}
  d_{\bfcA} = [L_{\bfcA_{J}}^{\sat}:L_{\bfcA_{J} }]\,
  \MV_{M/L_{\bfcA_{J}}^{\sat}}(\{\pi (\Delta_{i})\}_{i\notin J})
\end{equation*}
where $\pi$ is the projection $M\to M/L_{\bfcA_{J}}^{\sat}$.
\end{proposition}

\begin{proof}
The first claim follows from Remark \ref{rem:1}, whereas the second is
\cite[Proposition~3.13]{DS15}.
\end{proof}

Sparse eliminants are particular cases of sparse resultants.

\begin{proposition}
  \label{prop:8}
  Let $\bfcA_{J}$ be the fundamental subfamily of $\bfcA$, suppose
  that $J\ne\emptyset$ and consider $\bfcA_{J}$ as a family of $\#J$
  nonempty finite subsets of the lattice
  $L_{\bfcA_{J}}\simeq \Z^{\#J-1}$.  Then
  $\Elim_{\bfcA}=\pm\Elim_{\bfcA_{J}}=\pm\Res_{\bfcA_{J}}$.
\end{proposition}

\begin{proof}
  The first equality is given by \cite[Proposition 3.11]{DS15},
  whereas the second follows from the equality in \eqref{eq:48} and
  Proposition \ref{prop:12}.
\end{proof}

We also need the following auxiliary result.

\begin{proposition}
  \label{prop:22}
  Let $\wt{\bfcA}=(\wt \cA_{0},\dots, \wt\cA_{n})$ be a further family
  of supports in $M$ such that $\wt \cA_{i}\subset \cA_{i}$ for all
  $i$. Let $\wt{\bfcA}_{J}$ and $\bfcA_{K}$ be the respective
  fundamental subfamilies of supports. Then $J\subset K$.
\end{proposition}

\begin{proof}
  By the degree formula in \eqref{eq:2}, an index
  $j\in \{0,\dots, n\}$ lies in $ K$ if and only if
\begin{displaymath}
  \MV_{M}(\Delta_{0},\dots,
  \Delta_{j-1},\Delta_{j+1},\dots, \Delta_{n})>0.  
\end{displaymath}
The statement follows then from the monotonicity of the mixed volume
with respect to the inclusion of polytopes.
\end{proof}

The notions of sparse eliminant and of sparse resultant include the
classical homogeneous resultant introduced by Macaulay
\cite{Macaulay:sfe}, as we next explain.

\begin{example}
\label{exm:3}
For $\bfd=(d_{0},\dots, d_{n})\in (\N_{>0})^{n+1}$ let $\Res_{\bfd}$
be the homogeneous resultant, giving the condition for a system of
$n+1$ homogeneous polynomials in $n+1$ variables of degrees $\bfd$ to
have a zero in the $n$-dimensional projective space
\cite[\S3.2]{CLO05}.  It coincides, up to the sign, both with the
sparse eliminant and the sparse resultant for the lattice $M=\Z^{n}$
and the family of supports $\bfcA=(\cA_{0},\dots, \cA_{n})$ given~by
  \begin{displaymath}
    \cA_i=\{\bfa  \in\N^n \mid |\bfa| \leq d_i\} ,
  \end{displaymath}
  where for a lattice point $\bfa=(a_{1},\dots, a_{n})\in \Z^{n}$ we
  denote by $|\bfa|=\sum_{i=1}^{n}a_{i}$ its \emph{length}.  Then
  $\Delta_{i}= \{\bfx\in (\R_{\ge 0})^{n} \mid |\bfx|\le d_{i}\}$ for
  each $i$ and we can deduce from the degree formula
  in~\eqref{eq:2}~that
  \begin{displaymath}
\deg_{\bfu_{i}}(\Res_{\bfd})=\prod_{j\ne i} d_{j}, \
    i=0,\dots, n.
  \end{displaymath}
\end{example}

\subsection{Order and initial parts}\label{sec:initial-parts}

In this section we study the different orders and initial parts of the
sparse resultant.

\begin{definition}
  \label{def:13}
  Let $\bfomega\in \R^{\bfcA}$ and let $t$ be a variable. For
  $P\in \C[\bfu] \setminus\{0\}$ set
  \begin{equation}
    \label{eq:49}
      P^{\bfomega}= P( (t^{\omega_{i,a}}\, u_{i,a} )_{i\in \{0,\dots, n\}, a\in \cA_{i}}) \in \C[\bfu][t^{\R}] \setminus \{0\}.
  \end{equation}
The \emph{order} and the \emph{initial part} of $P$ with respect to
$\bfomega$ are the elements $\ord_{\bfomega}(P) \in \R$ and
$\init_{\bfomega}(P)\in \C[\bfu]\setminus\{0\}$ defined by the
equation
\begin{equation}
  \label{eq:42}
  P^{ \bfomega} = ( \init_{\bfomega}(P) + o(1)) \,
  t^{\ord_{\bfomega}(P)},
\end{equation}
where $o(1)$ denotes a sum of terms whose degree in $t$ is strictly
positive.

For a nonzero rational function $P\in \C(\bfu)^{\times}$ written as
$P={P_{1}}/{P_{2}}$ with $P_{i}\in \C[\bfu] \setminus\{0\}$, $i=1,2$,
the \emph{order} and the \emph{initial part} of $P$ with respect to
$\bfomega$ are defined~as
\begin{displaymath}
  \ord_{\bfomega}(P)=  \ord_{\bfomega}(P_{1}) - \ord_{\bfomega}(P_{2}) \and
  \init_{\bfomega}(P)  = \frac{\init_{\bfomega}(P_{1})}{\init_{\bfomega}(P_{2})}.
\end{displaymath}
These notions do not depend on the choice of $P_{1}$ and $P_{2}$  and
the maps
\begin{displaymath}
  \ord_{\bfomega}\colon \C(\bfu)^{\times} \longrightarrow \R \and  \init_{\bfomega}\colon \C(\bfu)^{\times} \longrightarrow \C(\bfu)^{\times}
\end{displaymath}
are group morphisms.  We extend them by setting
$\ord_{\bfomega}(0)=+\infty$ and $\init_{\bfomega}(0)=0$.  The notion
of initial part generalizes the definition in \eqref{eq:52} for
Laurent polynomials.
\end{definition}

As pointed out by Sturmfels, the initial part of the sparse resultant
in a given direction is closely related to the mixed subdivision of
$\Delta$ associated to the convex piecewise affine functions defined
by that direction \cite{Stu94}.

\begin{definition}
  \label{def:12}
  For $\bfomega=(\bfomega_{0},\dots, \bfomega_{n})\in \R^{\bfcA}$ let
  $\vartheta_{\bfomega_{i}}\colon \Delta_{i}\rightarrow \R$,
  $i=0,\dots, n$, and $\Theta_{\bfomega} \colon \Delta\rightarrow \R$
  be the associated convex piecewise affine functions as in
  \eqref{eq:14} and~\eqref{eq:19}.  Let $D$ be an $n$-cell of the
  mixed subdivision $S(\Theta_{\bfomega})$ of $\Delta$ and
  $D_{i} \in S(\vartheta_{\bfomega_{i}})$, $i=0,\dots, n$, its
  components as defined in \eqref{eq:66}.  The \emph{restriction} of
  $\bfcA$ to $D$ is the family of nonempty finite subsets of $M$
  defined as
\begin{displaymath}
  \bfcA_{D}=( \cA_{0}\cap
  D_{0},\dots, \cA_{n}\cap D_{n}).
\end{displaymath}
\end{definition}

The next theorem gives formulae for the order and the initial part of
the sparse resultant. The first part is a reformulation of a result
 by  Philippon and the third author for the Chow weights of a
multiprojective toric variety  \cite[Proposition~4.6]{PS08a}, whereas
the second is a generalization of a result  by  Sturmfels for
sparse eliminants in the case when the fundamental subfamily coincides
with $\bfcA$ \cite[Theorem 4.1]{Stu94}.

\begin{theorem}\label{mt1}
  Let $\bfomega \in \R^{\bfcA}$. Then
  \begin{displaymath}
    \ord_{\bfomega}(\Res_{\bfcA})=\MI_{M}(\vartheta_{\bfomega_{0}},\dots,
    \vartheta_{\bfomega_{n}}) \and
    {\init}_\bfomega(\Res_{\bfcA})=\pm\prod_{\mathclap{D\in
      S(\Theta_{\bfomega})^{n}}}\Res_{\bfcA_{D}}.
  \end{displaymath}
\end{theorem}

Before proving it, we need to establish some auxiliary results. For
$P\in\C[\bfu]$ we denote by $\supp(P)$ its \emph{support}, that is,
the finite subset of $\N^{\bfcA}$ of the exponents of the nonzero
terms of this polynomial.

\begin{lemma}\label{cool}
  For $\bfomega=(\bfomega_{0},\dots, \bfomega_{n})\in\R^{\bfcA}$ there
  is
  $\wt\bfomega=(\wt\bfomega_{0},\dots, \wt\bfomega_{n})\in \Z^{\bfcA}$
  such that
  \begin{enumerate}
  \item \label{item:11}   ${\init}_{\wt\bfomega}(\Res_{\bfcA})={\init}_\bfomega(\Res_{\bfcA})$,
  \item \label{item:12}   $ S(\Theta_{\wt\bfomega})=
    S(\Theta_{\bfomega})$,
  \item \label{item:13} for every $n$-cell of
    $ S(\Theta_{\wt\bfomega})$, its components with respect to the
    families of convex piecewise affine functions
    $\vartheta_{\wt\bfomega_{i}}$, $i=0,\dots, n$,  and
    $\vartheta_{\bfomega_{i}}$, $i=0,\dots, n$,  coincide.
  \end{enumerate}
\end{lemma}

\begin{proof}
  Set $\cS=\supp(\Res_{\bfcA}) \subset\N^{\bfcA}$ and $\cS^{\bfomega}$
  for the subset of $\cS$ of lattice points with minimal scalar
  product with respect to $\bfomega$. A vector
  $\wt\bfomega=(\wt\bfomega_{0},\dots, \wt\bfomega_{n})\in \R^{\bfcA}$
  verifies the condition~\eqref{item:11} if and only if
  \begin{equation}
    \label{eq:21}
 \langle \wt\bfomega,\bfc'-\bfc \rangle =0 \text{ for } \bfc,
      \bfc'\in
      \cS^{\bfomega} \and  \langle \wt\bfomega,\bfc'-\bfc \rangle >0 \text{ for } \bfc\in \cS^{\bfomega}
      \text{ and } \bfc'\in \cS\setminus \cS^{\bfomega}.
  \end{equation}

  With notation as in \eqref{eq:35}, for each
  $D\in S(\Theta_{\bfomega})^{n}$ set $v_{D}\in N_{\R}$ for the
  unique vector such that $ D=\Gamma(\Theta_{\bfomega},v_{D}) $. For
  each $i$ let $D_{i}^{0}\subset \cA_{i}$ be the set of vertices of
  the $i$-th component of $D$.  Then
  $D_{i}=\Gamma(\vartheta_{\wt\bfomega_{i}},v_{D})$ if and only if
\begin{equation}
  \label{eq:24}
    \begin{split}
      & \langle v_{D},a'-a \rangle + \wt\omega_{i,a'}-\wt\omega_{i,a}= 0
      \text{ for }  a,a'\in D_{i}^{0}, \\
      & \langle v_{D},a'-a \rangle + \wt\omega_{i,a'}-\wt\omega_{i,a}\ge 0
      \text{ for }  a\in D_{i}^{0} \text{ and } a'\in (\cA_{i} \cap
      D_{i})\setminus D_{i}^{0}, \\
      & \langle v_{D},a'-a \rangle  + \wt\omega_{i,a'}-\wt\omega_{i,a}>0
      \text{ for }a\in D_{i}^{0} \text{ and } a'\in \cA_{i}\setminus D_{i}.
    \end{split}
  \end{equation}
  If this condition holds, then $D=\Gamma(\Theta_{\wt\bfomega},v_{D})$
  by Proposition \ref{prop:10}\eqref{item:2}.

  Hence if $\wt\bfomega$ satisfies the condition \eqref{eq:21} and
  that in \eqref{eq:24} for all $D\in S(\Theta_{\bfomega})^{n}$,
  then it also verifies \eqref{item:11}, \eqref{item:12}
  and~\eqref{item:13}. These conditions amount to the fact that
  $\wt\bfomega$ lies in the relative interior of a polyhedral cone of
  $\R^{\bfcA}$ defined over $\Z$. This relative interior is nonempty
  as it contains $\bfomega$, and so it also contains a vector in
  $ \Z^{\bfcA}$.
\end{proof}

\begin{lemma}\label{cramer}
  Let $(v,l) \in N\times \Z$ be a primitive lattice vector with
  $l>0$. Let $ (v,l)^{\perp} $ be its orthogonal subspace of
  $M_{\R}\times \R$ and
  \begin{displaymath}
    \varphi\colon (v,l)^{\perp} \cap (M\times \Z) \longrightarrow  M
  \end{displaymath}
  the lattice map defined by $(a,q)\mapsto a$. Then
  $[M : \varphi((v,l)^{\perp} \cap (M\times \Z) )]=l.$
\end{lemma}

\begin{proof}
  Set for short $P=(v,l)^{\perp} \cap (M\times\Z)$, which is a
  sublattice of $M\times\Z$ of rank~$n$. The map
  $\varphi \colon P\to M$ is injective if and only if so is its dual
  $\varphi^{\vee}\colon M^{\vee} \to P^{\vee}$ and if this is the
  case, then
\begin{equation}
  \label{eq:53}
  [M:\varphi(P)]=[P^{\vee} :\varphi^{\vee}(M^{\vee})].
\end{equation}
We have that $M^{\vee}=N$ and
$P^{\vee}\simeq (M\times \Z) / \Z \, (v,l) $. With these
identifications, the dual map
$\varphi^{\vee}\colon N\to (M\times \Z) / \Z \, (v,l) $ writes down as
$\varphi^{\vee}(w)= (w,0)+ \Z\, (v,l)$.  Hence $\varphi^{\vee}$ is
injective because $l>0 $. Moreover, its image is the sublattice
$(M\times l\, \Z) / \Z \, (v,l) $~and~so
\begin{displaymath}
  [P^{\vee} :\varphi^{\vee}(M^{\vee})]=
\#\,  ( M\times \Z / \Z \, (v,l))/ (M\times l\, \Z /  \Z \, (v,l) )= \#
\, \Z/l\Z= l,
\end{displaymath}
which together with \eqref{eq:53} implies the statement.
\end{proof}

\begin{proof}[Proof of Theorem \ref{mt1}]
  By \cite[Proposition 4.5]{PS08b}, the degree of a monomial
  deformation of the sparse resultant can be computed in terms of
  mixed integrals as
  \begin{displaymath}
\deg_{t}(\Res_{\bfcA}^{-\bfomega})= -\MI_{M}(\vartheta_{\bfomega_{0}},\dots,
    \vartheta_{\bfomega_{n}}).
  \end{displaymath}
  Since
  $\ord_{\bfomega}(\Res_{\bfcA})=
  -\deg_{t}(\Res_{\bfcA}^{-\bfomega})$, this gives the first part of
  the statement.

  For the second part, we reduce without loss of generality to the
  case when $\bfomega\in\Z^\bfcA$ thanks to Lemma~\ref{cool}.  Set
  $\bfF^{\bfomega}=(F_{0}^{\bfomega_{0}},\dots, F_{n}^{\bfomega_{n}})$
  with
  \begin{equation*}
    F_{i}^{\bfomega_{i}}= F_{i}( (t^{\omega_{i,a}}\, u_{i,a} )_{a\in
      \cA_{i}}) \in \C[\bfu_{i}][M][t^{\pm1}],
    \ i=0,\dots, n.
  \end{equation*}
  With notation as in \eqref{eq:54} and \eqref{eq:49} we have that
  \begin{equation}
    \label{eq:17}
    \Res_{\bfcA}^{\bfomega} = \Res_{\bfcA}(\bfF^{\bfomega}).
  \end{equation}

  Consider the family of $n+2$ nonempty finite subsets of $M\times \Z$
  given by
  \begin{displaymath}
    \cC=\{(0,0),\,(0,1)\} \and \wh \cA_{i}=\{(a,\omega_{i,a})\}_{ a\in
    \cA_{i}}, \ i=0,\dots, n.
  \end{displaymath}
  Let $ \bfv=\{v_{(0,0)}, v_{(0,1)} \}$ be a set of variables, so that
  the general Laurent polynomial with support $\cC$ is
  $ v_{(0,0)}+v_{(0,1)} \, z $ and that with support $\wh\cA_{i}$ is
  $ F_{i}^{\bfomega}(z)$, the evaluation of $F_{i}^{\bfomega}$ at
  $t=z$ for each $i$.  Set $\wh \bfcA = (\wh \cA_{0}, \dots, \wh \cA_{n})$. By the ``hidden variable'' formula in
  \cite[Proposition~4.7]{DS15}, there is $d_{\bfomega} \in \Z$ such
  that
  $$
  \Res_{\bfcA}(\bfF^{\bfomega})= \pm t^{d_{\bfomega}}
  \Res_{\cC,\wh \bfcA}(z-t, \bfF^{\bfomega}(z)).
  $$

  Thanks to the formula in \eqref{eq:42}, we have that
  \begin{displaymath}
    \init_{\bfomega}(\Res_{\bfcA}) =   \Res_{\cC,\wh \bfcA}(z,\bfF^\bfomega(z))
  \end{displaymath}
  provided this latter polynomial is nonzero. To see this, consider a
  family of Laurent polynomials $f_{i}\in \C[M]$ with
  $\supp(f_{i}) \subset \cA_{i}$, $i=0,\dots, n$, that is sufficiently
  generic and set
  $\bff^{\bfomega}=(f_{0}^{\bfomega_{0}},\dots,
  f_{n}^{\bfomega_{n}})$.  By the invariance of the sparse resultant
  under translation of the first support and the Poisson formula
  (Theorem \ref{poisson}), with notation as therein we have that
\begin{equation}\label{unrav1}
  \begin{split}
  \Res_{\cC,\wh \bfcA}(z,\bff^\bfomega(z))&=
  \Res_{\cC-(0,1),\wh \bfcA}(1,\bff^\bfomega(z)) \\
  & =
\pm
\prod_{(v,l)}{\Res_{\wh {\bfcA}^{(v,l)}}(\init_{(v,l)}(\bff^{\bfomega}(z)))}^{-h_{\cC-(0,1)}((v,l))}\\
& =
\pm
\prod_{(v,l)}\Res_{\wh {\bfcA}^{(v,l)}}(\init_{(v,l)}(\bff^{\bfomega}(z)))^{\max\{0,l\}},
  \end{split}
\end{equation}
the products being over the primitive lattice vectors
$(v,l)\in N\times \Z$, and where $\wh \bfcA^{(v,l)}$ denotes the family of
supports $ (\wh \cA_{0}^{(v,l)}, \dots, \wh \cA_{n}^{(v,l)})$.  The
last equality follows from the fact that
$-h_{\cC-(0,1)}((v,l))=\max\{0,l\}$.  Since this holds for every
choice of $\bff$, we have~that
\begin{displaymath}
  \Res_{\cC,\wh \bfcA}(z,\bfF^\bfomega(z)) =
\pm
\prod_{(v,l)}\Res_{\wh {\bfcA}^{(v,l)}}^{\max\{0,l\}}
\end{displaymath}
because $\init_{(v,l)}(\bfF^{\bfomega}(z))$ is the general Laurent
polynomial with support $\wh {\bfcA}^{(v,l)}$.

Let $(v,l)\in N\times \Z$ primitive with $l>0$. For the linear map
$\varphi\colon (v,l)^{\bot }\cap (M\times \Z) \to M$ induced from the
projection onto the first factor we have that
$[M:\varphi ((v,l)^{\bot }\cap (M\times \Z))]=l$ by Lemma
\ref{cramer}.  For the cell $D=C\big(\Theta_{\bfomega}, \frac{1}{l}\, v\big)$ we also have
that $\varphi(\wh \cA_{i}^{(v,l)})=\cA_{i}\cap D_{i}$ for each $i$.
Proposition \ref{outil} then implies that
\begin{equation}
  \label{eq:50}
\Res_{\wh {\bfcA}^{(v,l)}}^{\max\{0,l\}}=\pm \Res_{\bfcA_{D}}  .
\end{equation}
Since every $n$-cell of $S(\Theta_{\bfomega})^{n}$ appears exactly
once in the product \eqref{unrav1}, this second part then follows from
\eqref{eq:17}, \eqref{unrav1} and \eqref{eq:50}.
\end{proof}

\subsection{Homogeneities and degrees}\label{sec:homogeinities}

The homogeneities of the sparse resultant are of two types: there are
$ \lambda_{i}  \in \Z$, $i=0,\dots, n$, and $\mu\in M$ such that for every
 $\bfc\in \N^{\bfcA}$ in the support of $\Res_{\bfcA}$ we have that
\begin{equation*}
  \sum_{a\in \cA_{i}} c_{i,a}=\lambda_{i}, \  i=0,\dots, n,  \and \sum_{i=0}^{n}\sum_{a\in \cA_{i}} c_{i,a}\, a=\mu,
\end{equation*}
see for instance \cite[Chapter 9, Proposition 1.3]{GKZ94} or
\cite[\S6]{Stu94}. The first type corresponds to the fact that the
sparse resultant is homogeneous in each set of
variables~$\bfu_{i}$. As noted in~\eqref{eq:2}, its partial degree
$\deg_{\bfu_{i}}(\Res_{\bfcA})=\lambda_{i}$ can be computed in terms mixed
volumes.

The second type corresponds to its equivariance with respect to the
action of the torus by translations. For $p\in \T_{M}$ denote by
$\tau_{p} \colon \T_{M}\to \T_{M}$ the translation by this point and
let $\tau_{p}^{*}\bfF=\bfF\circ \tau_{p}$ be the pullback of the
system of general Laurent polynomials $\bfF$ with respect to this
map. The fact that the sparse resultant satisfies this type of
homogeneity is then equivalent to the validity of identity
\begin{equation}
  \label{eq:70}
\Res_{\bfcA}(\tau_{p}^{*}\bfF)= \chi^{\mu}(p) \,
  \Res_{\bfcA}
\end{equation}
 for all $p\in \T_{M}$.

Let $\deg_{M}$ be the grading of the monomials of $\C[\bfu]$ with values
in $M$ defined by
\begin{equation}
  \label{eq:71}
  \deg_{M}(u_{i,a})=a \text{ for } i =0,\dots, n \text{ and } a\in \cA_{i}.
\end{equation}
Then \eqref{eq:70} is also equivalent to the fact that the sparse
resultant is homogeneous with respect to this grading, of degree
$\mu$.  As an application of Theorem \ref{mt1} we will reprove this
type of homogeneity and compute its degree in terms of mixed
integrals.

We first prove an auxiliary lemma. A point $v\in N$ can be seen as a
linear function on $M_{\R}$ and, in particular, can be restricted to
any subset of this linear space.

\begin{lemma}
  \label{lemm:11}
  The function $\mu_{\bfDelta}\colon N\to \Z$  given by
  \begin{displaymath}
\mu_{\bfDelta}(v)= \MI_{M}(v|_{\Delta_{0}},\dots, v|_{\Delta_{n}})
  \end{displaymath}
  is well-defined and linear. Therefore
  $\mu_{\bfDelta} \in M =N^{\vee} $.
\end{lemma}

\begin{proof}
  Let $v\in N$. For each subset $I\subset \{0,\dots, n\}$ we have that
  \begin{displaymath}
   \bigboxplus_{i\in I} v|_{\Delta_{i}}= v|_{\sum_{i\in I} \Delta_{i}} 
  \end{displaymath}
  because $v$ is linear. For $v'\in N$, the definition of the mixed
  integral then implies that
  \begin{displaymath}
    \MI_{M}((v+v')|_{\Delta_{0}},\dots, (v+v')|_{\Delta_{n}}) = \MI_{M}(v|_{\Delta_{0}},\dots, v|_{\Delta_{n}})+  \MI_{M}(v'|_{\Delta_{0}},\dots, v'|_{\Delta_{n}}),
  \end{displaymath}
  which shows that $\mu_{\bfDelta}$ is linear. Moreover,
  $v|_{\Delta_{i}}$ is a lattice convex piecewise affine function and
  so Proposition \ref{prop:4} implies that $\mu_{\bfDelta}(v)\in
  \Z$. The last claim follows from the previous ones.
\end{proof}

\begin{theorem}\label{antt}
  The sparse resultant $\Res_{\bfcA}$ is homogeneous with respect to
 $\deg_{M}$ and
  \begin{displaymath}
    \deg_{M}(\Res_{\bfcA})=\mu_{\bfDelta}\in M.
  \end{displaymath}
\end{theorem}

\begin{proof}
  Let $v\in N$. For the weight $\bfomega\in \Z^{\bfcA}$ defined by
  $\omega_{i,a}=\langle v,a\rangle $ for $i=0, \dots, n $ and
  $a\in \cA_{i}$, we have that
  $\vartheta_{\bfomega_{i}}=v|_{\Delta_{i}}$ for each $i$ and
  $\Theta_{\bfomega}=v|_{\Delta}$. Hence $\Delta$ is the unique
  $n$-cell of $S(\Theta_{\bfomega})$, and from Theorem \ref{mt1} we
  deduce that
  $$
  \ord_\bfomega(\Res_\bfcA)= \MI_{M}(v|_{\Delta_{0}},\dots,
  v|_{\Delta_{n}}) \and
\init_\bfomega(\Res_\bfcA)=\Res_\bfcA.
$$
This implies that for all  $\bfc \in \supp(\Res_{\bfcA})$ we have that
\begin{math}
\langle v,  \sum_{i,a} c_{i,a}\, a \rangle = \langle v,\mu_{\bfDelta}\rangle =\mu_{\bfDelta}(v).
\end{math}
Since this holds for all $v\in N$, we deduce the statement.
\end{proof}

\begin{remark}
  \label{rem:2}
  When $M=\Z^{n}$ we have that
  $\mu_{\bfDelta}= (\mu_{\bfDelta,1},\dots, \mu_{\bfDelta,n})$ with
  \begin{equation}
    \label{eq:30}
    \mu_{\bfDelta, i}= \MI_{M}(x_{i}|_{\Delta_{0}}, \dots,
    x_{i}|_{\Delta_{n}}), \ i=1,\dots, n.
    \end{equation}
In this case $\T_{M}=(\C^{\times})^{n}$, and for a point
    $\bfp=(p_{1},\dots, p_{n})$ in this torus we have that
  \begin{displaymath}
    \Res_{\bfcA}(\tau_{\bfp}^{*} \bfF)=\Res_{\bfcA}(\bfF(p_{1}\, x_{1},\dots, p_{n}\, x_{n}))=
\Big(\prod_{i=1}^{n}p_{i}^{\mu_{\bfDelta,i}}\Big) \, \Res_{\bfcA}.
  \end{displaymath}
  \end{remark}

\begin{example}
  \label{exm:2}
  Let $\Res_{\bfd}$ be the homogeneous resultant corresponding to a
  sequence of degrees $\bfd=(d_{0},\dots, d_{n}) \in(\N_{>0})^{n+1}$
  as in Example \ref{exm:3}.  Since the function $\bfx\mapsto x_{i}$
  is linear, for each $i$ the mixed integral in \eqref{eq:30} can be
  computed as
  \begin{multline*}
\mu_{\bfDelta,i} = \sum_{j=0}^n(-1)^{n-j}\hspace{-10pt} \sum_{{0\leq
  k_0<\ldots < k_j\leq n}} \int_{\Delta_{k_0}+\dots
                       +\Delta_{k_j}} x_{i} \, \dd \, \bfx    \\
= \sum_{j=0}^n(-1)^{n-j}\hspace{10pt} \sum_{\mathclap{0\leq
  k_0<\ldots < k_j\leq n}} \hspace{10pt}  \frac{(d_{k_0}+\dots
  +d_{k_j})^{n+1}}{(n+1)!}=\prod_{l=0}^{n} d_{l},
  \end{multline*}
  where the last equality can be proven with elementary algebra as in
  \cite[Theorem~3.7]{Ewa96}.  This gives the well-known {\it
    isobarism} of the homogeneous resultant, a result that goes back
  to Macaulay \cite[page 11]{Macaulay:atms}.
\end{example}

\subsection{Vanishing coefficients}\label{sec:vanish-coeff}

In this section we apply Theorem \ref{mt1} to obtain a formula for the
evaluation of the sparse resultant by setting some of the coefficients
of the system of Laurent polynomials $\bfF$ to zero.

For $i=0,\dots, n$ let $\wt \cA_{i}\subset \cA_{i}$ be a nonempty
subset, $ \wt \Delta_{i}\subset M_{\R}$ its convex hull,
$\wt \bfu_{i}$ the set of variables corresponding to $\wt \cA_{i}$,
and $\wt F_{i}$ the general Laurent polynomial with support
$\wt \cA_{i}$, which can be obtained from $F_{i}$ by setting
$u_{i,a}=0$ for all $a\notin \wt\cA_{i}$. Set then
\begin{displaymath}
\wt\bfcA= (\wt \cA_{0},\dots, \wt\cA_{n}), \quad \wt \bfu=(\wt
\bfu_{0},\dots , \wt \bfu_{n}) \and \wt \bfF=(\wt F_{0},\dots , \wt F_{n}).
\end{displaymath}
Consider the vector
$\bfomega=(\bfomega_{0},\dots, \bfomega_{n})\in \Z^{\bfcA}$ given, for
$i=0,\dots, n$ and $a\in \cA_{i}$,~by
\begin{equation*}
\omega_{i,a}=
\begin{cases}
0& \text{ if } \bfa\in\wt\cA_i,\\
1& \text{ otherwise,}
\end{cases}
\end{equation*}
and let $\vartheta_{\bfomega_{i}}\colon \Delta_{i}\to \R$,
$i=0,\dots, n$, and $\Theta_{\bfomega}\colon \Delta\to~\R$ be the
associated convex piecewise affine functions as in \eqref{eq:14} and
\eqref{eq:19}.

\begin{theorem}
\label{mt2}
The following conditions are equivalent:
\begin{enumerate}
\item \label{item:19}   $\Res_{\bfcA}(\wt \bfF) \ne 0$,
\item \label{item:20} $\MI_{M}(\vartheta_{\bfomega_{0}},\dots, \vartheta_{\bfomega_{n}})=0$,
\item \label{item:21} for every $n$-cell $D$ of
  $ S(\Theta_{\bfomega})$ we have that
  $ \Res_{\bfcA_{D}}(\wt \bfF)\ne 0$.
\end{enumerate}
If any of these conditions holds, then $\Res_{\bfcA_{D}} \in \Z[\wt
\bfu]$ for all $D\in S(\Theta_{\bfomega})^{n}$ and
\begin{equation}\label{fact0}
  \Res_{\bfcA}(\wt \bfF) =\pm \prod_{\mathclap{D\in S(\Theta_{\bfomega})^{n}}}
  \Res_{\bfcA_{D}}.
 \end{equation}
\end{theorem}

\begin{proof}
By Theorem~\ref{mt1}, following Definition \ref{def:13}, we
  have that
  \begin{equation}
    \label{eq:85}
    \Res^{\bfomega}_{\bfcA}= \pm \Big( \prod_{D} \Res_{\bfcA_{D}}  + o(1) \Big) \,
    t^{\ord_{\bfomega}(\Res_{\bfcA})},
  \end{equation}
the product being over the $n$-cells $D$ of $S(\Theta_{\bfomega})$.

Since $\bfomega\in\N^{\bfcA}$ we have that
$\Res_{\bfcA}^{\bfomega} \in \C[\bfu][t]$ and
$\Res_{\bfcA}^{\bfomega}\big|_{t=0}=\Res_{\bfcA}(\wt \bfF)$. Hence
$\Res_{\bfcA}(\wt \bfF)\ne 0$ if and only if
$\ord_{\bfomega}(\Res_{\bfcA})=0$, and so the expression in
\eqref{eq:85} gives the equivalence between \eqref{item:19} and
\eqref{item:20}.  If any of these conditions holds, then
\begin{equation}
  \label{eq:81}
\Res_{\bfcA}(\wt \bfF)=  \init_{\bfomega}(\Res_{\bfcA}) =\pm \prod_{D}
  \Res_{\bfcA_{D}}.
\end{equation}
Since the left-hand side of \eqref{eq:81} lies in $\Z[\wt \bfu]$ and
the right-hand side is a polynomial, the factors of the latter lie in
$\Z[\wt \bfu]$, proving the last part of the statement and
implying~\eqref{item:21}.

Conversely, suppose that the condition \eqref{item:21}
holds. Evaluating the expression in~\eqref{eq:85} by setting
$u_{i,a}=0$ for $i=0,\dots, n$ and $a\in \cA_{i}\setminus \wt \cA_{i}$
we deduce that
\begin{displaymath}
   \Res_{\bfcA}(\wt \bfF)= \pm \Big( \prod_{D} \Res_{\bfcA_{D}}(\wt \bfF)
 + o(1)\Big) \, t^{\ord_{\bfomega}(\Res_{\bfcA})} ,
\end{displaymath}
which implies \eqref{item:19} and concludes the proof.
\end{proof}

\begin{example}
\label{exm:1} Let $M=\Z$, $\cA_{0}=\cA_{1}=\{0,1\}$ and set
$\bfcA=(\cA_{0},\cA_{1})$. Then
\begin{equation}
  \label{eq:4}
  \Res_{\bfcA}=\det(u_{i,j})_{i,j\in \{0,1\}}=u_{0,0}\, u_{1,1}
  -u_{0,1}\, u_{1,0}.
\end{equation}
Set $\wt \cA_{i}=\{0\}$, $i=0,1$, and let
$\wt \bfF=(u_{0,0}, u_{1,0})$ be the corresponding system of
Laurent polynomials in $\C[t^{\pm1}]$.  With notation as in
Theorem~\ref{mt2}, in this case we have that
$\vartheta_{\bfomega_{i}}(x)= x$ for $i=0,1$ and $x\in [0,1]$ and so
$\Theta_{\bfomega}(x)=x$ for $x\in [0,2]$, as shown in Figure
\ref{fig:4}.  Hence
\begin{displaymath}
  \MI_{\Z}(\vartheta_{\bfomega_{0}}, \vartheta_{\bfomega_{1}})= \int_{0}^{2}\Theta_{\bfomega}(x) \, \text{d}x -
  \int_{0}^{1}\vartheta_{\bfomega_{0}}(x) \, \text{d}x - \int_{0}^{1}\vartheta_{\bfomega_{1}}(x) \, \text{d}x
  =2-\frac12 - \frac12 =1 \ne 0.
\end{displaymath}
This result then tells us that $\Res_{\bfcA}(\wt \bfF)=0$, which can
also be verified from~\eqref{eq:4}.

\begin{figure}[ht]
\begin{tikzpicture}
\draw[->, color= darkgray] (-0.5,0) -- (1.7,0);
\draw[->, color= darkgray] (0,-0.5) -- (0,1.7);
\draw[shift={(0,1)},color=darkgray,thin] (-3pt,0pt) -- (3pt,0pt);
\draw[shift={(1,0)},color=darkgray,thin] (0pt,-3pt) -- (0pt,3pt);
\draw (0,0)--(1,1);
\draw (0.3,1.5) node[right] {$\wt \cA_{0} = \wt \cA_{1} = \{ 0 \}$};
\draw (0.9,0.6) node[right] {$\vartheta_{\bfomega_{0}} = \vartheta_{\bfomega_{1}}$};
\end{tikzpicture} \hspace{1cm}
\begin{tikzpicture}
\draw[->, color= darkgray] (-0.5,0) -- (2.5,0);
\draw[->, color= darkgray] (0,-0.5) -- (0,2.5);
\foreach \x in {1,2} {
\draw[shift={(0,\x)},color=darkgray,thin] (-3pt,0pt) -- (3pt,0pt);
\draw[shift={(\x,0)},color=darkgray,thin] (0pt,-3pt) -- (0pt,3pt);
}
\draw (0,0)--(2,2);
\draw (1.8,1.5) node[right] {$\Theta_{\bfomega}$};
\end{tikzpicture}
  \caption{Convex piecewise affine functions for subsets of the supports}
  \label{fig:4}
  \vspace{-1mm}
\end{figure}
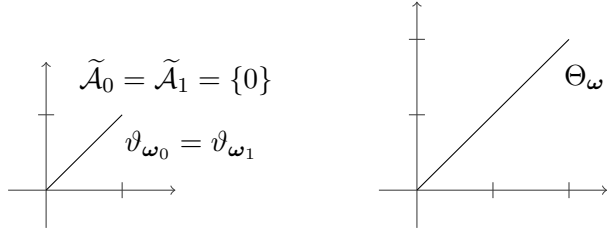

Set also $\wt \cA_{0}=\{0\}$ and $\wt \cA_{1}=\{1\}$, and let
$\wt \bfF=(u_{0,0}, u_{1,1}\, t)$ be the corresponding
system of Laurent polynomials. Then $ \vartheta_{\bfomega_{0}}(x)= x$
and $\vartheta_{\bfomega_{1}}(x)= 1-x$ for $x\in [0,1]$, and so
$\Theta_{\bfomega}(x)= \max\{1-x, x-1\}$ for $x\in [0,2]$, as shown in
Figure \ref{fig:3}.  Hence
\begin{displaymath}
  \MI_{\Z}(\vartheta_{\bfomega_{0}}, \vartheta_{\bfomega_{1}})= 1-\frac12-\frac12=0,
\end{displaymath}
and so Theorem~\ref{mt2} implies that $\Res_{\bfcA}(\wt \bfF)\ne0$.
The mixed subdivision $S(\Theta_{\bfomega})$ has the two $1$-cells
$D=[0,1]$ and $D'=[1,2]$, that decompose as $D=0+[0,1]$ and
$D'=[0,1]+1$. Hence this result also implies that
\begin{displaymath}
  \Res_{\bfcA}(\wt \bfF) = \Res_{\bfcA_{D}} \cdot
  \Res_{\bfcA_{D'}} = u_{0,0}\, u_{1,1},
\end{displaymath}
which can also be verified from \eqref{eq:4}.
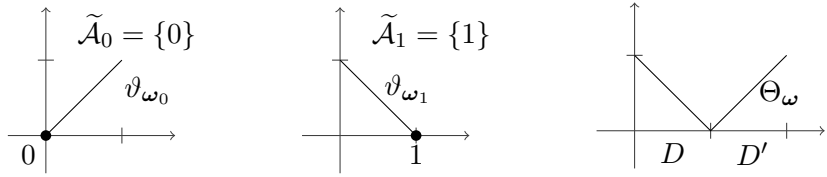
\begin{figure}[ht]
  \begin{tikzpicture}
\draw[->, color= darkgray] (-0.5,0) -- (1.7,0);
\draw[->, color= darkgray] (0,-0.5) -- (0,1.7);
\draw[shift={(0,1)},color=darkgray,thin] (-3pt,0pt) -- (3pt,0pt);
\draw[shift={(1,0)},color=darkgray,thin] (0pt,-3pt) -- (0pt,3pt);
\draw (0,0)--(1,1);
\fill (0,0) circle (2pt) node[below left] {$0$};
\draw (1.2,1) node[above] {$\wt \cA_{0} = \{ 0 \}$};
\draw (0.9,0.6) node[right] {$\vartheta_{\bfomega_{0}}$};
\end{tikzpicture} \hspace{1cm}
\begin{tikzpicture}
\draw[->, color= darkgray] (-0.5,0) -- (1.7,0);
\draw[->, color= darkgray] (0,-0.5) -- (0,1.7);
\draw[shift={(0,1)},color=darkgray,thin] (-3pt,0pt) -- (3pt,0pt);
\draw[shift={(1,0)},color=darkgray,thin] (0pt,-3pt) -- (0pt,3pt);
\draw (0,1)--(1,0);
\fill (1,0) circle (2pt) node[below] {$1$};
\draw  (1.2,1) node[above] {$\wt \cA_{1} = \{ 1 \}$};
\draw (0.45,0.3) node[above right] {$\vartheta_{\bfomega_{1}}$};
\end{tikzpicture} \hspace{1cm}
\begin{tikzpicture}
\draw[->, color= darkgray] (-0.5,0) -- (2.5,0);
\draw[->, color= darkgray] (0,-0.5) -- (0,1.7);
\foreach \x in {1,2} {
\draw[shift={(\x,0)},color=darkgray,thin] (0pt,-3pt) -- (0pt,3pt);}
\draw[shift={(0,1)},color=darkgray,thin] (-3pt,0pt) -- (3pt,0pt);
\draw (0,1)--(1,0)--(2,1);
\draw (1.5,0.5) node[right] {$\Theta_{\bfomega}$};
\draw (0.8,-0.3) node[left] {$D$};
\draw (1.2,-0.3) node[right] {$D'$};
\end{tikzpicture}
  \caption{Convex piecewise affine functions for other subsets}
  \label{fig:3}
\end{figure}
\end{example}

\begin{remark}\label{osts}
  The Minkowski sum $\wt \Delta= \sum_{i=0}^{n}\wt \Delta_{i}$ is the
  cell of the mixed subdivision $S(\Theta_{\bfomega})$ corresponding
  to the vector $0\in N_{\R}$ and its components are the polytopes
  $\wt\Delta_{i}$, $i=0,\dots, n$. Hence
  $\bfcA_{\wt \Delta}=\wt \bfcA$.  We have that either $\wt \Delta$ is
  an $n$-cell of $S(\Theta_{\bfomega})$ or $\Res_{\wt \bfcA}=\pm1$. In
  the presence of any of the equivalent conditions in Theorem
  \ref{mt2}, the factorization in \eqref{fact0} holds and it can be
  alternatively written as
\begin{equation*}
  \Res_{\bfcA}(\wt \bfF) =\pm \Res_{\wt \bfcA} \cdot \prod_{\mathclap{D\ne \wt \Delta}} \Res_{\bfcA_{D}}.
\end{equation*}
When $\Res_{\bfcA}$ is known, this factorization can be useful to
compute the sparse resultant $\Res_{\wt \bfcA}$ as a factor of the
evaluation $\Res_{\bfcA}(\wt \bfF)$.
\end{remark}

The next proposition gives two factorizations for the particular case
when $\tilde{\cA}_i=\cA_i$ for $i=1,\ldots, n$.  The first one follows
directly from Theorem~\ref{mt2}, whereas the second is a consequence
of the Poisson formula.

\begin{proposition} \label{minn} Let $\wt \cA_{0}\subset \cA_{0}$ be a
  nonempty subset and $\wt F_{0}$ the general Laurent polynomial with
  support $\wt \cA_{0}$. With notation as in Theorems \ref{mt2} and~\ref{poisson}, we
  have that
\begin{displaymath}
  \Res_{\cA_0,\cA_{1},\ldots, \cA_n}(\wt{F}_0,F_1,\ldots, F_n)
  =\pm \prod_{\mathclap{D\in S(\Theta_{\bfomega})^{n}}}
  \Res_{\bfcA_{D}}
 = \pm\Res_{\wt{\cA}_0, \cA_1,\ldots,\cA_n}\cdot
  \prod_{v}\Res_{\cA_{1}^{v}, \dots,
  \cA_{n}^{v}}^{h_{\wt \cA_{0}}(v)-h_{\cA_{0}}(v)},
\end{displaymath}
the last product being over the primitive vectors $v\in N$.
\end{proposition}

\begin{proof}
  Let $\wt f_{0}\in \C[M]$ with
  $\supp( \wt f_{0} ) \subset \wt \cA_{0}$ and $f_{i}\in \C[M]$
  with $\supp(f_{i})\subset \cA_{i}$, $i=1,\dots, n$, such that
  $\Res_{\cA_{1}^{v}, \dots, \cA_{n}^{v}}(\init_{v}(f_{1}),\dots,
  \init_{v}(f_{n})) \ne 0$ for all $v\in N\setminus\{0\}$. By~Theorem
  \ref{poisson} we have that
  \begin{multline*}
\Res_{\cA_0,\cA_{1},\ldots, \cA_n}(\wt f_{0}, f_{1},\dots, f_{n})\\= \pm
  \prod_{v}\Res_{\cA_{1}^{v}, \dots, \cA_{n}^{v}}(\init_{v}(f_{1}),\dots, \init_{v}(f_{n}))^{-h_{\cA_{0}}(v)}
  \cdot \prod_{p} \wt f_{0}(p)^{m_{p}} ,
  \end{multline*}
  the first product being over the primitive vectors $v\in N$ and the
  second over the solutions $p\in \T_{M}$ of $f_{1}=\dots= f_{n}=0$,
  where $m_p$ denotes the corresponding intersection multiplicity, and
  similarly
\begin{multline*}
\Res_{\wt{\cA}_0, \cA_1,\ldots,\cA_n}(\wt f_{0}, f_{1},\dots, f_{n})  \\= \pm
  \prod_{v}\Res_{\cA_{1}^{v}, \dots, \cA_{n}^{v}}(\init_{v}(f_{1}),\dots, \init_{v}(f_{n}))^{-h_{\wt\cA_{0}}(v)}
  \cdot \prod_{p} \wt f_{0}(p)^{m_{p}}.
\end{multline*}
Taking the quotient between these two formulae we deduce the second
equality in the statement evaluated at
$\wt f_{0}, f_{1}, \dots, f_{n}$. Since these Laurent polynomials are
generic, we deduce that this equality holds for the general Laurent
polynomials $\wt F_{0},F_{1}, \dots, F_{n}$, as stated.  This also
implies that
$ \Res_{\cA_0,\cA_{1},\ldots, \cA_n}(\wt{F}_0,F_1,\ldots, F_n) \ne 0$,
and so the first equality follows from \eqref{fact0}.
\end{proof}

\begin{remark}
  \label{rem:13}
  In \cite{Min03}, Minimair also studied the factorization of the
  evaluation of sparse resultant at systems of Laurent polynomials
  with smaller supports.  Unfortunately, his result is not consistent,
  since its statement involves the exponent introduced in
  \cite[Remark 3]{Min03}  that, as explained
  in \cite[\S5]{DS15}, is  not well-defined.
\end{remark}

\section{Canny-Emiris matrices}\label{sec:canny-emir-matr}

\subsection{Construction and basic properties}\label{sec:constr-first-prop}

In  \cite{CannyEmiris:easmr, CanEmi:sbasr}  Canny and Emiris presented a
class of matrices whose determinants are nonzero multiples of the
sparse eliminant.  These matrices are associated to some data
including a family of affine functions on polytopes. Shortly
afterwards, this construction was extended by Sturmfels to the convex
piecewise affine case \cite{Stu94}.  Here we recall it and study its
basic properties.

We keep the notations of the previous sections.  In particular,
\begin{itemize}
\item $\bfcA=(\cA_{0},\dots, \cA_{n})$ is a family of $n+1$ supports
  in  the lattice $M$,
\item $ \bfDelta= (\Delta_{0},\dots, \Delta_{n})$ is the family of
  $n+1$ polytopes of the vector space $M_{\R}$ given by the convex hull of these
  supports,
\item $\bfu=(\bfu_{0},\dots, \bfu_{n})$ is the family of $n+1$ sets of
  variables indexed by the elements of the supports,
\item $\bfF=(F_{0},\dots, F_{n}) $ is the associated system of $n+1$ general
  Laurent polynomials.
\end{itemize}

For $i=0,\dots, n$ let $ \rho_{i} \colon \Delta_{i}\rightarrow \R$ be
a convex piecewise affine function on $\Delta_{i}$ defined on
$\cA_{i}$, that is, a convex piecewise affine function of the form
$\rho_{i}=\vartheta_{\bfnu_{i}}$ with $\bfnu_{i}\in \R^{\cA_{i}}$ as
in~\eqref{eq:14}.  Set $\bfrho=(\rho_{0},\dots, \rho_{n})$ and
consider the Minkowski sum and the inf-convolution respectively
defined as
\begin{equation*}
  \Delta=\sum_{i=0}^{n}\Delta_{i}\and
\rho=\bigboxplus_{i=0}^{n}\rho_{i}.
\end{equation*}
We assume that the mixed subdivision $S(\rho)$ of $\Delta$ is
tight (Definition~\ref{def:7}). Choose also a vector
$\delta\in M_{\R}$ such that
\begin{equation}
  \label{eq:20}
( |S(\rho)^{n-1}|+\delta) \cap M = \emptyset,
\end{equation}
where $|S(\rho)^{n-1}|$ denotes the $(n-1)$-skeleton of $S(\rho)$.

The \emph{index set} is the finite set of lattice points
\begin{equation*}
  \cB= (\Delta+\delta) \cap M.
\end{equation*}
Each $b\in \cB$ lies in a unique \emph{translated} $n$-cell of
$S(\rho)$, that is, a polytope of the form $C+\delta$ with
$C\in S(\rho)^{n}$.  Let $C_{i}$, $i=0,\dots, n$, be the components of
this cell, as defined in~\eqref{eq:66}. Since $S(\rho)$ is
tight, there is at least one $i$ such that $\dim(C_{i})=0$, in which case
$C_{i}$ consists of a single lattice point in $ \cA_{i}$ because
$\rho_{i}$ is defined on this support.  Set~then
\begin{displaymath}
 i(b) \in \{0,\dots, n\} \and
 a(b) \in \cA_{i(b)}
\end{displaymath}
for the \emph{largest} of those indexes and the unique lattice point in the
corresponding component, respectively.

\begin{definition}
  \label{def:10}
  The \emph{row content function} associated to $\bfcA$, $\bfrho$ and
  $\delta$ is the function
  $ \rc\colon \cB \rightarrow \bigcup_{i=0}^{n} (\{i\}\times\cA_{i})$
  defined by $\rc(b)=(i(b),a(b))$ for $b\in \cB$.
\end{definition}

Consider the subsets
\begin{equation}
  \label{eq:1}
  \cB_{i}=\{ b\in \cB \mid i(b)=i\}, \ i=0,\dots, n,
\end{equation}
which form a partition of $\cB$.  Set also $\K=\C(\bfu)$ and consider
the finite-dimensional linear subspaces of the group algebra $\K[M]$
defined as
\begin{equation}
  \label{eq:5}
  V_{i}=\sum_{b\in \cB_{i}}  \K \, \chi^{b-a(b)}, \   i=0,\dots, n,  \and
  V=\sum_{b\in \cB}  \K \, \chi^{b}.
\end{equation}

\begin{lemma}
  \label{lemm:9}
Let $i\in \{0,\dots, n\}$ and $b\in \cB_{i}$. Then
  \begin{enumerate}
  \item \label{item:31} for $b'\in \cB_{i}$ we have that
    $b'-a(b')=b-a(b)$ if and only if $b'=b$,
  \item \label{item:32}  $b-a(b)+\cA_{i}\subset \cB$.
  \end{enumerate}
In particular $\dim(V_{i})=\#\cB_{i}$
  and for all $G\in V_{i}$ we have that $G\, F_{i}\in V$.
\end{lemma}

\begin{proof}
  Let $b,b'\in \cB_{i}$ such that $b-a(b)=b'-a(b')$, and denote by $C$
  and $C'$ the $n$-cells of $ S(\rho)$ corresponding to these lattice
  points.  With notation as in ~\eqref{eq:13}, the complementary cells
  $C_{i}^{\cc}$ and $C_{i}^{\prime \cc}$ have both dimension $n$
  and the lattice point $b-a(b)=b'-a(b')$ lies both in
  $ \ri(C_{i}^{\cc})+\delta$ and in $ \ri(C_{i}^{\prime \cc})+\delta$,
  the translates of the relative interiors of these cells. This
  implies that $C_{i}^{\cc}=C_{i}^{\prime \cc}$, and so $C=C'$ by
  Proposition \ref{prop:20}.  We deduce that
  $\{a(b)\}=C_{i}=C_{i}'=\{a(b')\} $ and so $b=b'$, proving
  \eqref{item:31}.

  We also have that
  $ b-a(b)\in C_{i}^{\cc} + \delta \subset \Delta_{i}^{\cc}+\delta$
  and so
  \begin{displaymath}
   b-a(b) +\cA_{i}  \subset (\Delta_{i}^{\cc}+\delta + \Delta_{i})\cap
   M= (\Delta+\delta)\cap M=\cB
  \end{displaymath}
  as stated in \eqref{item:32}.  The  last two  claims follow
  directly from \eqref{item:31} and \eqref{item:32}.
  \end{proof}

Consider the linear map
  $ \Phi_{\bfcA}\colon \K[M]^{n+1} \rightarrow \K[M] $ defined, for
  $\bfG=(G_{0},\dots, G_{n})\in   \K[M]^{n+1}$, by
\begin{displaymath}
  \Phi_{\bfcA}(\bfG)=\sum_{i=0}^{n}G_{i}\, F_{i}.
\end{displaymath}
By Lemma \ref{lemm:9}\eqref{item:32}, if
$\bfG \in \bigoplus_{i=0}^{n} V_{i}$ then $\Phi_{\bfcA}(\bfG)\in V$.

Fixing an order on $\cB$, the right decomposition in \eqref{eq:5}
gives a basis of $V$ indexed by this finite subset. This order induces
an order on each $\cB_{i}$ through the row content function, and
thanks to Lemma \ref{lemm:9}\eqref{item:31} the left decomposition in
\eqref{eq:5} gives a basis for the linear subspace $V_{i}$ indexed by
$\cB_{i}$. The induced basis for the direct sum $\bigoplus_{i} V_{i}$
is then indexed by $\cB$.

For a subset $\cC\subset \cB$ with the induced order, we denote by
$\K^{\cC\times \cC}$ the set of matrices with entries in $\K$ and
whose rows and columns are indexed by the elements of $\cC$.

\begin{definition}
  \label{def:11}
  The \emph{Sylvester map} associated to $\bfcA$, $\bfrho$ and
  $\delta$ is the linear map
  $ \Phi_{\bfcA, \bfrho, \delta}\colon \bigoplus_{i=0}^{n} V_{i}
  \rightarrow V$ given by the restriction of $\Phi_{\bfcA}$ to these
  linear subspaces.  The \emph{Canny-Emiris matrix} associated to
  $\bfcA$, $\bfrho$ and $\delta$, denoted by
  $ \cH_{\bfcA,\bfrho,\delta} \in \K^{\cB\times \cB}$, is the matrix
  of this linear map in terms of row vectors. We set
  $ H_{\bfcA,\bfrho,\delta}=\det(\cH_{\bfcA,\bfrho,\delta}) \in
  \Z[\bfu]$ for the corresponding \emph{Canny-Emiris determinant}.

  Since the vector $\delta$ is fixed throughout our constructions, we
  omit it from the notation, and so this linear map, matrix and
  determinant will be respectively denoted~by
\begin{displaymath}
\Phi_{\bfcA,\bfrho}, \quad  \cH_{\bfcA, \bfrho} \and H_{\bfcA,\bfrho}.
\end{displaymath}
\end{definition}


\begin{remark}
  \label{rem:8}
  For $\bfG\in \bigoplus_{i}V_{i}$ we have~that
  $[\bfG] \cdot \cH_{\bfcA,\bfrho} = [ \Phi_{\bfcA,\bfrho}(\bfG)] $,
  where $[\bfG] $ and $ [\Phi_{\bfcA,\bfrho}(\bfG)] $ denote the row
  vectors of $\bfG $ and of $ \Phi_{\bfcA,\bfrho}(\bfG) $ with respect
  to the bases of $ \bigoplus_{i} V_{i}$ and of $V$ given by the
  decomposition in \eqref{eq:5}.  Hence the row of the Canny-Emiris
  matrix corresponding to an element $b\in \cB$ codifies the
  coefficients of the Laurent polynomial $\chi^{b-a(b)} \,
  F_{i(b)}$. Precisely, the entry corresponding to a pair
  $b,b'\in \cB$ is
\begin{equation*}
  \cH_{\bfcA,\bfrho}[b,b']=
  \begin{cases}
u_{i(b),b'-b+a(b)} & \text{ if } b'-
b+a(b)\in \cA_{i(b)}, \\
0 & \text{ otherwise.}
  \end{cases}
  \end{equation*}
\end{remark}

For a subset $\cC\subset \cB$, we respectively denote by
\begin{equation*}
  \cH_{\bfcA,\bfrho, \cC} =(\cH_{\bfcA,\bfrho}[b,b'])_{b,b' \in \cC}  \in \K^{\cC \times \cC } \and H_{\bfcA,\bfrho, \cC} =\det(\cH_{\bfcA,\cC}) \in \Z[\bfu]
\end{equation*}
the corresponding principal submatrix and minor of the Canny-Emiris
matrix.

\begin{definition}
  \label{def:3}
 The \emph{nonmixed index subset}, denoted by $\cB^{\circ} $, is the
set of elements of $\cB$  lying in the translated $n$-cells of
  $S(\rho)$ that are not $i$-mixed for any~$i$ (Definition
  \ref{def:7}).  We denote by
\begin{equation*}
  \cE_{\bfcA,\bfrho} =\cH_{\bfcA, \bfrho,\cB^{\circ}} \in \K^{\cB^{\circ}\times \cB^{\circ}}\and E_{\bfcA,\bfrho}=H_{\bfcA,\bfrho, \cB^{\circ}} \in \Z[\bfu]
\end{equation*}
the corresponding principal submatrix and minor of
$\cH_{\bfcA,\bfrho}$.
\end{definition}

We next compute  the homogeneities and  corresponding
degrees of the Canny-Emiris determinants and, more generally, of
its principal minors.

\begin{proposition}
  \label{prop:6}
  For $\cC\subset \cB$, the principal minor $H_{\bfcA,\bfrho, \cC}$ is
  homogeneous in each set of variables $\bfu_{i}$ and with respect to
  the grading $\deg_{M}$ defined in \eqref{eq:71}. Moreover,
  \begin{displaymath}
    \deg_{\bfu_{i}}(H_{\bfcA,\bfrho,\cC})=\#(\cB_{i}\cap \cC), \ i=0,\dots, n, \and
    \deg_{M}(H_{\bfcA,\bfrho, \cC})=\sum_{b\in \cC}a(b).
  \end{displaymath}
\end{proposition}

\begin{proof}
  Let $i\in \{0,\dots, n\}$. For $b\in \cC$, the entries in the
  corresponding row of $\cH_{\bfcA,\bfrho,\cC}$ are homogeneous in
  $\bfu_{i}$ of degree $1$ if $i(b)=i$ and of degree $0$
  otherwise. Expanding  $H_{\bfcA,\bfrho, \cC}$ along rows,
  we deduce that it is homogeneous in $\bfu_{i}$ of degree
  $\#(\cB_{i}\cap \cC)$.

  For the claims concerning $\deg_{M}$, first extend this grading to
  $\C[\bfu][M]$  by  declaring that $\deg_{M}(\chi^{a})=a$ for $a\in
  M$. Consider then the matrix
  $ \wt \cH\in \C[\bfu][M]^{\cC\times \cC}$ obtained from
  $\cH_{\bfcA,\bfrho,\cC}$ multiplying by $\chi^{b-a(b)}$ the row
  corresponding to a lattice point $b$, for each $b\in \cC$. By Remark
  \ref{rem:8}, the entry corresponding to a pair $b,b'\in \cC$ is
\begin{equation*}
  \wt \cH[b,b']=
  \begin{cases}
 \chi^{b-a(b)}\, u_{i(b),b'-b+a(b)} & \text{ if } b'-
b+a(b)\in \cA_{i(b)}, \\
0 & \text{ otherwise.}
  \end{cases}
  \end{equation*}
  Hence for $b'\in \cC$, the entries in the corresponding column of
  $\wt \cH$ are homogeneous with respect to $\deg_{M}$ of degree
  $b'$. Expanding the determinant $\wt H=\det(\wt \cH)$ along columns, we deduce that
  it is homogeneous with respect to $\deg_{M}$ of degree
  \begin{displaymath}
\sum_{b'\in \cC}b'.
  \end{displaymath}
  These claims then follow from the fact that
  $ H_{\bfcA,\bfrho, \cC} = \wt H \cdot \prod_{b\in \cC}\chi^{-b+a(b)}
  $.
\end{proof}

\begin{remark}
  The argument for the homogeneity with respect to $\deg_{M}$ of the
  principal minors of a Canny-Emiris matrix is an extension of
  that of Macaulay  for the isobarism of the
  homogeneous resultant in \cite[page 11]{Macaulay:atms}.
\end{remark}

\subsection{Restriction of data and initial
  parts }\label{sec:initial-forms-canny}

In this section we study the interplay between the Canny-Emiris matrix
associated to the data $\bfcA$, $ \bfrho$ and $\delta$, and the mixed
subdivisions of $\Delta$ that are coarser than the tight mixed
subdivision $S(\rho)$. We first introduce the notion of restriction of
data to an $n$-cell of a mixed subdivision and study the compatibility
of the Canny-Emiris construction with this operation.

Let $ \phi_{i} \colon \Delta_{i}\to \R$, $i=0,\dots, n$, be another
family of convex piecewise affine functions, set
$\phi=\bigboxplus_{i=0}^{n}\phi_{i}$ for their inf-convolution,
and let $S(\phi)$ be the associated mixed subdivision of $\Delta$.

Let $D$ be an $n$-cell of $S(\phi)$. Similarly as in Definition
\ref{def:12}, we define the \emph{restriction} of $\bfcA$ and of
$\bfrho$ to $D$ as
\begin{displaymath}
 \bfcA_{D}=( \cA_{0}\cap D_{0},\dots, \cA_{n}\cap D_{n}) \and  \bfrho_{D}=(\rho_{0}|_{D_{0}},\dots, \rho_{n}|_{D_{n}}).
\end{displaymath}

We suppose that $S(\phi) \preceq S(\rho)$ for the rest of this
section. We are not assuming that $S(\phi)$ is tight and in the sequel,
the considered row content function is the one induced by the family
$\bfrho$.

\begin{proposition}
\label{prop:17}
Let $D$ be an $n$-cell of $S( \phi)$. Then
  \begin{enumerate}
  \item \label{item:4} $\rho_{i}|_{D_{i}}$ is a convex piecewise
    affine function on $D_{i}$ defined on $\cA_{i}\cap D_{i}$ for each
    $i$,
    \item \label{item:6}   $\bigboxplus_{i=0}^{n}\rho_{i}|_{D_{i}}=\rho|_{D} $,
    \item \label{item:7} the mixed subdivision
      $S\big(\bigboxplus_{i=0}^{n}\rho_{i}|_{D_{i}}\big)$ of $D$ is
      tight,
    \item \label{item:8} the vector $\delta\in M_{\R}$ is generic with
      respect to $S\big(\bigboxplus_{i=0}^{n}\rho_{i}|_{D_{i}}\big)$~in~the~sense~of~\eqref{eq:20}.
  \end{enumerate}
\end{proposition}

\begin{proof}
  Clearly, the restriction $\rho_{i}|_{D_{i}}$ is a convex piecewise
  affine function on $D_{i}$. Since $S(\phi)\preceq S(\rho)$ we
  have that $D_{i}$ is a union of $n$-cells of $S(\rho_{i})$. Hence
  $\rho_{i}|_{D_{i}}$ is defined on the set of vertices of these
  $n$-cells and so on $\cA_{i}\cap D_{i}$, which proves
  \eqref{item:4}.

  For \eqref{item:6}, note that for $x\in D$ we have that
  $ (\bigboxplus_{i}\rho_{i}|_{D_{i}})(x)$ (respectively
  $\rho|_{D}(x)$) is defined as the infimum of the sum
  \begin{equation}
    \label{eq:29}
    \sum_{i=0}^{n} \rho_{i}(x_{i})
  \end{equation}
  with $x_{i} \in D_{i}$ (respectively $x_{i} \in \Delta_{i}$) for all
  $i$ such that $\sum_{i=0}^{n}x_{i}=x$.  Let $C\in S(\rho)$ such that
  $x\in C$ and $C\subset D$.  By Proposition
  \ref{prop:10}\eqref{item:5}, the infimum of the sum in~\eqref{eq:29}
  with $x_{i}\in \Delta_{i}$, $i=0,\dots, n$, such that
  $\sum_{i=0}^{n}x_{i}=x$ is attained when $x_{i}\in C_{i}$ for all
  $i$. Since $S(\phi) \preceq S(\rho)$, we have that
  $C_{i}\subset D_{i}$ and so $x_{i}\in D_{i}$ for all $i$. This
  implies that $ (\bigboxplus_{i}\rho_{i}|_{D_{i}})(x) =\rho|_{D}(x) $
  and so $\bigboxplus_{i=0}^{n}\rho_{i}|_{D_{i}}=\rho|_{D} $, as
  stated.

  The statements in  \eqref{item:7} and \eqref{item:8}  follow
   directly  from that in~\eqref{item:6}.
\end{proof}

By Proposition \ref{prop:17}, for $D\in S(\phi)^{n}$ the data
$( \bfcA_{ D}, \bfrho_{D}, \delta)$ satisfies the hypothesis in
Definition~\ref{def:11}, and  so we can consider its corresponding
Sylvester map, Canny-Emiris matrix, and determinant.  To set up the
 notation,  for
$i=0,\dots, n$ consider  the set of variables
$\bfu_{D,i}= \{u_{i,a}\}_{a\in \cA_{i}\cap D_{i}} $ and the general
Laurent polynomial with support $\cA_{i}\cap D_{i}$ defined as
  \begin{displaymath}
    F_{D,i}=\sum_{a\in \cA_{i}\cap D_{i}} u_{i,a} \, \chi^{a} \in \C[\bfu_{D,i}][M].
  \end{displaymath}
  Let $ \bfu_{D}=(\bfu_{D,0},\dots, \bfu_{D,n})$ and
  $\K_{ D}=\C(\bfu_{D})$.  Set then
  \begin{equation}
    \label{eq:73}
  \cB_{D}=\cB \cap (D+\delta)  \and
  \cB_{D,i}=\cB_{i}\cap (D+\delta), \ i=0, \dots, n,
\end{equation}
and consider the
  linear subspaces of $\K_{D}[M]$ defined~as
\begin{equation}
  \label{eq:34}
  V_{D,i}= \sum_{b\in  \cB_{D,i}} \K_{D}\, \chi^{b-a(b)}, \ i=0,\dots,
  n, \and V_{D}=\sum_{b\in \cB_{D}} \K_{D} \, \chi^{b}.
\end{equation}
Then the corresponding Sylvester map
$ \Phi_{\bfcA_{ D},\bfrho_{D}}\colon \bigoplus_{i=0}^{n} V_{D,i}
\rightarrow V_{D}$ is defined by
\begin{displaymath}
\Phi_{\bfcA_{ D},\bfrho_{D}}( \bfG)= \sum_{i=0}^{n} G_{i} \, F_{D,i},
\end{displaymath}
the Canny-Emiris matrix
$ \cH_{\bfcA_{ D},\bfrho_{D}} \in \K_{ D}^{\cB_{D}\times \cB_{D}}$ is
the matrix of this linear map with respect to the bases of
$\bigoplus_{i}V_{D,i}$ and of $V_{D}$ given by the decomposition in
\eqref{eq:34}, and $H_{\bfcA_{ D},\bfrho_{D}} \in \Z[\bfu_{ D}]$ is
its determinant.

For  $\cC\subset \cB$ let
$\cC_{D}=\cC\cap (D+\delta) $. This is a subset of $\cB_{D}$, and so
we can consider the corresponding
principal submatrix and minor of $ \cH_{\bfcA_{ D},\bfrho_{D}}$,
respectively denoted by
\begin{equation*}
  \cH_{\bfcA_{ D},\bfrho_{D},\cC_{D}}\in \K_{ D}^{\cC_D\times
    \cC_D} \and H_{\bfcA_{ D},\bfrho_{D},\cC_{D}}\in \Z[\bfu_{ D}].
\end{equation*}

The next result shows that the Canny-Emiris matrix of
$(\bfcA_{D},\bfrho_{D},\delta)$ coincides with a principal
submatrix of  the  evaluation of the Canny-Emiris matrix of
$(\bfcA,\bfrho,\delta)$ setting to zero the coefficients which are not
in $\bfcA_{D}$.

\begin{proposition}
  \label{prop:7}
  The matrix $  \cH_{\bfcA_{ D},\bfrho_{D},\cC_{D}}$ is the evaluation of the
  principal submatrix $\cH_{\bfcA, \bfrho, \cC_{D}}$ by setting
  $u_{i,a} =0$ for $i=0,\dots, n$ and $a\in \cA_{i}\setminus D_{i}$.
\end{proposition}

\begin{proof}
   By Proposition \ref{prop:17}, the row content function associated to
  the restricted data $(\bfcA_{D}, \bfrho_{D}, \delta)$ coincides
  with that of $(\bfcA, \bfrho, \delta)$ restricted to the index set
  $\cB_{D}$.

  For each $i$, the general Laurent polynomial $F_{D,i}$ is the
  evaluation of $F_{i}$ setting $u_{i,a} =0$ for
  $a\in \cA_{i}\setminus D_{i}$. Hence the Sylvester map
  $\Phi_{\bfcA_{ D},\bfrho_{D}}$ is the restriction of
  $\Phi_{\bfcA,\bfrho}$ to the linear subspace
  $\bigoplus_{i=0}^{n} V_{D,i}$ composed with the evaluation that sets
  $u_{i,a} =0$ for all $i$ and $a\in \cA_{i}\setminus D_{i}$.

  This implies the statement when $\cC= \cB$. The case of an arbitrary
  subset $\cC\subset \cB$ follows from this one by considering the
  corresponding principal submatrices.
\end{proof}

Next we turn to the study of the orders and initial parts of the
Canny-Emiris determinant and, more generally, of its principal minors.

\begin{theorem}
  \label{thm:2}
Set $\bfomega=(\phi_{i}(a))_{i,a}\in \R^{\bfcA}$
  and let $\cC\subset \cB$. Then
  \begin{displaymath}
    \ord_{\bfomega}(H_{\bfcA,\bfrho, \cC} )=\sum_{b\in \cC } \phi_{i(b)}(a(b))
    \and     \init_{\bfomega}(H_{\bfcA,\bfrho, \cC} ) =   \prod_{\mathclap{D\in S(\phi)^{n}}}
    H_{\bfcA_{ D},\bfrho_{D}, \cC_{D}}.
  \end{displaymath}
\end{theorem}

Before proving the theorem, we will establish some necessary results.
The next lemma is a wide generalization of \cite[Lemma
4.5]{CanEmi:sbasr} and it plays a key role in the proof of Theorem
\ref{thm:2}.

\begin{lemma}
  \label{lemm:3}
  Let $b , b'\in \cB$ such that $b'\in b-a(b)+\cA_{i(b)}$ and set
  $a'=b'-b+a(b)\in \cA_{i(b)}$.  Then
  \begin{equation}
    \label{eq:9}
\phi (b'-\delta) \le   \phi  (b-\delta) - \phi_{i(b)}(a(b))+\phi_{i(b)}(a')
  \end{equation}
  and the equality holds if and only if there is
  $D\in S(\phi)^{n} $ with $b,b'\in D+\delta$ and
  $a'\in D_{i(b)}$.
\end{lemma}

\begin{proof}
With notation as in \eqref{eq:13}, we  have that $b-\delta-a(b) \in \Delta_{i(b)}^{\cc}$ and
  $a'\in \Delta_{i(b)}$. Since $\phi=\phi_{i(b)}^{\cc}\boxplus
  \phi_{i(b)}$, this implies that
  \begin{equation}
    \label{eq:11}
    \phi(b'-\delta) \le \phi_{i(b)}^{\cc} (b-\delta-a(b)) +
    \phi_{i(b)} (a') .
    \end{equation}
    Let $C\in S(\rho)^{n}$ such that $b\in C+\delta$ and
    $D\in S(\phi)^{n}$ with $C\subset D$.  Then
    $b-a(b)-\delta\in C_{i(b)}^{\cc} $ and $a(b)\in C_{i(b)}$. Since
    $S(\rho)\succeq S(\phi)$, we have that
    $ C_{i(b)}^{\cc} \subset D_{i(b)}^{\cc}$ and
    $ C_{i(b)} \subset D_{i(b)}$ and~so
    \begin{equation}
      \label{eq:31}
b-\delta \in D, \quad b-\delta -a(b) \in D_{i(b)}^{\cc}  \and a(b) \in D_{i(b)}.
    \end{equation}
Proposition \ref{prop:10}\eqref{item:5} then implies that
    $\phi (b-\delta)= \phi_{i(b)}^{\cc}
    (b-\delta-a(b))+\phi_{i(b)}(a(b))$. The inequality in
    \eqref{eq:9} follows from this together with \eqref{eq:11}.

    Now if $b'\in D+\delta$ and $a'\in D_{i(b)}$ then Proposition
    \ref{prop:10}\eqref{item:5} together with \eqref{eq:31} implies
    that the inequality in \eqref{eq:11} is an equality, and so is
    \eqref{eq:9}.

    Conversely suppose that \eqref{eq:9} is an
    equality or equivalently, that this is the case for
    \eqref{eq:11}. Let $D'\in S(\phi)^{n}$ such that
    $b'\in D'+\delta$.  Applying again Proposition
    \ref{prop:10}\eqref{item:5},
\begin{displaymath}
 b-\delta -a(b) \in D_{i(b)}^{\prime\cc} \and
a'\in D'_{i(b)}.
\end{displaymath}
Since $S(\rho)$ is tight we have that
$\dim(C_{i(b)}^{\cc})=n-\dim(C_{i(b)})=n$, and since
$b-\delta\in \ri(C)$ we also have that
$ b-\delta -a(b) \in \ri(C_{i(b)}^{\cc})$. Hence
$ \ri(C_{i(b)}^{\cc}) \subset\ri(D^{\prime\cc}_{i(b)})$ and so
$ b-\delta -a(b) \in \ri(D_{i(b)}^{\prime\cc})$.  Using \eqref{eq:31}
we deduce that the $n$-cells $D_{i(b)}^{\prime\cc}$ and
$D_{i(b)}^{\cc}$ coincide. Proposition \ref{prop:20} then implies that
$D^{\prime}=D$, completing the proof.
\end{proof}

  \begin{corollary}
    \label{cor:1}
    Let $b , b'\in \cB$ such that $b'\in b-a(b)+\cA_{i(b)}$ and set
    $a'=b'-b+a(b)\in \cA_{i(b)}$.  Then
  \begin{equation*}
\rho (b'-\delta) \le   \rho  (b-\delta) - \rho_{i(b)}(a(b)) + \rho_{i(b)}(a')
  \end{equation*}
  and the equality holds if and only if $b'=b$.
  \end{corollary}

  The next result generalizes \cite[Theorem 6.4]{CanEmi:sbasr} which
  is stated for the case when the $\rho_{i}$'s are affine, the
  fundamental subfamily of supports coincides with $\bfcA$ and the
  lattice $L_{\bfcA}$ coincides with~$M$. The proof follows
  \emph{mutatis mutandis} the scheme in \cite[Theorem 3.1]{Stu94} and
  \cite[Theorem 6.4]{CanEmi:sbasr}.
  
  \begin{proposition}
    \label{prop:3}
    Let $\bfomega=(\bfomega_{0}, \dots, \bfomega_{n})\in \R^{\bfcA}$ such that
    $\vartheta_{\bfomega_{i}}=\rho_{i}$ for all $i$ and
    $\cC\subset \cB$. Then
  \begin{displaymath}
    \ord_{\bfomega}(H_{\bfcA,\bfrho, \cC})=\sum_{b\in \cC}  \omega_{i(b),a(b)}
    \and     \init_{\bfomega}(H_{\bfcA,\bfrho, \cC}) =   \prod_{b\in \cC}
    u_{i(b), a(b)}.
  \end{displaymath}
  In particular $H_{\bfcA,\bfrho, \cC}\ne 0$.
  \end{proposition}

  \begin{proof}
Set
    $ \cH_{\bfcA,\bfrho,\cC}^{\bfomega}= \cH_{\bfcA,\bfrho,\cC}( (t^{\omega_{i,a}}\, u_{i,a}
    )_{i\in \{0,\dots, n\}, a\in \cA_{i}}) \in \K(t)^{\cC \times\cC }$
    and let $\wt \cH$ be the matrix obtained from it
    multiplying by $t^{\rho(b-\delta) - \rho_{i(b)}(a(b))}$ the row
    corresponding to a lattice point $b$, for each $b\in \cC$.  The
    entry corresponding to a pair $b,b'\in \cC$~is
  \begin{displaymath}
    \wt \cH[b,b']=
    \begin{cases}
      t^{\rho(b-\delta)-   \rho_{i(b)}(a(b))   +\omega_{i(b),a'}}
       \, u_{i(b),a' } &
      \text{ if } a'\in \cA_{i}, \\
      0 & \text{ otherwise},
    \end{cases}
  \end{displaymath}
  with $a'=b'-b+a(b)$. For $b'\in \cC $ we have that
  $ \rho_{i(b)}(a') \le \omega_{i(b),a'}$ by the definition of this
  piecewise affine function. Moreover, let $C \in S(\rho)^{n}$ such
  that $b\in C+\delta$. Then $C_{i(b)}=\{a(b)\}$, and since
  $\rho_{i(b)}= \vartheta_{\bfomega_{i(b)}}$ this implies that
  $\rho_{i(b)}(a(b))= \omega_{i(b),a(b)}$. By Corollary \ref{cor:1}
\begin{displaymath}
  \rho (b'-\delta) \le   \rho(b-\delta)-   \rho_{i(b)}(a(b))   +\omega_{i(b),a'}
\end{displaymath}
and the equality holds if and only if $b'=b$. Hence for $b'\in \cB$
the entry in the corresponding column of $\wt \cH$ for $b\in \cC$ is
of order at least $\rho(b'-\delta)$, and this value is only attained
when $b=b'$. We have that
$ \wt \cH[b,b]= u_{i(b),a(b) } \,t^{\rho(b-\delta)}$ and so
 \begin{multline*}
   H_{\bfcA,\bfrho, \cC}^{\bfomega}  = \det(\cH_{\bfcA,\bfrho,
     \cC}^{\bfomega})  =
   \det \big(\wt \cH \big)  \cdot
   \prod_{b\in \cC} t^{-\rho(b-\delta)+ \rho_{i(b)}(a(b)) }
\\   =
   \Big(  \prod_{b\in \cC} u_{i(b), a(b) } + o(1)\Big) \,
   t^{\sum_{b\in \cC} \omega_{i(b), a(b)}},
\end{multline*}
proving the statement.
\end{proof}

\begin{proof}[Proof of Theorem \ref{thm:2}]
  This result can be proven similarly as it was done for Proposition
  \ref{prop:3}, by considering the matrix
  $ \cH_{\bfcA,\bfrho,\cC}^{\bfomega}= \cH_{\bfcA,\bfrho,\cC} (
  (t^{\omega_{i,a}}\, u_{i,a} )_{ i, a}) \in \K(t)^{\cC \times\cC }$
  and the modified matrix $\wt \cH$ obtained multiplying by
  $t^{\phi(b-\delta) - \phi_{i(b)}(a(b))}$ the row of
  $ \cH_{\bfcA,\bfrho,\cC}^{\bfomega}$ corresponding to a lattice
  point $b$, for each $b\in \cC$.

  Let $D\in S(\phi)^{n}$. By Lemma \ref{lemm:3}, the lowest order in
  $t$ in the column of $\wt \cH$ corresponding to a lattice point
  $b' \in\cC\cap (D+\delta) $ is $\phi(b'-\delta)$, and it is attained
  exactly when $b\in D+\delta$ and $a'=b'-b+a(b) \in D_{i(b)}$.  Hence
  the matrix extracted from $\cH_{\bfcA,\bfrho, \cC}$ by keeping only
  these entries of minimal order in each column is block diagonal,
  with blocks corresponding to the $n$-cells of $S(\phi)$. Moreover,
  the block corresponding to an $n$-cell $D$ coincides with
  $\cH_{\bfcA_{D},\bfrho_{D}, \cC_{D}}$. Hence
\begin{equation*}
H_{\bfcA,\bfrho, \cC}^{\bfomega} = \det(\wt \cH)  \cdot \prod_{b\in
  \cC} t^{-\phi(b-\delta) + \phi_{i(b)}(a(b))}
 =
\Big( \quad  \prod_{\mathclap{D\in S(\phi)^{n}}}
H_{\bfcA_{D},\bfrho_{D}, \cC_{D}} + o(1)\Big) \,
t^{\sum_{b\in \cC } \phi_{i(b)}(a(b))}.
\end{equation*}
By Proposition \ref{prop:3}, all the
$H_{\bfcA_{D},\bfrho_{D}, \cC_{D}}$'s are nonzero, which completes the
proof.
\end{proof}

\subsection{Divisibility properties} \label{sec:divis-prop}

An important feature of Canny-Emiris determinants is that they provide
nonzero multiples of the sparse eliminant. The next proposition
generalizes \cite[Theorem 6.2]{CanEmi:sbasr} and \cite[Theorem
3.1]{Stu94}, which are stated for the case when the fundamental
subfamily of supports coincides with $\bfcA$.

\begin{proposition}
  \label{prop:5}
  $\Elim_{\bfcA} \mid H_{\bfcA,\bfrho}$ in $\Z[\bfu]$.
\end{proposition}

To prove it, we need the following lemma giving a formula for the
right multiplication of a Canny-Emiris matrix by column vectors of a
certain type.  For a point $p\in \T_{M}$ consider the vectors
\begin{equation}
  \label{eq:33}
   \zeta_{p}\in \C^{\cB} \and \eta_{p,i}\in \C^{\cB}, \ i=0,\dots, n,
\end{equation}
respectively defined for $b\in \cB$ by $ \zeta_{p,b}=\chi^{b}(p) $,
and by $\eta_{p,i,b}=\chi^{b-a(b)} (p) $ if $ b\in \cB_{i}$ and by
$\eta_{p,i,b}=0$ otherwise, for the subset $\cB_{i}$ defined in
\eqref{eq:1}.

\begin{lemma}
    \label{lemm:8}
    For $p\in \T_{M}$ we have that
    $ \cH_{\bfcA,\bfrho}\cdot \zeta_{p}^{\transposed} = \sum_{i=0}^{n} F_{i}(p) \,
    \eta_{p,i}^{\transposed}$.
\end{lemma}

\begin{proof}
  In terms of the dual basis of $V$, right multiplication of a row
  vector by $\zeta_{p}^{\transposed}$ corresponds to
  the linear functional $\ev_{p}\colon V \to \R $ defined by
  $G\mapsto G(p)$.  In terms of the dual basis of
  $\bigoplus_{j}V_{j}$, right multiplication by
  $\eta_{p,i}^{\transposed}$ corresponds to the linear functional
  $\ev_{p,i}\colon \bigoplus_{j }V_{j} \to \R$ defined by
  $(G_{0},\dots, G_{n})\mapsto G_{i}(p)$.  With these identifications,
  for $\bfG \in\bigoplus_{j }V_{j}$ we have that
  \begin{multline*}
    [\bfG]\cdot \cH_{\bfcA,\bfrho}\cdot \zeta_{p}^{\transposed}  =
    \ev_{p}( \Phi_{\bfcA,\bfrho}(\bfG))
    = \sum_{i=0}^{n} G_{i}(p)\, F_{i}(p) \\
    =
\sum_{i=0}^{n}  \ev_{p,i}(\bfG) \, F_{i}(p)
=    [\bfG]\cdot  \Big(\sum_{i=0}^{n} F_{i}(p) \,  \eta_{p,i}^{\transposed}\Big).
  \end{multline*}
  The lemma follows from the fact that this equality is valid for
  every $\bfG$.
\end{proof}

\begin{proof}[Proof of Proposition \ref{prop:5}]
  If $\Elim_{\bfcA}=\pm1$ then the statement is trivial. Else, by
  Definition~\ref{def:16} we have that $\Elim_{\bfcA} \in \Z[\bfu]$ is
  irreducible and the points
  $\bfu=(\bfu_{0},\dots, \bfu_{n})\in \C^{\bfcA}$ such that there
  exists $p\in \T_{M}(\C)$ with
  $F_{0}(\bfu_{0},p)=\dots= F_{n}(\bfu_{n},p)=0 $ form a dense subset
  of the hypersurface
  $Z(\Elim_{\bfcA})\subset \prod_{i}\P(\C^{\cA_{i}})$.  By Lemma
  \ref{lemm:8}, for all these points we have that
  $ \cH_{\bfcA,\bfrho}(\bfu)\cdot \zeta_{p}^{\transposed}=0 $ and so
  $\ker(\cH_{\bfcA,\bfrho}(\bfu))\ne 0$. Hence
  $H_{\bfcA,\bfrho}(\bfu)=0$, which implies that
  $\Elim_{\bfcA} \mid H_{\bfcA,\bfrho}$ in $\Z[\bfu]$, as stated.
\end{proof}

The next result  strengthens  Proposition \ref{prop:5} by
showing that the Canny-Emiris determinant is a multiple of the sparse
resultant and not just of the sparse eliminant, under a restrictive
hypothesis which nevertheless is sufficiently general for our
purposes.

\begin{proposition}
  \label{prop:14}
  Let $i\in \{0,\dots, n\}$ such that $\cB_{i}$ is contained in the
  union of the translated $i$-mixed $n$-cells of $S(\rho)$. Then
  \begin{displaymath}
\frac{H_{\bfcA,\bfrho}}{\Res_{\bfcA}} \in \Q(\bfu_{0},\dots,
\bfu_{i-1},\bfu_{i+1},\dots, \bfu_{n}).
  \end{displaymath}
  Moreover, if $\cB_{i}\ne \emptyset$ then
  ${H_{\bfcA,\bfrho}}/{\Res_{\bfcA}} \in \Z[\bfu_{0},\dots,
  \bfu_{i-1},\bfu_{i+1},\dots, \bfu_{n}]$.
\end{proposition}

To prove it we need some further lemmas. Set for short
\begin{equation}
  \label{eq:18}
 m_{i}= \MV_{M}(\Delta_{0},\dots, \Delta_{i-1}, \Delta_{i+1},\dots,
\Delta_{n}), \ i=0,\dots, n.
\end{equation}

\begin{lemma}
  \label{lemm:10}
  For each $i$, the function $\cB_{i}\to M$ defined by
  $ b\mapsto b-a(b)$ gives a bijection between
  \begin{enumerate}
  \item \label{item:37} the set of lattice points of $\cB_{i}$
    lying in translated $i$-mixed $n$-cells of $S(\rho)$,
  \item \label{item:38} the set of lattice points of
    $\Delta_{i}^{\cc}+\delta$ lying in translated mixed $n$-cells of
    $S(\rho_{i}^{\cc})$.
  \end{enumerate}
  The cardinality of both sets is equal to~$m_{i}$.
\end{lemma}

\begin{proof}
  Denote by $\cC_{i}$ and $\cC_{i}'$ the finite subsets of $M$ defined
  in \eqref{item:37} and in \eqref{item:38}, respectively. For
  $b\in \cC_{i}$ let $C$ be an $i$-mixed $n$-cell of $S(\rho)$ with
  $b\in C+\delta$. Then $C_{i}^{\cc} $ is a mixed $n$-cell of
  $S(\rho_{i}^{\cc})$ and
  $b-a(b)\in C_{i}^{\cc}+\delta \subset \Delta_{i}^{\cc}+\delta$, and
  so $ b-a(b) \in \cC_{i}'$. Hence the assignment $ b\mapsto b-a(b)$
  defines a function $\cC_{i}\to \cC_{i}'$ which, by Lemma
  \ref{lemm:9}\eqref{item:31}, is injective.

  Now for $c\in \cC_{i}'$  let $B$ be a mixed $n$-cell of
  $S(\rho_{i}^{\cc})$ with $c\in B+\delta$.  With notation as in
  \eqref{eq:35}, let $v\in N_{\R}$ be the unique vector such that
  $ B=\Gamma(\rho_{i}^{\cc}, v)$ and set
  \begin{displaymath}
    C=\Gamma(\rho, v)\in S(\rho)^{n}.
  \end{displaymath}
  By Proposition \ref{prop:10}\eqref{item:2}, $C$ is an $i$-mixed $n$-cell
  of $S(\rho)$ and $C_{i}^{\cc}=B$ and moreover, $C_{i}$ consists of a
  single lattice point $a\in \cA_{i}$. Setting $b=c+a \in C+\delta$,
  we have that $b-a(b)=b-a=c$ and so $\cC_{i}\to \cC_{i}'$ is surjective,
  proving the first claim.

  For the second, note that each mixed $n$-cell $B$ of
  $ S(\rho_{i}^{\cc})$ is a lattice parallelepiped, and so the
  genericity condition in \eqref{eq:20} implies that
  $\# (B+\delta)\cap M= \vol_{M}(B)$. Hence the cardinality of
  $\cC_{i}'$ is equal to the sum of the volumes of these $n$-cells
  which, by the formula in \eqref{eq:38}, coincides with $m_{i}$.
\end{proof}

We also need the next reformulation of a result by Pedersen and
Sturmfels~\cite{PedersenSturmfels:mmb} and by Emiris and Rege
\cite{EmirisRege:mbpss} on monomial basis of finite
dimensional~algebras.

\begin{lemma}
  \label{lemm:5}
  For $i=0,\dots, n$ let $\cC_{i}$ be the set of lattice points in
  $\cB_{i}$ lying in translated $i$-mixed $n$-cells of $S(\rho)$.
  Then there is a proper algebraic subset
  $Y_{i}\subset \prod_{j\ne i} \C^{\cA_{j}}$ such that for
  $(\overline \bfu_{j})_{j\ne i}\in \prod_{j\ne i} \C^{\cA_{j}} \setminus
  Y_{i}$, the zero set
  \begin{equation*}
    Z_{i}=Z( \{ F_{j}(\overline \bfu_{j}, \cdot)\}_{j\ne i}) \subset \T_{M}
  \end{equation*}
  has cardinality $m_{i}$ and the matrix
  $ (\chi^{b-a(b)}(p))_{b\in \cC_{i}, p\in Z_{i}} \in \C^{m_{i}\times
    m_{i}}$~is~nonsingular.
\end{lemma}

\begin{proof}
  Suppose without loss of generality that $i=0$.  By Bernstein's
  theorem \cite{Ber75} and the Bertini type theorem in \cite[Part I, Theorem
  6.3(3)]{Jou83}, there is a proper algebraic subset
  $Y_{0}\subset \prod_{j=1}^{n} \C^{\cA_{j}}$ such that for
  $(\overline \bfu_{1}, \dots, \overline \bfu_{n})\in \prod_{j=1}^{n}
  \C^{\cA_{j}}\setminus Y_{0}$, if we set
  $f_{j}=F_{j}(\overline \bfu_{j}, \cdot)\in\C[M]$, $j=0,\dots, n$, and
  $Z_{0}=Z(f_{1},\dots, f_{n})$, then the ideal
  $(f_{1},\dots, f_{n})\subset \C[M]$ is radical and $Z_{0}$ has
  cardinality $m_{0}$.

  By Lemma \ref{lemm:10}, the set $\{b-a(b)\}_{b\in \cC_{0}}$
  coincides with the set of lattice points in
  $\Delta_{0}^{\cc}+\delta=(\sum_{j=1}^{n}\Delta_{j}) +\delta$ lying
  in translated mixed $n$-cells of $S(\rho_{0}^{\cc})$.  After
  possibly enlarging $Y_{0}$, by \cite[Theorem
  1.1]{PedersenSturmfels:mmb} or \cite[Theorem 4.1]{EmirisRege:mbpss}
  the monomials $\chi^{b-a(b)}$, $b\in \cC_{0}$, form a basis of the
  quotient algebra $\C[M]/(f_{1},\dots, f_{n})$.

  Since the ideal $(f_{1},\dots, f_{n})$ is radical, the map
  $ \C[M]/(f_{1},\dots, f_{n}) \to \C^{Z_{0}}$ defined by
  $g\mapsto(g(p))_{p\in Z_{0}}$ is an isomorphism. Hence the vectors
  $(\chi^{b-a(b)}(p))_{p\in Z_{0}} \in \C^{Z_{0}}$, $b\in \cC_{0}$,
  are linearly independent, proving the lemma.
\end{proof}

\begin{proof}[Proof of Proposition \ref{prop:14}]
  By its definition in \eqref{eq:1}, the set $\cB_{i}$ contains the set
  of lattice points in the translated $i$-mixed $n$-cells of
  $S(\rho)$. Thus the hypothesis in  the  present statement
  amounts to the fact that $\cB_{i}$ is equal to this set of lattice
  points.  Lemma~\ref{lemm:10}, Proposition~\ref{prop:6} and the
  degree formula in \eqref{eq:2} then imply that
  \begin{displaymath}
\deg_{\bfu_{i}}(H_{\bfcA,\bfrho})=\deg_{\bfu_{i}}(\Res_{\bfcA})=m_{i}.
  \end{displaymath}

  When $m_{i}=0$ the statement is clear. Hence we suppose that
  $m_{i}>0$.  Consider then the $2 \times 2$-block decomposition
  \begin{displaymath}
    \cH_{\bfcA,\bfrho}=
    \begin{pmatrix}
      \cH_{1,1} & \cH_{1,2} \\
      \cH_{2,1} & \cH_{2,2}
    \end{pmatrix}
  \end{displaymath}
  where the first rows {and the first columns} correspond to $\cB_{i}$ and the others to
  $\cB_{j}$ for $j\ne i$. By Proposition~\ref{prop:3} the matrix
  $\cH_{2,2}$ is nonsingular and so
  \begin{equation}
    \label{eq:40}
  \begin{pmatrix}
    1 & - \cH_{1,2} \cdot \cH_{2,2}^{-1} \\
    0 & 1
  \end{pmatrix}
  \cdot
  \cH_{\bfcA,\bfrho}=
  \begin{pmatrix}
    \cH' & 0  \\
    \cH_{2,1} & \cH_{2,2}
  \end{pmatrix}
  \end{equation}
  with $\cH'=\cH_{1,1}- \cH_{1,2}\cdot \cH_{2,2}^{-1}\cdot \cH_{2,1}$.

  With notation as in Lemma \ref{lemm:5}, choose
  $(\overline \bfu_{j})_{j\ne i}\in \prod_{j\ne i} \C^{\cA_{j}} \setminus
  Y_{i}$.  For $\overline \bfu_{i}\in \C^{\cA_{i}}$ set
  $\overline \bfu=(\overline \bfu_{0},\dots, \overline \bfu_{n})\in \C^{\bfcA}$. Set also
  $f_{j}=F_{j}(\overline \bfu_{j},\cdot)\in \C[M]$ for each $j$ and denote by
  $Z_{i}\subset \T_{M}$ the zero set of the Laurent polynomials
  $f_{j}$, $j\ne i$.  With notation as in~\eqref{eq:33}, for each
  $p\in Z_{i}$ we have that
  \begin{equation}
    \label{eq:41}
    \cH_{\bfcA,\bfrho}(\overline \bfu) \cdot \zeta_{p}^{\transposed} = f_{i}(p) \,
    \eta_{p,i}^{\transposed}
  \end{equation}
  thanks to Lemma \ref{lemm:8}.  Consider the matrices in
  $ \C^{\cB_{i}\times Z_{i}} \simeq \C^{m_{i}\times m_{i}}$ defined as
\begin{displaymath}
\cP=(\chi^{b}(p))_{b\in \cB_{i}, p\in Z_{i}} \and \cQ=    ( \chi^{b-a(b)}(p) )_{b\in \cB_{i}, p\in Z_{i}}.
\end{displaymath}
From \eqref{eq:40} and \eqref{eq:41} we deduce  that
$ \cH'(\overline \bfu)\cdot \cP = \diag((f_{i}(p))_{p\in Z_{i}}) \cdot \cQ$ and so
\begin{displaymath}
    H_{\bfcA, \bfrho}(\overline \bfu)\cdot \det(\cP)= \det(\cH_{2,2}(\overline \bfu)) \cdot
  \det(\cH'(\overline \bfu)) \cdot \det(\cP)  = \det(\cH_{2,2}(\overline \bfu))
  \cdot\det(\cQ)\cdot \prod_{p\in Z_{i}}
f_{i}(p).
\end{displaymath}
By Lemma \ref{lemm:5}, the matrix $\cQ$ is nonsingular and so
$ H_{\bfcA,\bfrho}(\overline \bfu)=0$ only if there is $p\in Z_{i}$ such that
$f_{i}(p)=0$.  Hence both $H_{\bfcA,\bfrho}$ and $\Res_{\bfcA}$  are
polynomials of   degree  $m_{i}>0$ in the set of variables $\bfu_{i}$ that, for a generic choice of
$(\overline \bfu_{j})_{j\ne i}\in \prod_{j\ne i} \C^{\cA_{j}} $ vanish for
$\overline \bfu_{i}\in \C^{\cA_{i}}$ if and only if this also holds for the
irreducible polynomial $\Elim_{\bfcA}$ (recall that $\Res_{\bfcA}$ is a power of $\Elim_{\bfcA}$). Thus
\begin{displaymath}
  H_{\bfcA, \bfrho}= \gamma \cdot \Res_{\bfcA}
\end{displaymath}
with
$\gamma\in \Q(\bfu_{0},\dots, \bfu_{i-1},\bfu_{i+1},\dots,
\bfu_{n})^{\times}$, proving  the first claim.  The second  is a
direct consequence of the first together with Gauss' lemma, since $\Res_{\bfcA}$ is a primitive polynomial in $\Z[\bfu_{0},\dots, \bfu_{i-1},\bfu_{i+1},\dots,
\bfu_{n}][\bfu_{i}]$.
\end{proof}

\begin{corollary}
  \label{cor:6}
  We have that
  ${H_{\bfcA,\bfrho}}/{\Res_{\bfcA}} \in \Q(\bfu_{1},\dots,
  \bfu_{n})$.  Moreover, if $m_{0}>0$ then
  ${H_{\bfcA,\bfrho}}/{\Res_{\bfcA}} \in \Z[\bfu_{1},\dots,
  \bfu_{n}]$.
\end{corollary}

\begin{proof}
For  $b\in \cB_{0}$  let $C$ be the  $n$-cell of $S(\rho)$ such that
  $b\in C+\delta$. Then $\dim(C_{i})>0$ for all $i>0$ and so $C$ is
  $0$-mixed. Moreover, if  $m_{0}>0$ then $\cB_{0}\ne \emptyset $
  thanks to Lemma~\ref{lemm:10}. The corollary follows then from
  Proposition \ref{prop:14}.
\end{proof}

\begin{remark}
  \label{rem:4}
  We conjecture that $\Res_{\bfcA}\mid H_{\bfcA,\bfrho}$ in the
  general case. If this is true, then the hypothesis
  $\cB_{i}\ne \emptyset$ in Proposition \ref{prop:14} and that
  $m_{0}>0$ in Corollary \ref{cor:6} would not be necessary.
\end{remark}

The next corollary allows to compute the sparse resultant as the
greatest common divisor of a family of Canny-Emiris determinants, when
the fundamental subfamily of supports coincides with $\bfcA$. It
generalizes the method proposed in \cite[\S7]{CanEmi:sbasr} to the
situation when the sublattice $L_{\bfcA}$ is not necessarily equal to
$M$.

\begin{corollary}
  \label{cor:5}
  Suppose that the fundamental subfamily of supports coincides with
  $\bfcA$ and choose a permutation $\sigma_{i}$ of the index set
  $\{0,\dots, n\}$ with $\sigma_{i}(0)=i$ for each~$i$.~Then
  \begin{displaymath}
     \Res_{\bfcA}(\bfu)=\pm \gcd(H_{\sigma_{0}(\bfcA),
    \sigma_{0}(\bfrho)}(\sigma_{0}(\bfu)), \dots,
  H_{\sigma_{n}(\bfcA), \sigma_{n}(\bfrho)}(\sigma_{n}(\bfu))).
  \end{displaymath}
\end{corollary}

\begin{proof}
  Corollary \ref{cor:6} applied to the data
  $(\sigma_{i}(\bfcA), \sigma_{i}(\bfrho), \delta)$ implies that
  $ \Res_{\sigma_{i}(\bfcA)}$ divides
  $ H_{\sigma_{i}(\bfcA), \sigma_{i}(\bfrho)}$ and that both
  polynomials have the same degree in the set of variables
  $\bfu_{0}$. By the invariance of the sparse resultant under
  permutations of the supports we deduce that $ \Res_{\bfcA}(\bfu) $
  divides
  $ H_{\sigma_{i}(\bfcA), \sigma_{i}(\bfrho)} (\sigma_{i}(\bfu)) $ and
  that both polynomials have the same degree in the set of variables
  $\bfu_{i}$ for each $i$, which implies the statement.
\end{proof}

\subsection{The Macaulay formula for the sparse resultant}\label{formula}
In this section   we give the proof of our main result, Theorem
\ref{thm:9} in the introduction.  It is based on the
constructions and results from the previous sections, and we keep the
notations therein.  In particular,
\begin{itemize}
\item $\bfrho=(\rho_{0},\dots, \rho_{n})$ is a family of $n+1$ convex
  piecewise affine functions on the polytopes in
  $ \bfDelta= (\Delta_{0},\dots, \Delta_{n})$ defined on the supports
  in $\bfcA=(\cA_{0},\dots, \cA_{n})$ such that the mixed subdivision
  $S(\rho)$ defined by its inf-convolution is tight
  (Definition~\ref{def:7}),
\item $\delta$ is a vector in $ M_{\R}$ that is generic with respect
  to $S(\rho)$ in the sense of \eqref{eq:20},
\item $\cB$ is the index set and $ \cB_{i}$, $i=0,\dots, n$, the
  subsets partitioning it, and $\cB^{\circ}$ is the nonmixed index subset of $\cB$,
\item $\Phi_{\bfcA, \bfrho}$, $\cH_{\bfcA, \bfrho}$ and
  $H_{\bfcA, \bfrho}$ are the Sylvester map and the Canny-Emiris
  matrix and determinant associated to the data
  $(\bfcA, \bfrho, \delta)$,
\item
$ \cE_{\bfcA,\bfrho} $ and
$E_{\bfcA,\bfrho}$ are the  principal submatrix and minor
corresponding to  $\cB^{\circ}$,
\item $(\bfcA_{D}, \bfrho_{D})$ is the restriction of
  $(\bfcA, \bfrho)$ to an $n$-cell $D$ of a mixed subdivision
  that is coarser than $S(\rho)$.
\end{itemize}

To prove Theorem \ref{thm:9}, we use a descent argument similar to
that of Macaulay in \cite{Macaulay:sfe} and the first author in
\cite{DAndrea:msfsr}.  The following notion comprises the 
properties of a tight mixed subdivision that allow us to perform this descent.

\begin{definition}
  \label{def:5}
  The tight mixed subdivision $S(\rho)$ is \emph{admissible} if there
  is an incremental chain of mixed subdivisions  of~$\Delta$ (Definition
  \ref{def:9})
  \begin{equation}
    \label{eq:78}
S(\theta_{0})\preceq \dots\preceq S(\theta_{n})
  \end{equation}
  with $S(\theta_{n})\preceq S(\rho)$ such that for $k=0,\dots, n$,
  each $n$-cell $D$ of $S(\theta_{k})$ verifies at least one of the
  conditions:
  \begin{enumerate}
\item \label{item:62} the fundamental subfamily of $\bfcA_{D}$
  has at most one support,
\item \label{item:63} the subset $\cB_{D,k}$ of $\cB_{k}$ defined in
  \eqref{eq:73} is contained in the union of the translated $k$-mixed
  $n$-cells of $S(\rho_{D})$.
  \end{enumerate}
  The incremental chain of mixed subdivisions in \eqref{eq:78} is called \emph{admissible} (for~$S(\rho)$).
\end{definition}

The next result gives sufficient conditions for a given incremental
chain to be admissible that will allow us to recover Macaulay's
original formulation with our methods (Proposition \ref{prop:15}).

\begin{proposition}
  \label{prop:16}
  Let $S(\theta_{0})\preceq \dots\preceq S(\theta_{n})$ be an
  incremental chain of mixed subdivisions of $\Delta$ such that for
  $k=0,\dots, n$, each $n$-cell $D$ of $S(\theta_{k})$ verifies at
  least one of the conditions:
  \begin{enumerate}
  \item \label{item:52} there is $ J\subset \{0,\dots, n\}$ such that
    $\dim (\sum_{j \in J} D_{j}) < \# J -1$,
    \smallskip
  \item \label{item:51} there is $i\in\{0,\dots,n\}$ such that
    $\dim(D_{i})=0$,
      \smallskip
  \item \label{item:53} for all $i<k $ we have that
    $ \dim(\sum_{j \ne i,k} D_{j}) <n$.
  \end{enumerate}
  Then this incremental chain is admissible for any tight mixed
  subdivision of $\Delta$ that refines~$S(\theta_{n})$.
\end{proposition}

\begin{proof}
  Let $k=0,\dots, n$ and $D \in S(\theta_{k})^{n}$. If this $n$-cell
  satisfies the condition~\eqref{item:52}, then for $i=0,\dots, n$ we
  have that
  \begin{displaymath}
    \dim\Big(\hspace{2mm} \sum_{\mathclap{j\in J\setminus \{i\}}} D_{j}\Big) \le \dim\Big(
    \sum_{j\in J} D_{j}\Big)<  \# J -1 \le  \#(J\setminus
  \{i\}).
\end{displaymath}
By the basic properties of the mixed volume,
$\MV_{M}(D_{0},\dots, D_{i-1},D_{i+1},\dots, D_{n})=0$. Hence
$\Res_{\bfcA_{D}}=\pm1$ thanks to the degree formula in \eqref{eq:2}
and so the fundamental subfamily of $\bfcA_{D}$ is empty. Thus in this
case $D$ satisfies the condition \eqref{item:62} in
Definition~\ref{def:5}.

If $D$ satisfies the condition \eqref{item:51}, then also
$\MV_{M}(D_{0},\dots, D_{j-1}, D_{j+1}, \dots, D_{n})=0$ for all
$j\ne i $ and so the fundamental subfamily of $\bfcA_{D}$ is either
empty or consists of the single support $\cA_{i}$. In both cases, $D$
also satisfies the condition \eqref{item:62} in Definition
\ref{def:5}.

Finally suppose that $D$ satisfies the condition \eqref{item:53}.  For
$b\in \cB_{D,k}$ let $C$ be the $n$-cell of $S(\rho_{D})$ such that
$b\in C+\delta$.  We have that $\dim(C_{k})=0$ and that
$C_{j}\subset D_{j}$ for all $j$, and so for each $i<k$ we have that
\begin{displaymath}
  \dim(C_{i}) = n-\sum_{j\ne i, k} \dim(C_{j}) =n-\dim\Big(\sum_{j\ne i, k} C_{j}\Big)
  \ge n-\dim\Big(\sum_{j\ne i, k} D_{j}\Big)>0.
\end{displaymath}
Since $b\in \cB_{D,k}$ we also have that $\dim(C_{i}) >0$ for all
$i>k$ and so $C$ is $k$-mixed. Hence in this case $D$ satisfies the
condition \eqref{item:63} in Definition \ref{def:5}. We conclude that
the incremental chain is admissible for any tight mixed subdivision of
$\Delta$ refining~$S(\theta_{n})$.
\end{proof}

As a consequence of this result, we deduce that tight incremental
chains of mixed subdivisions are admissible.

\begin{proposition}
  \label{prop:19}
  An incremental chain
  $S(\theta_{0})\preceq \dots\preceq S(\theta_{n})$ of mixed
  subdivisions of $\Delta$ with $S(\theta_{n})\preceq S(\rho)$
  {that is tight in the sense of Definition \ref{def:9}},
  satisfies the conditions in Proposition~\ref{prop:16}. In
  particular, it is admissible.
\end{proposition}

\begin{proof}
  Since the incremental chain is tight, for $k=0,\dots, n$ and
  $D \in S(\theta_{k})^{n}$ we have that
  \begin{equation}
    \label{eq:32}
    \sum_{j=0}^{k-1}\dim(D_{j}) + \dim\Big( \sum_{j=k}^{n} D_{j}\Big) =n .
  \end{equation}
  If there is $i<k$ such that $D_{i}$ is a point, then $D$ satisfies
  the condition \eqref{item:51} in Proposition~\ref{prop:16}. Else
  $\dim(D_i)>0$ for all $i< k$ and so the equality in \eqref{eq:32}
  implies~that
  \begin{displaymath}
    \dim\Big(\sum_{j \ne i} D_{j}\Big) \le  \sum_{i\ne j <k} \dim(D_{j}) + \dim\Big( \sum_{j=k}^{n} D_{j}\Big)
= n - \dim(D_{i}) <n,   
  \end{displaymath}
  and so $D$ satisfies the condition \eqref{item:53} in Proposition
  \ref{prop:16}. This gives the first claim, whereas the second is an
  application of that proposition.
\end{proof}

\begin{corollary}
  \label{cor:2}
  The incremental chains of mixed subdivisions
  $ S(\theta_{0})\preceq \dots\preceq S(\theta_{n})$ of $\Delta$
  obtained by setting $s=n$ in Example \ref{exm:9}, are 
  admissible for $S(\theta_{n})$.
\end{corollary}

The next result gives the basic particular cases of Theorem
\ref{thm:9} that can be treated directly.  Recall that
$ m_{i}= \MV_{M}(\Delta_{0},\dots, \Delta_{i-1}, \Delta_{i+1},\dots,
\Delta_{n})$ for each $i$, as in \eqref{eq:18}.

\begin{proposition}
  \label{prop:18}
  Let $\bfcA_{J}$ be the fundamental subfamily of $\bfcA$. Then
\begin{enumerate}
\item \label{item:64} when $\bfcA_{J}=\emptyset $ we have that
  $\Res_{\bfcA}=\pm1$ and $H_{\bfcA,\bfrho}= E_{\bfcA,\bfrho}$,
\item \label{item:65} when $\bfcA_{J}$  consists  of a single support
  $\cA_{i}$, we have that $ \Res_{\bfcA}=\pm u_{i , a }^{m_{i }} $ and
  $ H_{\bfcA,\bfrho}=u_{i ,a }^{m_{i }} \, E_{\bfcA,\bfrho}$ for the
  unique lattice point  $a \in M$ such that $\cA_{i}=\{a \}$.
\end{enumerate}
\end{proposition}

\begin{proof}
  The first claim in \eqref{item:64} is a direct consequence of the
  hypothesis that $J=\emptyset$.  The same hypothesis together with
  the degree formula in \eqref{eq:2} implies that $m_{j} =0$ for all
  $j$. Lemma \ref{lemm:10} then implies that $\cB^{\circ}=\cB$, which
  gives the second claim.

  For \eqref{item:65}, first note that, by the rank condition in Proposition \ref{prop:12}, if $\#J =1$ for the fundamental subfamily $\bfcA_J$, its unique element $\cA_i$ is a singleton. Then, in this case, $\bfu_{i }=\{u_{i ,a }\}$ and
  $m_{j}=0$ for all $j\ne i $. The first claim follows then from the
  fact that $\Res_{\bfcA}$ is a primitive homogeneous polynomial in
  $\Z[u_{i ,a }]$ of degree $m_{i }$.  The hypothesis that
  $\cA_{i }=\{a \}$ also implies that $S(\rho)$ has no $n$-cells that
  are $j$-mixed for $j\ne i $.  By Lemma \ref{lemm:10}, the subset
  $\cC_{i} \subset \cB_{i }$ of lattice points lying in the translated
  $n$-cells that are $i $-mixed has cardinality $m_{i }$. For each
  $b \in \cC_{i}$ we have that $\rc(b)=(i ,a )$ and so, for
  $b'\in \cB$,
  \begin{equation*}
    \cH_{\bfcA,\bfrho}[{b,b'}]=
    \begin{cases}
      u_{i ,a } & \text{ if } b'=b, \\
      0 & \text{ otherwise.}
    \end{cases}
  \end{equation*}
  Since $\cB^{\circ}=\cB\setminus \cC_{i}$ and $\cE_{\bfcA,\bfrho}$ is
  the principal submatrix of $\cH_{\bfcA,\bfrho}$ corresponding to
  this subset, we deduce the second claim.
\end{proof}

Now we are ready for the proof of the main result of this paper,
corresponding to Theorem \ref{thm:9} in the introduction.

\begin{theorem}
\label{thm:3}
  Suppose that $S(\rho)$ is admissible and let
  $S(\theta_{0})\preceq \dots\preceq S(\theta_{n})$ with
  $S(\theta_{n})\preceq S(\rho)$ be an admissible incremental chain
  for it. Then
\begin{displaymath}
  \Res_{\bfcA} =\pm
  \frac{H_{\bfcA,\bfrho}}{E_{\bfcA,\bfrho}} \and
  E_{\bfcA, \bfrho}= \prod_{D} E_{\bfcA_{ D}, \bfrho_{D}},
\end{displaymath}
the product in the second formula being over the $n$-cells $D$ of $S(\theta_{1})$.
\end{theorem}

\begin{proof}
  Let $S(\theta_{0})\preceq \dots\preceq S(\theta_{n})$ with
  $S(\theta_{n})\preceq S(\rho)$ be an admissible incremental chain,
  and for $k=0,\dots, n$ let $\theta_{k,j}\colon \Delta_{j}\to \R$,
  $j=0,\dots, n$, be the family of convex piecewise affine functions
  corresponding to $\theta_{k}$. Set also $\theta_{n+1}=\rho$.

  We prove by reverse induction on $k$ that for every $n$-cell $D$ of
  $ S(\theta_{k})$ we have that
  \begin{equation}
    \label{eq:74}
    \frac{H_{\bfcA_{D}, \bfrho_{D}}}{\Res_{\bfcA_{D}}}= \pm E_{\bfcA_{D}, \bfrho_{D}}.
  \end{equation}
  The {first statement} in the theorem corresponds to the case
  when $k=0$ and $D=\Delta$.

  For $k=n+1$ we note that $S(\theta_{n+1})=S(\rho)$ is tight. Hence
  for $D\in S(\theta_{n+1})$ we have that
  $\sum_{i=0}^{n}\dim(D_{i})=n$ and so $\dim(D_{i})=0$ for at least
  one $i$. This implies that the fundamental subfamily of $\bfcA_{D}$
  is either empty or consists of the single support
  $\cA_{i} \cap D_{i}$. Hence $\bfcA_{D}$ verifies the hypothesis of
  Proposition~\ref{prop:18}, and so the equality in \eqref{eq:74}
  follows from this result.

  Hence suppose that $0\le k\le n$ and let $D\in S(\theta_{k})$.  In
  case $D$ satisfies the condition~\eqref{item:62} in
  Definition~\ref{def:5}, the equality in~\eqref{eq:74} follows
  similarly from Proposition~\ref{prop:18}. If this does not happen,
  then $D$ satisfies the condition \eqref{item:63} in this definition,
  and so the subset $\cB_{D,k}$ is contained in the union of the
  translated $k$-mixed $n$-cells of $S(\rho_{D})$.

  By Proposition \ref{prop:22}, for all $i<k$ and every $n$-cell $C$
  of $S(\theta_{i})$ containing~$D$, the index set of the fundamental
  subfamily of $\bfcA_{C}$ contains that of $\bfcA_{D}$, and so this
  $n$-cell cannot satisfy the condition \eqref{item:62} in
  Definition~\ref{def:5}. Hence it satisfies the condition
  \eqref{item:63} in this definition and so $\cB_{C,i}$ is contained
  in the union of the translated $i$-mixed cells of
  $S(\rho_{C})$. From here we deduce that $\cB_{D,i}$ is contained in
  the union of the translated $i$-mixed cells of $S(\rho_{D})$ and as
  mentioned, the same happens for $i=k$. Together with
  Proposition~\ref{prop:14}, this implies that
  \begin{displaymath}
    E_{\bfcA_{ D}, \bfrho_{D}} \in \Z[\bfu_{k+1},\dots, \bfu_{n}] \and
\frac{H_{\bfcA_{ D}, \bfrho_{D}}}{\Res_{\bfcA_{ D}}} \in
    \Q(\bfu_{k+1},\dots, \bfu_{n}).
  \end{displaymath}
  Consider the vector $\bfomega\in \R^{\bfcA}$ defined by
  $\omega_{j,a}=\theta_{k+1,j}(a)$ for $j=0,\dots, n$ and
  $a\in \cA_{j}$.  Since the chain is incremental, we have that
  $\omega_{j,a}=0$ for $j\ge k+1$ and $a\in \cA_{j}$ and so
  \begin{displaymath}
    E_{\bfcA_{ D}, \bfrho_{D}}=\init_{\bfomega}(E_{\bfcA_{ D},
      \bfrho_{D}}) \and    \frac{H_{\bfcA_{ D}, \bfrho_{D}}}{\Res_{\bfcA_{ D}}}= \init_{\bfomega}\Big(
    \frac{H_{\bfcA_{ D}, \bfrho_{D}}}{\Res_{{\bfcA_{ D}}}}\Big) =  \frac{\init_{\bfomega}
      (H_{\bfcA_{ D}, \bfrho_{D}})}{\init_{\bfomega}(\Res_{\bfcA_{ D}})}.
  \end{displaymath}
Applying Theorem~\ref{mt1} and Theorem \ref{thm:2} for the subsets $\cB_{D}$ and $\cB^{\circ}_{D}$, we deduce that
\begin{equation}
  \label{eq:22}
\init_{\bfomega}(E_{\bfcA_{ D}, \bfrho_{D}})=
\prod_{D'} E_{\bfcA_{ D'}, \bfrho_{D'}}
\and
\frac{\init_{\bfomega}
      (H_{\bfcA_{ D}, \bfrho_{D}})}{\init_{\bfomega}(\Res_{\bfcA_{ D}})} =
  \prod_{D'}    \frac{H_{\bfcA_{ D'}, \bfrho_{D'}}}{\Res_{\bfcA_{ D'}}},
  \end{equation}
  both products being over the $n$-cells $D'$ of
  $S(\Theta_{\bfomega})=S(\theta_{k+1})$ that are contained in $D$.
  The equality in \eqref{eq:74} then follows from the inductive
  hypothesis.

  {The second statement follows from the first equality in
    \eqref{eq:22} applied to the case $k=0$ and $D=\Delta$.}
\end{proof}

Theorem \ref{thm:3} generalizes the Canny-Emiris conjecture since, by
Proposition \ref{prop:8}, the sparse eliminant coincides with the
sparse resultant of the fundamental subfamily of supports with respect
to the minimal lattice containing it.

\begin{remark}
  \label{rem:6}
  We can extend the Canny-Emiris construction to a larger class of
  matrices, following an idea of the first author in
  \cite{DAndrea:msfsr}. The study of its properties can be done
  similarly as for the original formulation, and so we only indicate
  the modifications.

  Let $\psi\colon \Lambda \to \R$ be a convex piecewise affine
  function on a polytope such that the mixed subdivision
  $S( \rho\boxplus \psi)$ on $\Delta+\Lambda$ is tight and its
  $(n-1)$-th skeleton does not contain any lattice point. The case
  treated in this paper corresponds  to the situation when
  $\Lambda=\{\delta\}$ and $\psi$ takes any value at this point.

  The index set is defined as $ \cB= ( \Delta+\Lambda) \cap M $.  For
  each $b\in \cB$ there is a unique $n$-cell $C$ of
  $S(\rho\boxplus \psi)$ containing it, and we denote by
  $C_{i}\in S(\rho_{i})$, $i=0,\dots, n$, and $B \in S(\psi)$ its
  components.  The tightness condition implies that there is always an
  $i$ such that $\dim(C_{i})=0$, in which case $C_{i}$ consists of a
  single lattice point of $\cA_{i}$.  We then set
\begin{displaymath}
  i(b) \in \{0,\dots, n\} \and a(b)\in \cA_{i(b)}
\end{displaymath}
for the largest of these $i$'s and the unique lattice point in the
corresponding component, respectively.  For $b\in \cB$ we have that
$b-a(b)+\cA_{i(b)}\subset \cB$ and so we can define a Sylvester map
$\Phi_{\bfcA,\bfrho,\psi} $ and the corresponding Canny-Emiris matrix
$\cH_{\bfcA,\bfrho,\psi}$ and determinant $H_{\bfcA,\bfrho,\psi}$ in
the same way as it was done before.

For $i\in \{0,\dots, n\}$, we say that an $n$-cell $C$ of
$S(\rho\boxplus \psi)$ is $i$-mixed if $\dim(C_{j})=1$ for all
$j\ne i$. The nonmixed index subset $\cB^{\circ} $ is the subset of
$\cB$ of lattice points lying in the $n$-cells that are not $i$-mixed
for any $i$, and we denote by $\cE_{\bfcA,\bfrho,\psi}$ and
$E_{\bfcA,\bfrho,\psi}$ the corresponding principal submatrix and
minor of the Canny-Emiris matrix.

Theorem \ref{thm:3} can then be extended to the statement that
\begin{displaymath}
  \Res_{\bfcA}= \pm \frac{H_{\bfcA,\bfrho,\psi}}{E_{\bfcA,\bfrho,\psi}}
\end{displaymath}
whenever the mixed subdivision $S(\rho)$ is admissible in the sense of
Definition~\ref{def:5}, together with a factorization for
$ E_{\bfcA, \bfrho,\psi}$ in terms of the $n$-cells of the second
mixed subdivision in an admissible chain for $S(\rho)$.
\end{remark}

\begin{remark}
  \label{rem:7}
  The extraneous factor $ E_{\bfcA, \bfrho}$ does not depend on the
  choice of the admissible chain
  $S(\theta_{0})\preceq \dots\preceq S(\theta_{n})$ but its
  factorization in Theorem \ref{thm:3} in principle depends on the
  second mixed subdivision $S(\theta_{1})$. One can go further and
  refine this factorization using the subsequent mixed subdivisions in
  this chain. It would be interesting to exhibit concrete cases where
  different admissible chains produce different factorizations for
  this extraneous factor.

  In addition, our factorization of $ E_{\bfcA, \bfrho}$ consists of a
  product of extraneous factors of smaller systems in the same
  dimension, whereas the factorizations presented by Macaulay in
  \cite[page 14]{Macaulay:sfe} and by the first author in
  \cite[(47)]{DAndrea:msfsr} consist of Canny-Emiris determinants and
  extraneous factors of systems in lower dimensions. It would be
  interesting to compare these approaches and put them into a more
  general framework.
\end{remark}

\section{Homogeneous resultants} \label{sec:homog-result-1}

The goal of this section is to show that Macaulay's classical formula
for the homogeneous resultant can be recovered as a particular case of
our construction.  Macaulay's row content function is given in terms
of the exponents of the monomials indexing the matrix, and we will
exhibit a family of affine functions on multiples of the standard
simplex whose associated mixed subdivision reflects this
idea. Unfortunately, this mixed subdivision is not tight in dimension
$\ge 3$ (Remark \ref{rem:11}) but we show that any tight refinement of
it will do the work (Proposition \ref{prop:2}). This is done by
constructing a chain of mixed subdivisions for this refinement that
is admissible (Proposition \ref{prop:15}).

\subsection{The classical Macaulay formula} \label{sec:class-maca-form}

In this section we describe the Macaulay formula for the homogeneous
resultant from \cite{Macaulay:sfe}.

Let $\bfd=(d_{0},\dots, d_{n})\in (\N_{>0})^{n+1}$. For $i=0,\dots, n$
let $\bfu_{i}=\{u_{i,\bfc}\}_{ |\bfc|=d_{i}}$ be a set of
${d_{i}+n\choose n}$ variables indexed by the lattice points
$\bfc \in \N^{n+1}$ of length $|\bfc|=d_{i}$, put
$\bfu=(\bfu_{0}, \dots, \bfu_{n})$ and denote by
\begin{displaymath}
\Res_{\bfd} \in \Z[\bfu]
\end{displaymath}
the corresponding {homogeneous} resultant as in Example \ref{exm:3}.

Let  $\bft=\{t_{0},\dots, t_{n}\}$  be a further set of $n+1$
variables. By \cite[Chapter~3, Theorem 2.3]{CLO05}, $\Res_{\bfd}$
is the unique irreducible polynomial in $\Z[\bfu]$
vanishing when evaluated at the coefficients of a system of $n+1$
homogeneous polynomials in the variables $\bft$ of degrees $\bfd$ if
and only if this system has a zero in the projective space
$\P^{n}_{\C}$, and verifying that
$\Res_{\bfd}(t_{0}^{d_{0}},\dots, t_{n}^{d_{n}})=1$.

Set $\K=\C(\bfu)$ and choose  an integer  $ m  \ge |\bfd| - n$.
For $i=0,\dots, n$ consider the general {homogeneous} polynomial in
the variables $\bft$ of degree $d_{i}$
\begin{displaymath}
  P_{i}=\sum_{|\bfc|=d_{i}} u_{i,\bfc}\, \bft^{\bfc} \in
  \K[\bft]
\end{displaymath}
and the linear subspace of the homogeneous part $\K[\bft]_{ m -d_{i}} $
given~by
\begin{equation}
  \label{eq:43}
  T_{i}=\{ Q_{i}\in \K[\bft]_{ m -d_{i}} \mid \deg_{t_{j}}(Q_{i})<d_{j}
  \text{ for } j=i+1,\dots, n\},
\end{equation}
where $\deg_{t_{j}}(Q_{i})$ denotes the degree of $Q_{i}$ in the
variable $t_{j}$. Set also $T= \K[\bft]_{ m }$ and consider the
linear map $\Psi_{\bfd, m }\colon \bigoplus_{i=0}^{n}T_{i}\to T$
defined by
\begin{equation*}
  \Psi_{\bfd, m }(Q_{0},\dots, Q_{n}) = \sum_{i=0}^{n} Q_{i}\, P_{i}.
\end{equation*}
Let $ \cI=\{\bfc\in \N^{n+1}\}_{ |\bfc|= m }$ be the \emph{index
  set}, and consider also the finite subsets
\begin{displaymath}
  \cI_{i}=\{\bfc=(c_{0},\dots, c_{n})\in \cI \mid c_{i}\ge
  d_{i} \text{ and } c_{j}<d_{j} \text{ for } j>i \}, \
  i=0,\dots, n,
\end{displaymath}
which form a partition of it.  Denoting by $\wh e_{i}$, $i=0,\dots, n$,
the vectors in the standard basis of $\R^{n+1}$,  the sets of
monomials
\begin{equation}
  \label{eq:58}
 \{\bft^{\bfc-d_{i}\, \wh e_{i}}\}_{ \bfc\in
    \cI_{i}}, \ i=0,\dots, n, \and   \{\bft^{\bfc}\}_{ \bfc\in \cI}
\end{equation}
are bases of $T_{i}$, $i=0,\dots, n$, and of $T$ respectively.  Then
we set
\begin{displaymath}
  \cM_{\bfd, m } \in \K^{\cI\times \cI}
\end{displaymath}
for the matrix of $\Psi_{\bfd, m }$ in terms of row vectors and with
respect to the monomial bases of $ \bigoplus_{i=0}^{n}T_{i}$ and of
$T$ given by \eqref{eq:58}, both indexed by $\cI$.

Set also $\cI^{\red}\subset \cI$ for the subset of lattice
points $\bfc=(c_{0},\dots, c_{n}) \in \cI$ with a unique $i$ such
that $c_{i}\ge d_{i}$. We then denote by
$\cI^{\circ}=\cI\setminus \cI^{\red}$  the \emph{nonreduced
  index subset} and by $\cN_{\bfd, m }$ the corresponding principal
submatrix of $\cM_{\bfd, m }$, The Macaulay formula for the
homogeneous resultant  \cite{Macaulay:sfe} then states that
\begin{equation}
  \label{eq:39}
  \Res_{\bfd}=\frac{\det(\cM_{\bfd, m })}{\det(\cN_{\bfd, m })},
\end{equation}
see \cite[Proposition 3.9.4.4]{Jouanolou:firf} for a modern treatment.

\begin{example}
  \label{exm:4}
Let $n=2$, $\bfd=(1,2,2)$ and $ m =|\bfd|-n=3$. The corresponding
  general homogeneous polynomials are
\begin{align*}
  P_{0}&= \alpha_{0}\, t_{0}+\alpha_{1}\, t_{1}+\alpha_{2}\, t_{2} , \\
  P_{1}&= \beta_{0}\, t_{0}^{2}+\beta_{1}\, t_{0}\, t_{1}+ \beta_{2}\,
         t_{1}^{2}+\beta_{3}\, t_{0}\, t_{2}+\beta_{4}\, t_{1}\,
         t_{2}+\beta_{5} \, t_{2}^{2}, \\
  P_{2}&= \gamma_{0}\, t_{0}^{2} +\gamma_{1}\, t_{0}\, t_{1}+ \gamma_{2}\,
         t_{1}^{2}+\gamma_{3}\, t_{0}\, t_{2}+\gamma_{4}\, t_{1}\,
         t_{2}+\gamma_{5} \, t_{2}^{2},
\end{align*}
and the index set splits as
$\cI=\cI_{0}\sqcup\cI_{1}\sqcup\cI_{2}$ with
\begin{align*}
  \cI_{0}&=\{(3,0,0),(2,1,0),(2,0,1),(1,1,1)\}, \\
  \cI_{1}&=\{(1,2,0),(0,3,0),(0,2,1)\}, \\
  \cI_{2} &=\{(1,0,2),(0,1,2),(0,0,3)\}.
\end{align*}
The matrix $\cM_{\bfd, m }$ is constructed by declaring that the row
corresponding to each lattice point $\bfc\in \cI_{i}$, $i=0,1,2$,
consists of the coefficients of the polynomials
$\bft^{\bfc-d_{i}\, \wh e_{i}}P_{i}$ in the monomial basis of the
homogeneous part $\K[\bft]_{3}$. Hence this matrix is written as
\begin{equation}
  \label{eq:57}
  {  \kbordermatrix{ & t_{0}^{3} & t_{0}^{2}\, t_{1} & t_{0}^{2}\,
      t_{2} & t_{0}\, t_{1}\, t_{2} & t_{0}\, t_{1}^{2} &    t_{1}^{3} & t_{1}^{2}\, t_{2} & t_{0}\, t_{2}^{2} & t_{1}\, t_{2}^{2} &    t_{2}^{3} \\
      t_{0}^{2}\,     P_{0}             & \alpha_{0}& \alpha_{1} & \alpha_{2} & 0&0&0&0&0&0&0\\
      t_{0}\, t_{1}\, P_{0}     &                 0   & \alpha_{0} &0 & \alpha_{2}& \alpha_{1}&0&0&0&0 &0\\
      t_{0}\,  t_{2}\, P_{0}     &0&0&\alpha_{0}&\alpha_{1}&0&0&0&\alpha_{2}& 0&0\\
      t_{1}\, t_{2}\, P_{0}     &0&0&0&\alpha_{0}&0&0&\alpha_{1}&0&\alpha_{2} &0 \\
      t_{0}\, P_{1}
      &\beta_{0}&\beta_{1}&\beta_{3}&\beta_{4}&\beta_{2}&0&0&\beta_{5}&0 &0 \\
      t_{1}\, P_{1}     &0&\beta_{0}&0&\beta_{3}&\beta_{1}&\beta_{2}&\beta_{4}&0& \beta_{5}&0\\
      t_{2}\, P_{1}
      &0&0&\beta_{0}&\beta_{1}&0&0&\beta_{2}&\beta_{3}&\beta_{4}&\beta_{5} \\
      t_{0}\,  P_{2}
      &\gamma_{0}&\gamma_{1}&\gamma_{3}&\gamma_{4}&\gamma_{2}&0&0&\gamma_{5}& 0&0 \\
      t_{1}\, P_{2}     &0&\gamma_{0}&0&\gamma_{3}&\gamma_{1}&\gamma_{2}&\gamma_{4}&0& \gamma_{5}&0\\
      t_{2}\, P_{2}
      &0&0&\gamma_{0}&\gamma_{1}&0&0&\gamma_{2}&\gamma_{3}&\gamma_{4}&\gamma_{5}
    }}.
\end{equation}
We have that $\cI\setminus \cI^{\red}= \{(1,2,0),(1,0,2)\}$ and
so
\begin{displaymath}
  \cN_{\bfd, m }
  =\begin{pmatrix}
    \beta_{2}& \beta_{5}\\
    \gamma_{2}& \gamma_{5}
    \end{pmatrix}.
\end{displaymath}
By the identity in \eqref{eq:39}, $ \Res_{\bfd}$ is the quotient of
the determinants of these matrices.  It is an irreducible
trihomogeneous polynomial in $\Z[\bfalpha, \bfbeta, \bfgamma]$ of
tridegree $(4,2,2)$ having $234$ monomial terms.
\end{example}

\subsection{A mixed subdivision on a
  simplex}\label{sec:mixed-regul-subd}

In this section we study a specific mixed subdivision of a  scalar
 multiple of the standard simplex of $\R^{n}$ that, in the next
one, will be applied to the analysis of the classical Macaulay formula
for the homogeneous resultant.

For $i=0,\dots, n$ consider the simplex
$ \Delta_{i} = \{\bfx\in (\R_{\ge0})^{n}\mid |\bfx| \le d_{i}\} $ and
the affine function $\varphi_{i}\colon \Delta_{i}\to \R$ defined by
\begin{equation}
  \label{eq:87}
  \varphi_{i}(\bfx)=
  \begin{cases}
    |\bfx|    & \text{ if } i=0, \\
    d_{i}-x_{i}& \text{ if } i>0,
  \end{cases}
\end{equation}
where $|\bfx|=\sum_{i=1}^{n}x_{i}$ denotes the {length} of the point
$\bfx=(x_{1},\dots, x_{n})\in \R^{n}$.  Set then
\begin{displaymath}
  \Delta=\sum_{i=0}^{n} \Delta_{i}=\{\bfx\in (\R_{\ge 0})^{n}\mid |\bfx|\le
  |\bfd|\} \and  \varphi=\bigboxplus_{i=0}^{n}\varphi_{i} \colon
  \Delta\longrightarrow \R
\end{displaymath}
for the Minkowski sum of these simplexes~and the inf-convolution of
these affine functions, respectively.

For a subset $I\subset\{0,\dots, n\}$ denote by
$I^{\cc}=\{0,\dots, n\}\setminus I$ its complement and consider the
polytope of $\Delta$ defined as
\begin{equation}
  \label{eq:61} C_{I}= \{\bfx\in \Delta \mid \ell_{i}(\bfx) \ge 0
\text{ for } i \in I \text{ and } \ell_{i}(\bfx) \le 0 \text{ for } i \in I^{\cc}\},
\end{equation}
where $\ell_{i}\colon \Delta\to \R$ is the affine function given by
$\ell_{i}(\bfx)= \sum_{j=1}^{n}d_{j} - |\bfx| $ when $i=0$ and by
$\ell_{i}(\bfx)=x_{i}-d_{i}$ when $i>0$.  For a subset $J\subset
I^{\cc}$ and an index $l\in I$
consider the lattice point defined as
    \begin{equation*}
   v_{J,l}=\Big( \sum_{j\in J} d_{j}\Big)  e_{l} + \sum_{j\in
        J^{\cc}}d_{j}\,  e_{j}
  \end{equation*}
   with $e_{i}$ equal to the $i$-th vector in the standard basis of
  $\R^{n}$ when $i>0$ and $e_{0}=\bfzero\in \R^{n}$.  For
  $i=0,\dots, n$ consider also the face of $ \Delta_{i}$ defined~as
\begin{equation}
  \label{eq:25}
  C_{I,i}=
  \begin{cases}
    d_{i}\, e_{i} & \text{ if } i\in I, \\
    d_{i} \conv( e_{i}, \{e_{j}\}_{j\in I}) & \text{ if } i\in I^{\cc}.
  \end{cases}
\end{equation}

The next result collects the basic information about the mixed
subdivision $S(\varphi)$ of $\Delta$ that we need for the study of the
Macaulay formula for the homogeneous resultant.

\begin{proposition}
  \label{prop:11}
  We have that $\varphi=\sum_{i=0}^{n}\max\{0,\ell_{i}\}$ and the
  $n$-cells of $S(\varphi)$ are the polytopes $C_{I}$ for
  $I\subset \{0,\dots, n\}$ with $I,I^{\cc}\ne \emptyset$.  Moreover,
  for each $I$ we have that
  \begin{enumerate}
  \item \label{item:17} the vertices of $C_{I}$ are the lattice points
    $v_{J,l}$ for $J\subset I^{\cc}$ and $l \in I$,
    \item \label{item:42} the components of $C_{I}$
  are the polytopes $C_{I,i}$, $i=0,\dots, n$,
  \item \label{item:18}
    $\sum_{i=0}^{n} \dim(C_{I,i})=\# I \cdot\#I^{\cc}$.
  \end{enumerate}
\end{proposition}

\begin{example}
  \label{exm:6}
  For $n=2$, the mixed subdivision $S(\varphi)$ has $6$ maximal cells
  that, with notation as in Figure~\ref{fig:1}, decompose as
\begin{align*}
  C_{\{0\}}& =  (0,0)+ A_{1}+C_{2}, & C_{\{0,1\}}& = (0,0)+(d_{1},0)+\Delta_{2},\\
  C_{\{1\}}& =A_{0}+(d_{1},0)+B_{2},  & C_{\{1,2\}}& = \Delta_{0}+(d_{1},0)+(0,d_{2}), \\
  C_{\{2\}}& =  C_0 + B_1 + (0,d_2), & C_{\{0,2\}}& = (0,0)+\Delta_{1}+(0,d_{2}).
\end{align*}

\vspace*{-3mm}
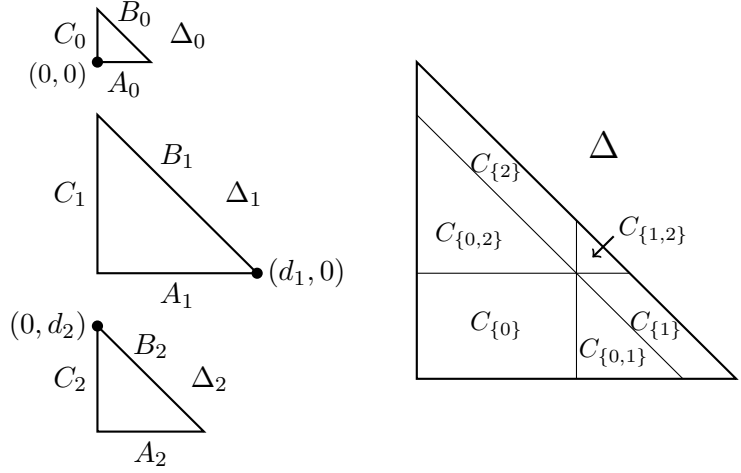
\begin{figure}[ht]
  \centering
\begin{tikzpicture}[scale=0.7]
\draw[thick, shift={(-6,6)}] (0,0)--(1,0)--(0,1)--cycle;
\draw[thick, shift={(-6,2)}] (0,0)--(3,0)--(0,3)--cycle;
\draw[thick, shift={(-6,-1)}] (0,0)--(2,0)--(0,2)--cycle;
\draw[thick] (0,0)--(6,0)--(0,6)--cycle;
\draw (0,5)--(5,0);
\draw (3,0)--(3,3);
\draw (0,2)--(4,2);
\draw (-5.5,6) node[below] {$A_0$};
\draw (-5.3,6.5) node[above] {$B_0$};
\draw (-6,6.5) node[left] {$C_0$};
\draw (-6,6.2) node[below left] {{\small $(0,0)$ }};
\draw (-4.3,6.5) node {$\Delta_0$};
\fill (-6,6) circle (3pt);
\draw (-4.5,2) node[below] {$A_1$};
\draw (-4.5,3.8) node[above] {$B_1$};
\draw (-6,3.5) node[left] {$C_1$};
\fill (-3,2) circle (3pt) node[right] {$(d_1,0)$};
\draw (-3.25,3.5) node {$\Delta_1$};
\draw (-5,-1) node[below] {$A_2$};
\draw (-5,0.2) node[above] {$B_2$};
\draw (-6,0) node[left] {$C_2$};
\fill (-6,1) circle (3pt) node[left] {$(0,d_2)$};
\draw (-3.9,0) node {$\Delta_2$};
\draw (1.5,1) node {{\small $C_{\{0\}}$}};
\draw (3.7,0.5) node {{\small $C_{\{0,1\}}$}};
\draw (4.5,1) node {{\small $C_{\{1\}}$}};
\draw (3.6,2.8) node[right] {{\small $C_{\{1,2\}}$}};
\draw[thick,->] (3.7, 2.7)-- (3.3, 2.3);
\draw (1.5,4) node {{\small $C_{\{2\}}$}};
\draw (1,2.7) node {{\small $C_{\{0,2\}}$}};
\draw (4,4) node[above left] {{\Large $\Delta$}};
\end{tikzpicture}
\vspace{-3mm}
\caption{A mixed subdivision in dimension 2}
  \label{fig:1}
\end{figure}
\end{example}

To prove Proposition \ref{prop:11}, we lift the previous constructions to $\R^{n+1}$.
For $i=0,\dots, n$ consider the simplex
$\wh \Delta_{i}= \{\bfz=(z_{0},\dots, z_{n})\in (\R_{\ge0})^{n+1}\mid |\bfz| =
d_{i}\}$ and the affine function
$\wh\varphi_{i}\colon \wh \Delta_{i}\to \R$ defined by
$ \wh\varphi_{i}(\bfz)=d_{i}-z_{i}$, and set
\begin{displaymath}
  \wh\Delta=\sum_{i=0}^{n} \wh\Delta_{i}=\{\bfz\in (\R_{\ge0})^{n+1}\mid
  |\bfz|=|\bfd|\} \and
  \wh\varphi=\bigboxplus_{i=0}^{n}\wh\varphi_{i} \colon
\wh\Delta\longrightarrow \R
\end{displaymath}
for the Minkowski sum of these simplexes and for the inf-convolution
of these affine functions, respectively.

For a subset $I\subset\{0,\dots, n\}$ consider the polytope of
$\wh \Delta$ defined as
    \begin{displaymath}
      \wh C_{I}= \{\bfz\in \wh\Delta \mid z_{i}\ge d_{i}
      \text{ for } i \in I \text{ and } z_{i}\le d_{i} \text{ for } i \in I^{\cc}\},
    \end{displaymath}
    and for each subset $J\subset I^{\cc}$ and each index $l\in I$
    consider the lattice point defined as
    \begin{equation}
    \label{eq:64}
    \wh   v_{J,l}=\Big( \sum_{j\in J} d_{j}\Big) \, \wh e_{l} + \sum_{j\in
        J^{\cc}}d_{j}\, \wh e_{j},
  \end{equation}
  where $\wh e_{i}$ denotes the $(i+1)$-th vector in the standard
  basis of $\R^{n+1}$.  For $i=0,\dots, n$ consider also the face of
  $\wh\Delta_{i}$ defined as
  \begin{equation}
\label{eq:89}
\wh   C_{I,i}=
  \begin{cases}
    d_{i}\, \wh e_{i} & \text{ if } i\in I, \\
    d_{i} \conv( \wh e_{i}, \{\wh e_{j}\}_{j\in I}) & \text{ if } i\in I^{\cc}.
  \end{cases}
\end{equation}
For each $i$ consider also the convex piecewise affine function
$\wh\eta_{i}\colon \wh \Delta\to \R$ defined by
$\wh\eta_{i}(\bfz)=\max\{0,z_{i}-d_{i}\}$ and set
$\wh \tau=\sum_{i=0}^{n}\wh\eta_{i}$.

\begin{lemma}
  \label{lemm:12}
  The $n$-cells of $S(\wh\tau)$ are the polytopes $\wh C_{I}$ for
  $I\subset \{0,\dots, n\}$ with $I,I^{\cc}\ne \emptyset$. For each
  $I$, the vertices of $\wh C_{I}$ are the lattice points
  $ \wh v_{J,l}$ for $J\subset I^{\cc}$ and $l\in I$.
\end{lemma}

\begin{proof}
  For each $i$, the subdivision $S(\wh \eta_{i})$ has two $n$-cells,
  one defined by $z_{i}\ge d_{i}$ and the other by $z_{i}\le d_{i}$.
  The $n$-cells of $S( \wh\tau)$ are intersections of $n$-cells of
  these subdivisions and so they are of the form $\wh C_{I}$ for
  $I\subset \{0,\dots, n\}$.

  We have that $\sum_{i=0}^{n} z_{i}-d_{i}=0$ on $\wh \Delta$, and so
  if either $I= \emptyset$ or $I^{\cc}=\emptyset$ then the polytope
  $\wh C_{I}$ reduces to the lattice point $\bfd$.  Otherwise, take
  $0<\varepsilon < 1$ and consider the point
  $\bfz=(z_{0},\dots, z_{n})$ defined by
  $z_{j}= d_{j}+\frac{\varepsilon}{\#I}$ if $ j\in I$ and as
  $ z_{j}=d_{j}-\frac{\varepsilon}{\#I^{\cc}}$ if $ j\in I^{\cc}$.  We
  have that $|\bfz|=|\bfd|$ and that $z_{j}>0$ for all $j$, and so
  $\bfz\in \ri(\wh\Delta)$. Also for each $i$ we have that $\bfz$ lies
  in the relative interior of the corresponding $n$-cell of
  $S(\wh\eta_{i})$. Hence $\wh C_{I}$ is $n$-dimensional in this case,
  proving the first claim.

  Now let $I\subset \{0,\dots, n\}$ with $I,I^{\cc}\ne \emptyset$.
  The vertices of $\wh C_{I}$ are the intersections in this polytope
  of $n$ of its supporting hyperplanes. These supporting hyperplanes are the zero set of one of the affine functions
  \begin{displaymath}
z_{j}, \ j\in I^{\cc}, \and z_{j}-d_{j}, \ j=0,\dots, n .
  \end{displaymath}
  To compute the vertices, take disjoint subsets $J\subset I^{\cc}$
  and $K\subset \{0,\dots, n\}$ with $\#J+\#K =n$. The intersection of
  the corresponding hyperplanes in the affine span of $\wh\Delta$ is
  \begin{displaymath}
    \{\bfz\in  \R^{n+1} \mid  |\bfz|=|\bfd|, \
    z_{j}=0 \text{ for } j\in J \text{ and } z_{j}=d_{j} \text{ for }  j\in K \}
  \end{displaymath}
  and it consists of the lattice point
  \begin{displaymath}
    \wh v_{J,l}   = \Big(|\bfd|-\sum_{j\in K} d_{j}\Big) \,
    \wh e_{l} + \sum_{j\in K} d_{j}\, \wh e_{j}
  \end{displaymath}
 for the unique index $l$ in the complement of
  $J \cup K$.

  When $J\ne\emptyset$ we have that $\wh v_{J,l} \in \wh C_{I}$ if and
  only if $l\in I$. When $J=\emptyset$ we have that
  $\wh v_{J,l}=\bfd$, which is also realized by taking $l\in I$ and
  $K=\{0, \dots, n\}\setminus \{l\}$, proving the second claim.
\end{proof}

Recall that for a convex piecewise affine function
$\rho\colon \Delta \to \R$ on a polyhedron $\Delta$ of $ \R^{n+1}$
and a vector $w\in \R^{n+1}$ we denote by $\Gamma(\rho,w)$ the
corresponding cell of the subdivision $S(\rho)$ of $\Delta $ as in
\eqref{eq:35}.

\begin{lemma}
  \label{lemm:2}
  Let $I\subset\{0,\dots, n\}$ with $I,I^{\cc}\ne \emptyset$ and set
  $w_{I}=-\sum_{k\in I}\wh e_{k}$. Then
  \begin{enumerate}
  \item \label{item:29} $\wh C_{I}=\Gamma(\wh\tau, w_{I})$ and for
     $\bfz\in \wh C_{I}$  we have that
    $\wh\tau(\bfz)=\sum_{j\in I}z_{j}-d_{j}$,
  \item \label{item:43} for $i=0,\dots, n$ we have that
    $\wh C_{I,i}=\Gamma(\wh \varphi_{i},w_{I})$ and for    $\bfz\in \wh C_{I,i}$ we have that
    $\wh\varphi_{i}(\bfz)= 0$ if $i\in I$ and
    $\wh\varphi_{i}(\bfz)= \sum_{j\in I} z_{j}$ if $i\in I^{\cc}$,
  \item \label{item:44} $\wh C_{I}=\sum_{i=0}^{n}\wh C_{I,i}$,
  \item \label{item:46} $\wh \varphi=\wh\tau$.
  \end{enumerate}
\end{lemma}

\begin{proof}
  For $\bfz\in \wh\Delta$ we have that
  $\wh \tau(\bfz)=\sum_{j=0}^{n}\max\{0,z_{j}-d_{j}\}$, which readily
  implies that 
  \begin{displaymath}
\wh \tau(\bfz)\ge \sum_{j\in I}z_{j}-d_{j}
  \end{displaymath}
  with equality if and only if $\bfz\in \wh C_{I}$. By the
  characterization of the cell $\Gamma(\wh\tau, w_{I})$ that follows
  from \eqref{eq:79}, this proves the statement in \eqref{item:29}.
  
  For each $i$, the vertices of $\wh \Delta_{i}$ are the lattice
  points $d_{i}\, \wh e_{j}$, $j=0,\dots, n$. For each $j$ we have
  that $ \langle w_{I}, d_{i}\, \wh e_{j}\rangle = -d_{i}$ if $j\in I$
  and $\langle w_{I}, d_{i}\, \wh e_{j}\rangle = 0$ if $j\in I^{\cc}$,
 whereas $ \wh\varphi_{i}(d_{i}\, \wh e_{j})=0$ if $j=i$ and
  $ \wh\varphi_{i}(d_{i}\, \wh e_{j})=d_{i}$ if $j\ne i$. Hence
  \begin{equation}\label{eq:3}
    \langle  w_{I}, d_{i}\, \wh e_{j}\rangle
    +\wh \varphi_{i}(d_{i}\,  \wh e_{j})=
\begin{cases}
  - d_{i} & \text{ if } j\in I \text{ and } j=i, \\
  0 & \text{ if } j\in I \text{ and } j \ne i,  \text{ or  } j\in I^{\cc}\text{ and } j=i,\\
  d_{i} & \text{ if } j\in I^{\cc} \text{ and } j\ne i.
\end{cases}
  \end{equation}
  Then the definition in \eqref{eq:89} and the characterization in
  \eqref{eq:79} easily imply that
  $\wh C_{I,i}=\Gamma(\wh \varphi_{i}, w_{I})$ which proves the first
  part of the statement in \eqref{item:43}.  Now let
  $\bfz\in \wh C_{I,i}$.  For $i\in I$ the polytope $\wh C_{I,i}$
  consists of the single lattice point $d_{i}\, \wh e_{i}$ and so
  $\wh \varphi_{i}(\bfz)=\wh \varphi_{i}(d_{i}\, \wh e_{i})= 0$. For
  $i\in I^{\cc}$ the vertices of $\wh C_{I,i}$ are the lattice points
  $d_{i}\, \wh e_{j}$, $j\in \{i\}\cup I$, and so
  $\bfz=\sum_{j\in \{i\}\cup I} z_{j}\, e_{j}$.  Hence~\eqref{eq:3}
  gives that
  $\wh \varphi_{i}(\bfz)= -\langle w_{I}, \bfz\rangle=\sum_{j\in I}
  z_{j}$, completing the proof of \eqref{item:43}.

Let $\wh v_{J,l}$ be the vertex of $\wh C_{I}$ associated to a
  subset $J\subset I^{\cc}$ and an index $l \in I$ as
  in~\eqref{eq:64}. We have that $d_{i}\, \wh e_{l}\in \wh C_{I,i}$ for
  all $i\in I^{\cc}$ and $d_{i}\, \wh e_{i}\in \wh C_{I,i}$ for all $i$,
and so
  $\wh v_{J,l} \in \sum_{i=0}^{n}\wh C_{I,i}$. Since this holds for
  all $J$ and $l$, Lemma \ref{lemm:12} implies that
\begin{equation}
  \label{eq:65}
  \wh C_{I}\subset \sum_{i=0}^{n}\wh C_{I,i}.
\end{equation}
From \eqref{item:43}  and Proposition \ref{prop:10}\eqref{item:2} we deduce
that the $\wh C_{I,i}$'s are the components of the cell
$\Gamma(\wh \varphi, w_{I})$ of $S(\wh\varphi)$ and by Proposition
\ref{prop:10}\eqref{item:3}, we have that  $\sum_{i=0}^{n}\wh C_{I,i}$ is a cell of
this mixed subdivision. On the other hand, the $\wh C_{I}$'s are
polytopes that cover $\wh\Delta$ and so the inclusion in \eqref{eq:65}
is an equality, as stated in \eqref{item:44}.

Now let $\bfz=(z_{0},\dots, z_{n})\in \wh C_{I}$ and
$\bfz_{i}\in \wh C_{I,i}$, $i=0,\dots,n$, such that
$\bfz=\sum_{i=0}^{n}\bfz_{i}$. Necessarily $\bfz_{i}=d_{i}\, \wh
e_{i}$ for $i\in I$ and so
\begin{displaymath}
  \wh\tau(\bfz)=\sum_{j\in I} z_{j}-d_{j}= \sum_{j\in I} \Big(\Big(
  \sum_{i=0}^{n}z_{i,j}\Big) -d_{j}\Big) = \sum_{\mathclap{i\in I^{\cc},j\in I}} z_{i,j}  = \sum_{i=0}^{n} \wh\varphi_{i}(\bfz_{i}) =  \wh\varphi(\bfz),
\end{displaymath}
where the first equality follows from \eqref{item:29} and the two last
from \eqref{item:43} and Proposition~\ref{prop:10}\eqref{item:5},
respectively.  Hence $\wh\tau$ and $\wh\varphi$ coincide on
$\wh C_{I}$ and so on the whole of $\wh\Delta$,
proving~\eqref{item:46}.
\end{proof}

\begin{proof}[Proof of Proposition \ref{prop:11}]
  Consider the projection $\pi\colon \R^{n+1}\to \R^{n}$ defined by
\begin{displaymath}
  \pi(z_{0},\dots, z_{n})=(z_{1},\dots, z_{n}).
\end{displaymath}
This linear map induces isomorphisms between $\wh\Delta$ and $\Delta$
and between $\wh \Delta_{i}$ and $\Delta_{i}$ for each $i$, and it
also satisfies that $ \wh \varphi_{i}=\varphi_{i}\circ\pi$ for each
$i$.

  For $\bfz\in \wh\Delta$ and $\bfx=\pi(\bfz)\in \Delta$, the
  condition that $ \bfz=\sum_{i=0}^{n}\bfz_{i} $ with
  $\bfz_{i}\in \wh \Delta_{i}$ translates into
  $\bfx=\sum_{i=0}^{n} \bfx_{i}$ with $\bfx_{i}\in \Delta_{i}$, and
  viceversa. Precisely, if the first condition holds then so does the
  second with $\bfx_{i}=\pi(\bfz_{i})$ and conversely, if the second
  condition holds then so does the first with $\bfz_{i}$ defined as
  the only preimage of $\bfx_{i}$ in $\wh\Delta_{i}$.

  We deduce from Lemma \ref{lemm:2}\eqref{item:46} that
  $\varphi\circ\pi = \wh \varphi =\wh\tau$, and the statement then
  follows from Lemmas \ref{lemm:12} and \ref{lemm:2}.
\end{proof}

\begin{remark}
  \label{rem:11}
  Proposition \ref{prop:11}\eqref{item:18} implies that the mixed
  subdivision $S(\varphi)$ is tight only when $n\le 2$.
\end{remark}

We next study a specific incremental chain of mixed subdivisions of
$\Delta$. For $k= 0,\dots, n$ consider the affine functions
$\theta_{k,i}\colon \Delta_{i}\to \R$, $i=0,\dots, n$, defined as
$\theta_{k,i}= \varphi_{i}$ if $i<k$ and as
$\theta_{k,i} =0|_{\Delta_{i}}$ if $i\ge k$, and set
$\theta_{k}=\bigboxplus_{i=0}^{n} \theta_{k,i}$ for their
inf-convolution.

For a subset $I\subset\{0,\dots, k-1\}$ let
$I^{\cc}=\{0,\dots, k-1\}\setminus I$
and consider the polytope of $\Delta$ defined as
\begin{equation*}
  C_{k,I}= \{\bfx\in \Delta \mid \ell_{i}(\bfx) \ge 0
\text{ for } i \in I \text{ and } \ell_{i}(\bfx) \le 0 \text{ for } i \in I^{\cc}\}.
\end{equation*}
For $J\subset I^{\cc}$ and $l\in I \cup\{k, \dots, n\} $ consider the
lattice point defined as
    \begin{equation*}
       v_{k,J,l}=\Big( \sum_{j\in J\cup\{k,\dots, n\}} d_{j}\Big) \, e_{l} + \sum_{j\in
        J^{\cc}}d_{j}\,  e_{j}
  \end{equation*}
  For $i=0,\dots, n$ consider also the face of $ \Delta_{i}$
  defined~as
\begin{equation*}
  C_{k,I,i}=
  \begin{cases}
    d_{i}\, e_{i} & \text{ if } i\in I, \\
    d_{i} \conv( e_{i}, \{e_{j}\}_{j\in I\cup \{k,\dots, n\}}) & \text{ if } i\in I^{\cc},\\
    d_{i} \conv( \{e_{j}\}_{j\in I\cup \{k,\dots, n\}}) & \text{ if }
    i\ge k.
  \end{cases}
\end{equation*}

The next result gives a detailed description of the mixed subdivision
$S(\theta_{k})$ of $\Delta$.

\begin{proposition}
  \label{prop:13}
  We have that
  $\theta_{k} =\sum_{i=0}^{k-1} \max\{0,\ell_{i}\} +
  \sum_{i=k}^{n}\ell_{i} $ and the $n$-cells of the mixed subdivision
  $S(\theta_{k})$ are the polytopes $C_{k,I}$ for
  $I\subset \{0,\dots, k-1\}$. Moreover, for each $I$
  \begin{enumerate}
  \item \label{item:49} the vertices of $C_{k,I}$ are the lattice
    points $ v_{k,J,l}$ for $J\subset I^{\cc}$ and
    $l\in I \cup\{k, \dots, n\} $,
    \item \label{item:55}  the
  components of $C_{k,I}$ are the polytopes $C_{k,I,i}$,
  $i=0,\dots, n$,
      \item \label{item:50}
    $\sum_{i=0}^{n} \dim(C_{k,I,i})=\# I^{\cc} \cdot
    (\#I+n-k+1)+(n-k+1) (\#I+n-k)$.
  \end{enumerate}
\end{proposition}

\begin{proof}
  Denote with a hat the corresponding objects in $\R^{n+1}$ as it was
  previously done for the study of $\varphi$, and consider
  also the convex
  piecewise affine function on $\Delta$ defined~as
\begin{displaymath}
  \tau_{k}=\sum_{i=0}^{k-1}\max\{0,\ell_{i}\} + \sum_{i=k}^{n}
  \ell_{i}
\end{displaymath}

The proof of these properties is a direct generalization of that for
Proposition \ref{prop:11} and so we only indicate the main steps:
  \begin{itemize}
  \item \label{item:23} show that the $n$-cells of $S(\wh \tau_{k})$
    are the polytopes $\wh C_{k,I}$ for $I\subset \{0,\dots, k-1\}$,
  \item \label{item:25} for each $I$, compute the vertices of $\wh C_{k,I}$ by
    considering the intersections of the supporting hyperplanes of this
    polytope,
  \item \label{item:28} compute the face of $\wh \Delta_{i}$
defined by the slope of $\wh \tau_{k}$ on $\wh C_{k,I}$,
  \item \label{item:30} show that the Minkowski sum of these faces
    coincides with  $\wh C_{k,I}$,
  \item \label{item:35} show that $\wh \theta_{k}$ coincides with
    $\wh \tau_{k}$ on each $\wh C_{k,I}$, and so on the whole of $\Delta$.
    \end{itemize}
    Finally, the obtained results are  brought back  to $\R^{n}$
    via the projection $\pi$, as in the proof of Proposition
    \ref{prop:11}.
\end{proof}

\begin{proposition}
  \label{prop:15}
  We have that $S(\theta_{0})\preceq \cdots\preceq S(\theta_{n})$ is
  an incremental chain of mixed subdivisions of $\Delta$ with
  $S(\theta_{n}) \preceq S(\varphi)$ that satisfies the conditions in
  Proposition \ref{prop:16}. In particular, this incremental chain is
  admissible for any tight mixed subdivision of $\Delta$ that
  refines~$S(\varphi)$.
\end{proposition}

\begin{proof}
  For each $I\subset\{0,\dots, n\} $ with $I,I^{\cc}\ne \emptyset$ we
  have that
  \begin{displaymath}
    C_{I}\subset C_{n,I\cap \{0,\dots, n-1\}} \and
  C_{I,i}\subset C_{n,I\cap \{0,\dots, n-1\},i}, \ i=0,\dots, n.
  \end{displaymath}
  Propositions \ref{prop:11} and \ref{prop:13} then imply that
  $S(\varphi)\succeq S(\theta_{n})$.  Similarly, for
  $k\in \{1,\dots, n\}$ and each $I\subset \{0,\dots, k-1\}$ we have
  that
  \begin{displaymath}
    C_{k,I}\subset C_{k-1,I\cap \{0,\dots, k-2\}} \and
    C_{k,I,i}\subset C_{k-1,I\cap \{0,\dots, k-2\},i} , \ i=0,\dots,
    n,
  \end{displaymath}
  and Proposition \ref{prop:13} implies that
  $S(\theta_{k})\succeq S(\theta_{k-1})$. Hence
  $ S(\theta_{0})\preceq \cdots\preceq S(\theta_{n}) \preceq
  S(\varphi)$.  Since $\theta_{k,i}=0|_{\Delta_{i}}$ for all $k$ and
  $i\ge k$, the chain
  \begin{equation*}
    S(\theta_{0})\preceq \cdots\preceq S(\theta_{n})
  \end{equation*}
  is incremental.

  For $k=0,\dots, n$, let $I\subset \{0,\dots, k-1\}$. If
  $I\ne \emptyset$ then for the $n$-cell $C_{I}$ of $S(\theta_{k})$,
  each component $C_{I,i}$ with $i\in I$ consists of a lattice point
  and so this component verifies the condition \eqref{item:51} in
  Proposition \ref{prop:16}. Else $I=\emptyset$ and so
  $C_{k,I,i}=d_{i} \conv( e_{i}, \{e_{j}\}_{j\in \{k,\dots, n\}})$ if
  $ i<k$ and
  $ C_{k,I,i}= d_{i} \conv( \{e_{j}\}_{j\in \{k,\dots, n\}})$ if
  $ i\ge k$.  For $i=0,\dots, k-1$ consider the nonzero vector
  $w_{i}\in \R^{n}$ defined as $w_{i}=\sum_{j=1}^{n}e_{j}$ if $i=0$
  and as $w_{i}=e_{i}$ if $i>0$. For all $j\ne i$ we have that
  $C_{k,I,j}$ lies in a hyperplane that is parallel to $w_{i}^{\bot}$
  and so
\begin{displaymath}
  \dim\Big( \sum_{j\ne i} C_{k,I,j}\Big) <n.
\end{displaymath}
Hence $C_{k,I}$ satisfies the condition \eqref{item:53} in Proposition
\ref{prop:16}. The last statement is a direct consequence of that
proposition.
\end{proof}

Figure \ref{fig:2} shows this admissible incremental chain of mixed
subdivisions for a case in dimension $n=2$.

\begin{figure}[ht]
\begin{tikzpicture}[scale=0.5]
\draw[thick] (0,0)--(6,0)--(0,6)--cycle;
\draw (3,-0.1) node[below] {$S(\theta_0)$};
\end{tikzpicture} \quad
\begin{tikzpicture}[scale=0.5]
\draw[thick] (0,0)--(6,0)--(0,6)--cycle;
\draw (3,-0.1) node[below] {$S(\theta_1)$};
\draw (0,5)--(5,0);
\end{tikzpicture} \quad
\begin{tikzpicture}[scale=0.5]
\draw[thick] (0,0)--(6,0)--(0,6)--cycle;
\draw (0,5)--(5,0);
\draw (3,0)--(3,3);
\draw (3,-0.1) node[below] {$S(\theta_2)$};
\end{tikzpicture} 
\vspace{-3mm}
\caption{An admissible incremental chain}
  \label{fig:2}
\end{figure}
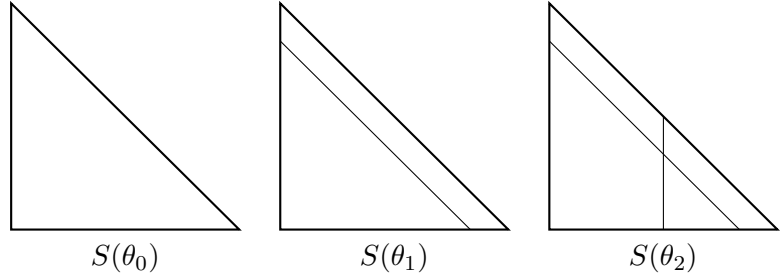

\subsection{Polyhedral interpretation} \label{sec:polyh-interpr}

Let $\bfd=(d_{0},\dots, d_{n})\in (\N_{>0})^{n+1}$. In this section we
study the classical Macaulay formula \eqref{eq:74} for the situation
when
\begin{displaymath}
  m =|\bfd|-n,
\end{displaymath}
which is the main case of interest.  We keep the notation of
\S\ref{sec:class-maca-form}. In particular, for $i=0,\dots, n$ we
denote by $\bfu_{i}$ a set of ${d_{i}+n\choose n}$ variables indexed
by the lattice points $\bfc \in \N^{n+1}$ of length
$|\bfc|=d_{i}$. Also $\bfu=(\bfu_{0}, \dots, \bfu_{n})$ and
$\K=\C(\bfu)$.

In this situation, we respectively denote the corresponding index set
and nonreduced index subset, linear map, Macaulay matrix and
distinguished principal submatrix by
\begin{displaymath}
\cI^{\circ} \subset \cI\subset \Z^{n+1}, \quad    \Psi_{\bfd}\colon
      \bigoplus_{i=0}^{n}T_{i} \to T , \quad  \cM_{\bfd} \in
      \K^{\cI \times \cI} \and \cN_{\bfd} \in
      \K^{\cI^{\circ} \times \cI^{\circ}},
\end{displaymath}
where $T_{i}$, $i=0,\dots, n$, and $T$ are the finite dimensional
linear subspaces of the polynomial ring
$\K[\bft]=\K[t_{0},\dots, t_{n}]$ in \eqref{eq:43}.

As explained in Example \ref{exm:3}, the homogeneous resultant
$\Res_{\bfd}$ coincides, up to the sign, with the sparse resultant
corresponding to the lattice $M=\Z^{n}$ and the family of supports
$\bfcA=(\cA_{0},\dots, \cA_{n})$ defined by
$\cA_i=\{\bfa\in\N^n \mid |\bfa| \leq d_i\}$ for each $i$.

We also use the notation of \S\ref{sec:mixed-regul-subd}. In
particular, for $i=0,\dots, n$ we consider the simplex
$ \Delta_{i} = \conv(\cA_{i})=\{\bfx\in (\R_{ \ge 0})^{n}\mid |\bfx|
\le d_{i}\} $ and the affine function
$\varphi_{i}\colon \Delta_{i}\to \R$ defined by
$ \varphi_{i}(\bfx)= |\bfx| $ if $i=0$ and as
$\varphi(\bfx) = d_{i}-x_{i}$ if $i>0$.  Let also
$ \Delta=\sum_{i=0}^{n} \Delta_{i}$ be the Minkowski sum of these
polytopes and $\varphi=\bigboxplus_{i=0}^{n}\varphi_{i}$ the
inf-convolution of these affine functions. By Proposition
\ref{prop:11} we have that
\begin{displaymath}
  \Delta=\{\bfx\in (\R_{\ge0})^{n}\mid |\bfx|\le
|\bfd|\} \and
\varphi=\sum_{i=0}^{n}\max\{0, \ell_{i}\},
\end{displaymath}
where $\ell_{i}\colon \Delta\to \R$ is the affine function defined by
$\ell_{i}(\bfx)= \sum_{j=1}^{n}d_{j} - |\bfx| $ when $i=0$ and by
$\ell_{i}(\bfx)=x_{i}-d_{i}$ when $i>0$.

For $i=0,\dots, n$ choose a linear function
$\mu_{i}\colon \R^{n}\to \R$ and set $
\rho_{i}=\varphi_{i}+\mu_{i}$. Set then
$\bfrho=(\rho_{0},\dots, \rho_{n})$ and
$\rho=\bigboxplus_{i=0}^{n}\rho_{i}$, and suppose that the mixed
subdivision $ S(\rho)$ of $\Delta$ is tight and refines
$S(\varphi)$.  By Propositions~\ref{prop:1} and \ref{prop:9}, both
conditions are attained when the $\mu_{i}$'s are sufficiently generic
and small.  Choose also
$\bfdelta = (\delta_{1},\dots, \delta_{n}) \in \R^{n}$ with
$\delta_{i}+1 > 0$ for all $i$ and $\sum_{i=1}^{n}(\delta_{i}+1) <1$,
and satisfying the genericity condition~\eqref{eq:20} with respect to
$S(\rho)$.

Consider then the index set and nonmixed index subset, Sylvester map,
Canny-Emiris matrix and distinguished principal submatrix
corresponding to $\bfcA$, $\bfrho$ and $\bfdelta$, respectively
denoted by
\begin{displaymath}
\cB^{\circ}\subset   \cB \subset \Z^{n}, \quad
\Phi_{\bfcA,\bfrho}\colon \bigoplus_{i=0}^{n} V_{i} \to V,  \quad
  \cH_{\bfcA,\bfrho} \in \K^{\cB\times \cB} \and \cE_{\bfcA,\bfrho} \in \K^{\cB^{\circ}\times \cB^{\circ}}
\end{displaymath}
where $V_{i}$, $i=0,\dots, n$, and $V$ are the finite dimensional
linear subspaces of
$\K[\Z^{n}]= \K[\bfs^{\pm1}]=\K[s_{1}^{\pm1},\dots, s_{n}^{\pm1}]$
defined in \eqref{eq:5}.

The next proposition shows that this Canny-Emiris matrix coincides
with that of Macaulay, and that this is also the case for their
distinguished principal submatrices.

\begin{proposition}
  \label{prop:2}
  The morphism of algebras
  $ \pi_{*}  \colon \K[\bft]\to \K[\bfs]$ defined by
  $\pi_{*}(t_{0})= 1$ and  $\pi_{*}(t_{i})= s_{i}$,  $i=1,\dots, n$,
  induces a commutative diagram
  \begin{displaymath}
    \xymatrixcolsep{5pc}
    \xymatrix{
      \bigoplus_{i=0}^{n}T_{i} \ar[d]^{\pi_{*}}\ar[r]^-{\Psi_{\bfd}}& T\ar[d]^{\pi_{*}}\\
      \bigoplus_{i=0}^{n}V_{i} \ar[r]^-{\Phi_{\bfcA,\bfrho}}& V
    }
\end{displaymath}
and bijections between the monomial bases of
$ \bigoplus_{i=0}^{n}T_{i}$ and $ \bigoplus_{i=0}^{n}V_{i}$, and
between those of $T$ and $V$. In particular
$\cH_{\bfcA,\bfrho}=\cM_{\bfd}$ and $\cE_{\bfcA,\bfrho}=\cN_{\bfd}$.
\end{proposition}

To prove it, we first need to  establish  some auxiliary lemmas.

\begin{lemma}
  \label{lemm:16}
  Let $C$ be an $n$-cell of $S(\rho)$. Let $I\subset \{0,\dots, n\}$
  with $I, I^{\cc}\ne\emptyset $ such that $C\subset C_{I}$ and
  $i\in \{0, \dots, n\}$. Then
  \begin{enumerate}
  \item \label{item:57} the $i$-th component $C_{i}$ is a point if and
    only if $i\in I$, and if this is the case then
    $C_{i}=\{d_{i}\, e_{i}\}$,
  \item \label{item:58} $C$ is $i$-mixed if and only if $I=\{i\}$.
  \end{enumerate}
\end{lemma}

\begin{proof}
  Since $S(\rho)\succeq S(\varphi)$ we have that
  $C_{i}\subset C_{I,i}$. If $i\in I$ then \eqref{eq:25} implies that
  $C_{i}=\{d_{i}\, e_{i}\}$.  Conversely, for $i \in I^{\cc}$ consider
  the vector $w_{i}\in \R^{n}$ defined as
  $w_{i}=\sum_{j=1}^{n}e_{j}$ if $i=0$ and as $w_{i}=e_{i}$ if
  $i>0$. For all $j \ne i$ we have that $C_{I,j}$ lies in a hyperplane
  that is parallel to $w_{i}^{\bot}$, and so does $C_{j}$. Hence
  \begin{displaymath}
    \dim(C_{i}) = n-
  \dim\Big( \sum_{j\ne i} C_{j}\Big) >0,
\end{displaymath}
proving \eqref{item:57}.  The statement in \eqref{item:58} is a direct
consequence of that in~\eqref{item:57}: $C$ is $i$-mixed if and only
if $C_{I,i}$ is the unique component of $C_{I}$ of dimension 0, which
is equivalent to the fact that $I=\{i\}$.
\end{proof}

\begin{lemma}
  \label{lemm:14}
$ \cB=\{\bfb \in \N^{n} \mid |\bfb| \le
|\bfd|-n\}$.
\end{lemma}

\begin{proof}
  Let $\bfb=(b_{1},\dots, b_{n}) \in \Z^{n}$. Then
  $\bfb \in \Delta+\bfdelta$ if and only if $b_{i}-\delta_{i}\ge 0$
  for all $i$ and $|\bfb - \bfdelta|\le |\bfd|$. Since $\bfb$ is a
  lattice point, $\delta_{i}+1 > 0$ for all $i$ and
  $\sum_{i=1}^{n}(\delta_{i}+1) <1$, these conditions are equivalent
  to $b_{i}\ge 0$ for all $i$ and $|\bfb|\le |\bfd|-n$, proving the
  lemma.
\end{proof}

\begin{lemma}
  \label{lemm:19}
  Let $\bfb=(b_{1},\dots, b_{n})\in \cB$ and set
  $b_{0}=|\bfd|-n-|\bfb|$. Then
  \begin{enumerate}
  \item \label{item:27} for $I\subset \{0,\dots, n\}$ with
    $I, I^{\cc}\ne \emptyset$ we have that $\bfb\in C_{I}+\bfdelta$ if
    and only if $b_{i}\ge d_{i}$ for all $i\in I$ and $b_{i}< d_{i}$
    for all $i\in I^{\cc}$,
  \item \label{item:59} for each $i$ we have that $\bfb\in \cB_{i}$ if
    and only if if $b_{i}\ge d_{i}$ and $b_{j}< d_{j}$ for $j>i$, and
    if this is the case then $a(b)=d_{i}\, e_{i}$,
\item \label{item:60} $\bfb\in \cB\setminus \cB^{\circ}$ if and only
  if there is a unique $i$ such that $b_{i}\ge d_{i}$.
  \end{enumerate}
\end{lemma}

\begin{proof}
  With notation as in \eqref{eq:61} we have that
  $\bfb\in C_{I}+\bfdelta$ if and only if
  $\ell_{i}(\bfb-\bfdelta) \ge 0$ for $i\in I$ and
  $\ell_{i}(\bfb-\bfdelta) \le 0$ for $i\in I^{\cc}$. Arguing as in
  the proof of Lemma \ref{lemm:14}, we deduce that these conditions
  are equivalent to $b_{i}\ge d_{i}$ for all $i\in I$ and
  $b_{i}< d_{i}$ for all $i\in I^{\cc}$, as stated in \eqref{item:27}.

  The statements in \eqref{item:59} and \eqref{item:60} follow from
  that in \eqref{item:27} together with Lemma~\ref{lemm:16}.
\end{proof}

\begin{proof}[Proof of Proposition \ref{prop:2}]
  Let $\pi\colon \Z^{n+1}\to \Z^n$ be the linear map given by
  $\pi(c_{0},\dots, c_{n})= (c_{1},\dots, c_{n})$, so that
  $\pi_{*}(\bft^{\bfc})=\bfs^{\pi(\bfc)}=s_{1}^{c_{1}}\dots
  s_{n}^{c_{n}}$ for all $ \bfc=(c_{0},\dots, c_{n})\in \N^{n+1}$.

  By Lemma \ref{lemm:14},  $\pi$ induces a bijection
  between the index sets $\cI $ and $\cB$, and so the morphism of
  algebras $\pi_{*}$ gives a bijection between the monomial bases of $T$
  and of $V$. Similarly, by Lemma \ref{lemm:19}\eqref{item:59} $\pi$
  also induces a bijection between $\cI_{i}$ and $\cB_{i}$ for each
  $i$, and so $\pi_{*}$ gives a bijection between the monomial bases of
  $ \bigoplus_{i=0}^{n}T_{i}$ and $ \bigoplus_{i=0}^{n}V_{i}$, proving
  the second claim.  Moreover, for $\bfc\in \cI$ we have that
  \begin{displaymath}
    \pi_{*}( \Psi_{\bfd}(\bft^{\bfc}))=\pi_{*}( \bft^{\bfc- d_{i}\, \wh
      e_{i}}) = \bfs^{\pi(\bfc)-d_{i}\, e_{i}} = \Phi_{\bfcA,
      \bfrho}(\bfs^{\pi(\bfc)})= \Phi_{\bfcA,
      \bfrho}(\pi_{*}(\bft^{\bfc})),
  \end{displaymath}
  which shows the commutativity of the diagram. The last claim is a
  direct consequence of the two previous.
\end{proof}

\begin{corollary}
  \label{cor:4}
   $\Res_{\bfd}= {\det(\cM_{\bfd})}/{\det(\cN_{\bfd})}$.
\end{corollary}

\begin{proof}
  By Proposition \ref{prop:15}, the mixed subdivision $S(\rho)$ is
  admissible. Theorem \ref{thm:9} and Proposition~\ref{prop:2} then
  imply that
  \begin{displaymath}
\Res_{\bfd}=\pm \frac{\det(\cH_{\bfcA, \bfrho})}{\det(\cE_{\bfcA, \bfrho})}=\pm \frac{\det(\cM_{\bfd})}{\det(\cN_{\bfd})}.
\end{displaymath}
The sign can be determined by considering the evaluation of both sides
of this equality at the coefficients of systems of polynomials
$t_{i}^{d_{i}}$, $i=0,\dots, n$.
\end{proof}

\begin{example}
  \label{exm:5}
  Consider again the case when $n=2$, $\bfd=(1,2,2)$ and $m=3$.
  Then
\begin{displaymath}
  \cA_{0}=\{(0,0),(1,0),(0,1)\}  \and
  \cA_{1}=\cA_{2}=\{(0,0),(1,0),(2,0), (1,0), (1,1), (0,2)\}.
\end{displaymath}
By Proposition \ref{prop:11}\eqref{item:18}, the mixed subdivision
$S(\varphi)$ is tight and so we can take $\rho_{i}=\varphi_{i}$,
$i=0, 1,2$. We also choose
$\bfdelta =(-\frac{2}{3}, -\frac{3}{4}) \in \R^{2}$.

As shown in Figure \ref{fig:5}, the index set $\cB$ splits as
$\cB=\cB_{0}\sqcup\cB_{1}\sqcup\cB_{2}$ with
\begin{align*}
  \cB_{0}&=\{(0,0),(1,0),(0,1),(1,1)\}, \\
  \cB_{1}&=\{(2,0),(3,0),(2,1)\}, \\
  \cB_{2} &=\{(0,2),(1,2),(0,3)\},
\end{align*}
and the row content function assigns to the elements of $\cB_{i}$,
$i=0, 1, 2$, the vertices $(0,0)\in \Delta_{0}$,
$(2,0)\in \Delta_{1}$, $(0,2)\in \Delta_{2}$, respectively. Moreover,
the elements of $\cB$ lying in the translated non-mixed $2$-cells are
$(2,0)$ and $(0,2)$.

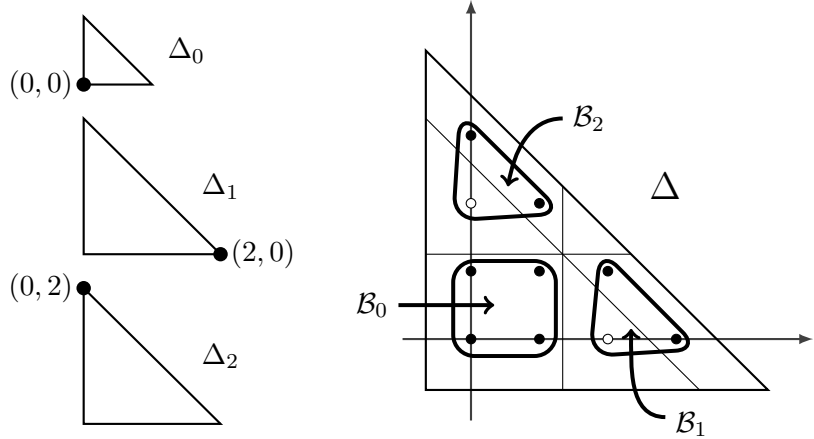
\begin{figure}[ht]
  \begin{tikzpicture}[scale=0.9]
\draw[shift={(0.66, 0.75)},-latex,color=darkgray,thick] (-1,0) -- (5,0);
\draw[shift={(0.66, 0.75)},-latex,color=darkgray,thick] (0,-1.2) -- (0,5);
\draw[thick, shift={(-5,4.5)}] (0,0)--(1,0)--(0,1)--cycle;
\draw[thick, shift={(-5,2)}] (0,0)--(2,0)--(0,2)--cycle;
\draw[thick, shift={(-5,-0.5)}] (0,0)--(2,0)--(0,2)--cycle;
\draw[thick] (0,0)--(5,0)--(0,5)--cycle;
\fill (-5,4.5) circle (3pt) node[left] {$(0,0)$};
\fill (-3,2) circle (3pt) node[right] {$(2,0)$};
\fill (-5,1.5) circle (3pt) node[left] {$(0,2)$};
  \draw (2,0)--(2,3);
  \draw (0,2)--(3,2);
  \draw (0,4)--(4,0);
  \draw (-3.5,5) node {$\Delta_0$};
  \draw (-3,3) node {$\Delta_1$};
  \draw (-3,0.5) node {$\Delta_2$};
  \draw (3.5,3) node {{\Large $\Delta$}};
  \foreach \x in {0,1} {
            \foreach \y in {0,1} {
            \filldraw[color=black](\x+0.66,\y+0.75) circle (2pt);
            }}
  \foreach \x in {0,1,2,3} {
            \filldraw[color=black](\x+0.66,3.75-\x) circle (2pt);
            }
  \draw[fill=white] (2.66,0.75) circle (2pt);
  \draw[fill=white] (0.66,2.75) circle (2pt);
  \draw[line width=1.5pt, rounded corners=8pt] (0.4,0.5)--(1.9, 0.5)--(1.9,1.9)--(0.4,1.9)--cycle;
  \draw[line width=1.5pt, rounded corners=9pt] (2.4,0.5)--(4, 0.6)--(2.55,2.07)--cycle;
  \draw[line width=1.5pt, rounded corners=9pt] (0.4,2.5)--(2, 2.6)--(0.5,4.1)--cycle;
  \draw (-0.4,1.25) node[left] {$\cB_{0}$};
  \draw[line width=1.5pt, ->] (-0.4,1.25) to (1,1.25);
  \draw (2,4) node[right] {$\cB_{2}$};
\draw[line width=1.5pt, ->] (2,4) to [out=180, in=80] (1.2,3);
\draw (3.5,-0.5) node[right] {$\cB_{1}$};
\draw[line width=1.5pt, ->] (3.5,-0.4) to [out=180, in=-90] (3,0.9);
\end{tikzpicture}
\vspace{-3mm}
\caption{The index set of a mixed subdivision}  \label{fig:5}
\end{figure}

The Canny-Emiris matrix   $\cH_{\bfcA,\bfrho}$ and its principal submatrix
$\cE_{\bfcA,\bfrho}$ respectively coincide with the Macaulay matrices
$\cM_{\bfd}$ and $\cN_{\bfd}$ in Example \ref{exm:4}, in agreement
with Proposition~\ref{prop:2}.
\end{example}

\begin{remark}
  \label{rem:10}
  The Macaulay matrix $\cM_{\bfd}$ cannot be produced by a mixed
  subdivision admitting a tight incremental chain
  (Definition~\ref{def:9}).  This can be shown by inspecting the
  different mixed subdivisions of $\Delta$ that allow such an
  incremental chain and verifying that none  of them  coincides with
  $S(\rho)$. One of these incremental chains for the case when $n=2$
  and $\bfd=(1,3,2)$ is shown in Example \ref{exm:8}.
\end{remark}

We next show that  when the mixed subdivision $S(\rho)$ is not
admissible, the formula in Theorem \ref{thm:9} might fail to hold.

\begin{example}
  \label{exm:7}
   Let notation be as in Example \ref{exm:5} and instead of the
  $\varphi_{i}$'s, consider  the affine functions
  $\rho_{i}\colon \Delta_{i}\to \R$, $i=0,1,2$, defined by
\begin{align*}
  &\rho_{0}(0,0)=0, \ \rho_{0}(1,0)=1, \ \rho_{0}(0,1)=1, \\
  &\rho_{1}(0,0)=0, \ \rho_{1}(2,0)=0, \ \rho_{1}(0,2)=3, \\
  &\rho_{2}(0,0)=0, \ \rho_{2}(2,0)=3, \ \rho_{2}(0,2)=0.
\end{align*}
Let $\rho\colon \Delta\to \R$ be their inf-convolution. The mixed
subdivision $S(\rho)$ of $\Delta$ is tight and has 6 maximal
cells as shown in Figure \ref{fig:6}, that decompose as
\begin{align*}
  C_{1 }&= (0,0)+A_{1}+B_{2} , &      C_{2}&=  A_{0}+(2,0)+B_{2}, &    C_{3}&= (1,0)+ (2,0)+\Delta_{2},      \\
  C_{4}&= B_{0}+A_{1}+(0,2),      &  C_{5}&= \Delta_{0}+(2,0)+(0,2),      &  C_{6}&= (0,1)+\Delta_{1}+(0,2).
\end{align*}

\begin{figure}[ht]
  \begin{tikzpicture}[scale=0.9]
\draw[shift={(0.66, 0.75)},-latex,color=darkgray,thick] (-1,0) -- (5,0);
\draw[shift={(0.66, 0.75)},-latex,color=darkgray,thick] (0,-1.2) -- (0,5);
\draw[thick, shift={(-5,5)}] (0,0)--(1,0)--(0,1)--cycle;
  \draw[thick, shift={(-5,2.25)}] (0,0)--(2,0)--(0,2)--cycle;
  \draw[thick, shift={(-5,-0.5)}] (0,0)--(2,0)--(0,2)--cycle;
  \draw[thick] (0,0)--(5,0)--(0,5)--cycle;
  \draw (2,0)--(2,3);
  \draw (0,2)--(3,2);
  \draw (3,0)--(3,2);
  \draw (0,3)--(2,3);
\fill (-5,5) circle (3pt) node[left] {$(0,0)$};
\fill (-3,2.25) circle (3pt) node[right] {$(2,0)$};
\fill (-5,1.5) circle (3pt) node[left] {$(0,2)$};
  \draw (-4.5,5) node[below] {$A_0$};
  \draw (-5,5.5) node[left] {$B_0$};
  \draw (-3.5,5) node {$\Delta_0$};
  \draw (-4,2.25) node[below] {$A_1$};
  \draw (-5,3.25) node[left] {$B_1$};
  \draw (-3,3.25) node {$\Delta_1$};
  \draw (-4,-0.5) node[below] {$A_2$};
  \draw (-5,0.5) node[left] {$B_2$};
  \draw (-3,0.5) node {$\Delta_2$};
  \foreach \x in {0,1,2} {
            \foreach \y in {0,1} {
            \filldraw[color=black](\x+0.66,\y+0.75) circle (2pt);
            }}
  \foreach \x in {0,1} {
            \filldraw[color=black](\x+0.66,2.75) circle (2pt);
            }
  \draw[fill=white] (3.66,0.75) circle (2pt);
  \draw[fill=white] (0.66,3.75) circle (2pt);
\draw[line width=1.5pt, ->] (1,-0.5) to (1,0.4);
\draw (1,-0.5) node[below] {$C_1$};
\draw[line width=1.5pt, ->] (2.5,-0.5) to (2.5,0.4);
\draw (2.5,-0.5) node[below] {$C_2$};
\draw[line width=1.5pt, ->] (4.25,-0.7) to [out=180, in=-90] (3.75,0.4);
\draw (4.25,-0.7) node[right] {$C_3$};
\draw[line width=1.5pt, ->] (-0.4,2.5) to (1,2.5);
\draw (-0.4,2.5) node[left] {$C_4$};
\draw[line width=1.5pt, ->] (3,3) to [out=180, in=90] (2.5,2.2);
\draw (3,3) node[right] {$C_5$};
\draw[line width=1.5pt, ->] (1.5,4) to [out=180, in=75] (1,3.5);
\draw (1.5,4) node[right] {$C_6$};
\end{tikzpicture}
  \vspace{-3mm}
   \caption{A non-admissible mixed subdivision} \label{fig:6}
\end{figure}
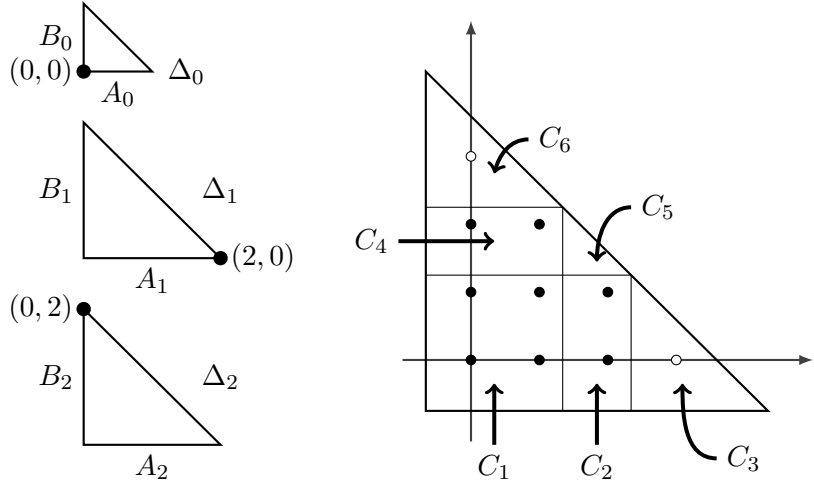

The index set and row function corresponding to the data
$\bfcA=(\cA_{0}, \cA_{1}, \cA_{2})$,
$\bfrho=(\rho_{0},\rho_{1},\rho_{2})$ and $\bfdelta$ are equal to
those in Example \ref{exm:5}, and so the Canny-Emiris matrix
$\cH_{\bfcA,\bfrho}$ also coincides with that in \eqref{eq:57}.
However, the translated non-mixed cells of $S(\rho)$ differ from those
in Example \ref{exm:5}. Their lattice points are $(3,0)$ and $(0,3)$
and the corresponding principal submatrix is
\begin{displaymath}
  \cE_{\bfcA,\bfrho}=
  \begin{pmatrix}
    \beta_{2}& 0\\
    0& \gamma_{5}
  \end{pmatrix}.
\end{displaymath}
The determinant of this matrix does not divide that of
$\cH_{\bfcA,\bfrho}$ and so the formula in Theorem \ref{thm:9} does
not hold in this case. In particular, $S(\rho)$ is not admissible.

 Indeed, this latter observation can be verified directly: let
\begin{equation*}
  S(\theta_{0}) \preceq S(\theta_{1}) \preceq S(\theta_{2})
\end{equation*}
be an incremental chain of mixed subdivisions of $\Delta$ with
$S(\theta_{2})\preceq S(\rho)$ and for each $k=0,1,2$ let
$\theta_{k,i}\colon \Delta_{i}\to \R$, $i=0,1,2$, be the corresponding
family of convex piecewise affine functions.

If $\theta_{1,0}\colon \Delta_{0}\to \R$ is not constant, then
$S(\theta_{1})$ has a cell that is a translate of the triangle
$ \Delta_{1}+\Delta_{2}= \{(x_{1},x_{2}) \in (\R_{\ge 0})^{2} \mid
x_{1}+x_{2} \le 4\}$ which is not compatible with the assumption that
$S(\rho)$ is a refinement of $S(\theta_{1})$, as it can be verified on
Figure~\ref{fig:5}.  We deduce that
$\theta_{1,0}\colon \Delta_{0}\to \R$ is constant, but in this case
$S(\theta_{1})$ is the trivial mixed subdivision of $\Delta$ and this
incremental chain does not verify the conditions in
Definition~\ref{def:5} for $k=1$, and so it is not admissible.
\end{example}



\newcommand{\noopsort}[1]{} \newcommand{\printfirst}[2]{#1}
  \newcommand{\singleletter}[1]{#1} \newcommand{\switchargs}[2]{#2#1}
  \def\cprime{$'$}
\providecommand{\bysame}{\leavevmode\hbox to3em{\hrulefill}\thinspace}
\providecommand{\MR}{\relax\ifhmode\unskip\space\fi MR }
\providecommand{\MRhref}[2]{%
  \href{http://www.ams.org/mathscinet-getitem?mr=#1}{#2}
}
\providecommand{\href}[2]{#2}

\end{document}